\input ./style/arxiv-general.cfg
\documentclass[aop,MSNbibl,seceqn,dvips]{arximspdf}
\makeatletter
   \@ifpackageloaded{graphicx}{}{\usepackage{graphicx}}
\makeatother
\usepackage{mathbh}

\doi{10.1214/14-AOP985}
\volume{44}
\issue{1}
\pubyear{2016}
\firstpage{628}
\lastpage{683}
\docsubty{FLA}

\makeatletter
\newtheorem{theorem}{Theorem}[section]
\newtheorem{lemma}[theorem]{Lemma}
\newtheorem{lemmaa}{Lemma}[section]
\newtheorem{corollary}[theorem]{Corollary}
\newtheorem{proposition}[theorem]{Proposition}
\newproclaim{definition}[theorem]{Definition}
\newtheorem{property}{Property}
\newproclaim{remark}[theorem]{Remark}

\def\b{\beta}
\def\G{\Gamma}
\def\cE{\mathcal{E}}
\def\D{\Delta}
\def\k{\kappa}
\def\L{\Lambda}
\def\n{\nu}
\def\s{\sigma}
\def\t{\tau}
\def\P{\Psi}
\def\O{\Omega}

\def\ve{\varepsilon}

\def\Z{{\mathbb Z}}
\def\R{{\mathbb R}}
\def\C{{\mathbb C}}
\def\H{{\mathbb H}}
\def\P{{\mathbb P}}
\def\E{{\mathbb E}}

\def\rL{{\mathrm{L}}}

\def\Ra{\Rightarrow}
\def\Uda{\Updownarrow}
\def\Lra{\Leftrightarrow}
\def\La{\Leftarrow}

\def\iy{\infty}
\def\es{\varnothing}
\def\ss{\subset}
\def\pa{\partial}
\def\iint{\int\!\! \int}
\def\1{\mathbh{1}}

\def\Area{\mathop{\mathrm{Area}} }
\def\arg{\mathop{\operatorname{arg}} }
\def\const{\mathop{\operatorname{const}} }
\def\diam{\mathop{\operatorname{diam}} }
\def\dist{\mathop{\operatorname{dist}} }
\def\Im{\mathop{\operatorname{Im}} }
\def\Int{\mathop{\operatorname{Int}} }
\def\Re{\mathop{\operatorname{Re}} }
\def\Length{\mathop{\mathrm{Length}} }

\newcommand{\textfrac}[2]{{ \frac{#1}{#2}}}

\def\mesh{\delta}
\renewcommand{\hm}[3]{\omega_#1(#2;#3)}

\def\RWsign{\mathrm{RW}}
\newcommand{\RW}[3]{\RWsign_{#1}(#2;#3)}
\def\ZRWsign{\mathrm Z}
\newcommand{\ZRW}[3]{\ZRWsign_{#1}(#2;#3)}
\def\TTsign{\mathrm T}

\newcommand{\ZRWT}[4]{\ZRWsign_#1[\TTsign_#2](#3;#4)}
\def\RRsign{\mathrm R}
\newcommand{\RR}[4]{\RRsign_#1(#2;#3,#4)}
\def\xYsign{\mathrm Y}
\newcommand{\xY}[5]{\xYsign_#1(#2,#3;#4,#5)}
\def\xXsign{\mathrm X}
\newcommand{\xX}[5]{\xXsign_#1(#2,#3;#4,#5)}
\def\ddsign{\mathrm d}
\newcommand{\dd}[2]{\ddsign_#1(#2)}
\def\BBsign{\mathrm B}
\newcommand{\BB}[2]{\BBsign_#1(#2)}
\def\Asign{\mathrm A}
\newcommand{\A}[2]{\Asign_{#1}(#2)}
\def\Csign{\mathrm C}
\def\ELsign{\mathrm L}
\newcommand{\EL}[3]{\ELsign_{#1}(#2;#3)}

\def\ccdot{ \cdot}
\makeatother

\begin{document}
\begin{frontmatter}

\title{Robust discrete complex analysis: A toolbox\thanksref{T1}}
\runtitle{Robust discrete complex analysis: A toolbox}

\begin{aug}
\author[A]{\fnms{Dmitry}~\snm{Chelkak}\corref{}\ead[label=e1]{dchelkak@pdmi.ras.ru}}
\runauthor{D. Chelkak}
\affiliation{St.~Petersburg Department of Steklov Institute (PDMI RAS)
and Chebyshev Laboratory at St.~Petersburg State University}
\address[A]{Institute for Theoretical Studies\\
ETH Z\"urich\\
Clausiusstrasse 47\\
Building CLV, 8092 Z\"urich\\
Switzerland\\
\printead{e1}}
\end{aug}
\thankstext{T1}{Supported in part by Pierre Deligne's 2004 Balzan
prize in Mathematics (associated research scholarship of the author in
2009--2011), and the Chebyshev Laboratory at St. Petersburg State
University under the Russian Federation Government Grant 11.G34.31.0026
and JSC ``Gazprom Neft.''}

%
\received{\smonth{1} \syear{2013}}
%
\revised{\smonth{8} \syear{2014}}

%
\begin{abstract}
We prove a number of double-sided estimates relating discrete
counterparts of several classical conformal invariants of a
quadrilateral: cross-ratios, extremal lengths and random walk partition
functions. The results hold true for any simply connected discrete
domain $\O$ with four marked boundary vertices and are uniform with
respect to $\O$'s which can be very rough, having many fiords and
bottlenecks of various widths. Moreover, due to results from
[Boundaries of planar graphs, via circle packings (2013)
Preprint], those estimates are fulfilled for domains drawn on any
infinite ``properly embedded'' planar graph $\Gamma\subset \mathbb C$ (e.g., any
parabolic circle packing) whose vertices have bounded degrees. This
allows one to use classical methods of geometric complex analysis for
discrete domains ``staying on the microscopic level.'' Applications
include a discrete version of the classical Ahlfors--Beurling--Carleman
estimate and some ``surgery technique'' developed for discrete quadrilaterals.
\end{abstract}

%
\begin{keyword}[class=AMS]
\kwd[Primary ]{39A12}
\kwd[; secondary ]{60G50}
\end{keyword}

\begin{keyword}
\kwd{Planar random walk}
\kwd{discrete potential theory}
\end{keyword}
\end{frontmatter}

\section{Introduction}
\label{Sect:Intro}

\subsection{Motivation}

This paper was originally motivated by the recent activity devoted to
the analysis of interfaces arising in the critical 2D lattice models on
regular grids (e.g., see \cite{Smi06,Smi10} and references therein),
particularly the random cluster representation of the Ising model \cite
{KS12,CDCH13,CDCHKS13}. The other contexts where techniques developed
in this paper could be applied are the analysis of random planar graphs
and their limits \cite{BS01,GR13,GN13} or lattice models where some
connection to discrete harmonic measure can be established (or is
already plugged into the model, e.g., as in DLA-type processes).
However, note that below we essentially use the ``uniformly bounded
degrees'' assumption, especially when proving a duality estimate for
(edge) extremal lengths. In particular, all results of this paper hold
true for discrete domains which are subsets of any given parabolic
circle packing with uniformly bounded degrees; see \cite{HS95}.
Nevertheless, some important setups (notably, circle packings of random
planar maps) are not covered, requiring some additional input
(possibly, a kind of a ``surgery'' near high degree vertices; cf. \cite
{GN13}). At the same time, the paper has an independent interest, being
devoted to one of the central objects of discrete potential theory on a
(weighted) graph $\G$ embedded into a complex plane: partition
functions of the random walk running in a discrete simply connected
domain $\O\subset\G$.

Dealing with some 2D lattice model and its scaling limit (an
archetypical example is the Brownian motion in $\O$, which can be
realized, e.g., as the limit of simple random walks on refining square
grids $\mesh\Z^2$), one usually works in the context when the lattice
mesh $\mesh$ tends to zero. Then it can be argued that a discrete
lattice model is sufficiently close to the continuous one, if $\mesh$
is small enough: for example, random walks hitting probabilities (discrete
harmonic measures) converge to those of the Brownian motion (continuous
harmonic measure; cf. \cite{Kak44}) as $\mesh\to0$. After rescaling
the underlying grid by $\mesh^{-1}$, statements of that sort provide an
information about properties of the random walk running in \emph{large}
discrete domains $\O\ss\Z^2$.

Unfortunately, this setup is not sufficient when we are interested in
fine geometric properties of 2D lattice models (e.g., full collection
of interfaces in the random cluster representation of the critical
Ising model): sometimes it turns out that one needs to consider not
only macroscopic $\O$'s but also their subdomains ``\emph{on all
scales}'' (like $\delta^\ve$ or even several lattice steps) simultaneously
in order to gain some macroscopic information. Questions of that kind
are still tractable by classical means if those microscopic parts of
$\O
$ are regular enough (e.g., rectangular-type subsets of $\Z^2$;
cf. \cite{DCHN11,KS12}). Nevertheless, if no such regularity
assumptions can be made due to some monotonicity features of the
particular lattice model, the situation immediately becomes much more
complicated; cf. \cite{CDCH13,CDCHKS13}.

Having in mind the classical geometric complex analysis as a guideline,
in this paper we construct its discrete version ``staying on the
microscopic level'' (i.e., without any passage to the scaling limit or
any coupling arguments) which allows one to handle discrete domains by
more-or-less the same methods as continuous ones. Namely, we prove a
number of \emph{uniform estimates} (a ``toolbox'') which hold true for
any simply connected $\O$, possibly having many {fiords} and
{bottlenecks} of various widths, including very thin (several lattice
steps) ones.

Being interested in estimates rather than convergence, we do not need
any nice ``complex structure'' on the underlying weighted planar graph.
Instead, we assume that the (locally finite) embedding $\G\subset\C$
satisfies the following mild assumptions: neighboring edges have
comparable lengths and angles between them are bounded away from $0$
and do not exceed $\pi- \eta_0$ for some constant $\eta_0>0$; see
Section~\ref{SubSect:GraphAssumptions}. In the very recent paper \cite
{ABGGN} it is shown that these assumptions imply two crucial properties
of the corresponding random walk on $\G$: \ref{PropertyS} the
probability of the event that the random walk started at the center of
a \emph{Euclidean disc} exits this disc through a given boundary arc of
angle $\pi- \eta_0$ is uniformly bounded from below and \ref
{PropertyT} the expected time spent by the random walk in this disc is
uniformly comparable to its area; see Section~\ref
{SubSect:RandomWalkAssumptions} for details. For general properly
embedded graphs $\G$, we base all the considerations on these estimates
from \cite{ABGGN}, using them as a starting point for the analysis of
random walks in \emph{rough domains}. On the other hand, our results
seem to be new even if \mbox{$\G=\Z^2$}, so the reader not interested
in full generality may always think about this, probably the simplest
possible case in which \ref{PropertyS} and \ref{PropertyT} can be
easily derived from standard properties of the simple random walk on
the square grid.

In order to shorten the presentation, below we widely use the following
notation: assuming that all ``structural parameters'' of a planar graph
$\G$ listed in Section~\ref{Sect:NotationAndPreliminaries} are fixed
once forever (or if we work with some concrete $\G$):
\begin{itemize}
\item by ``$\const$'' we denote \emph{positive}
constants (like $\frac{1}{2\pi}$ or $7^{812}$) which do not depend on
geometric properties (the shape of $\O$, positions of boundary points,
etc.) of the configuration under consideration or additional parameters
we deal with (thus ``$f\le\const$'' means that there
exists a positive constant $C$ such that the inequality $f\le C$ holds
true \emph{uniformly} over all possible configurations);
\item we write ``$ f\asymp g$'' if there exist two
positive constants $C_{1,2}$ such that one has $C_1f\le g\le C_2f$
uniformly over all possible configurations (in other words, $f$ and $g$
are comparable up to some uniform constants which we do not specify);
\item we write, for example, ``\emph{if} $f\ge\const$, then
$g_1\asymp g_2$'' if and only if, for \emph{any} given constant
$c>0$, the estimate $f\ge c$ implies $C_1g_1\le g_2\le C_2g_1$, where
$C_{1,2}=C_{1,2}(c)>0$ may depend on $c$ but are independent of all
other parameters involved.
\end{itemize}

\subsection{Main results}
The main objects of interest are (discrete) quadrilaterals, that is,
simply connected domains $\O$ with four marked boundary points
$a,b,c,d$ listed counterclockwise. Focusing on quadrilaterals, we have
two motivations. First, in the classical theory this is the ``minimal''
configuration which has a nontrivial conformal invariant (e.g., all
simply connected $\O$'s with three marked boundary points are
conformally equivalent due to the Riemann mapping theorem). Second,
this is an archetypical configuration in the 2D lattice models
theory, where one often needs to estimate probabilities of
crossing-type events in $(\O;a,b,c,d)$.

Note that even if $\G=\Z^2$, there is a crucial difference between
discrete and continuous theories. The latter is essentially based on
conformal mappings and conformal invariance of various quantities,
notably the conformal invariance of extremal lengths; see \cite{Ahl73},
Chapter~4 and \cite{GM05}, Chapter IV. Using conformal
invariance, one typically may rewrite the question originally
formulated in $\O$ as the same question for some canonical domain (unit
circle, half-plane, rectangle, etc.), thus simplifying the problem
drastically; for example, see \cite{GM05}, Theorem IV.5.2. In
particular, up to conformal equivalence, $(\O;a,b,c,d)$ can be
described by a single real parameter (modulus). Therefore, all
conformal invariants of those $\O$'s ({cross-ratios, extremal lengths,
partition functions of the Brownian motion}) are just some concrete
functions of each other.

This picture changes completely when coming down to the discrete level:
for discrete domains (subsets of a \emph{fixed} graph $\G$) we do not
have any reasonable notion of conformal equivalence. Nevertheless, for
a discrete quadrilateral, one can easily introduce natural analogues of
all classical conformal invariants listed above. Namely, let $\ZRWsign_\O=
\ZRW\O{[ab]_\O}{[cd]_\O}$ denote the total \emph
{partition function of random walks} running from the boundary arc
$[ab]_\O\subset\O$ to another arc $[cd]_\O\subset\pa\O$ inside~$\O$.
In the particular case of the simple random walk on $\G=\Z^2$, this means
\[
\ZRWsign_{\O}\bigl([ab]_\O;[cd]_\O\bigr) =\sum
_{\gamma\in S_\O([ab]_\O;[cd]_\O
)}\frac
{1}{4^{\#\gamma}},
\]
where $S_\O([ab]_\O;[cd]_\O)$ denotes the set of all nearest-neighbor
paths connecting $[ab]_\O$ and $[cd]_\O$ inside $\O$, and $\#\gamma
$ is
the length (number of steps) of $\gamma$; see Section~\ref
{SubSect:PartitionFctRW} for further details. Then, we define the \emph
{discrete cross-ratio} $\xYsign_\O=\xY{\O}abcd$ of
boundary points
$a,b,c,d\in\pa\O$ as
\[
\xYsign_\O:=\biggl[\frac{
\ZRWsign_{\O}(a;d)
\ZRWsign_{\O}(b;c)
}{
\ZRWsign_{\O}(a;b)
\ZRWsign_{\O}(c;d)
} \biggr]^{{1}/{2}},
\]
where, for example, $\ZRW{\O}ad$ denotes the similar partition function
of random walks running from $a$ to $d$ in $\O$; see Section~\ref
{Sect:XYandZ} for further details. We also use the classical definition
of \emph{discrete extremal length} (or, equivalently, effective
resistance of the corresponding electrical network) $\ELsign
_\O=\EL
{\O}{[ab]_\O}{[cd]_\O}$ between $[ab]_\O$ and $[cd]_\O$ which goes back
to Duffin \cite{Duf62}; see Section~\ref{Sect:ExtremalLength} for details.

Certainly, one cannot hope that $\ZRWsign_\O,\xYsign_\O$ and
$\ELsign_\O
$ are related by the same \emph{identities} as in the classical theory.
Nevertheless, one may wonder if those can be replaced by some \emph
{double-sided estimates} which do not depend on geometric properties of
$(\O;a,b,c,d)$. One of the main results of our paper, Theorem~\ref
{Thm:DblSidedEstimates}, gives the positive answer to this question.
Namely, it says that, provided $\ELsign_\O\ge\const$, one has
\[
\ZRWsign_\O\asymp\xYsign_\O\quad\mbox{and}\quad\log\bigl(1+
\xYsign_\O ^{-1}\bigr)\asymp\ELsign_\O,
\]
\emph{uniformly} over all possible discrete quadrilaterals. Note that
we use discrete cross-ratio $\xYsign_\O$ as an intermediary that allows
us to relate ``analytic'' partition function $\ZRWsign_\O$ and
``geometric'' extremal length $\ELsign_\O$ in a way which is very
similar to the classical setup.

In order to illustrate a potential of the toolbox developed in our
paper, we include two applications of a different kind. The first,
given in Section~\ref{Sect:Surgery}, is a ``surgery technique'' for
discrete quadrilaterals which is important for the fine analysis of
interfaces in the critical Ising model; see \cite{CDCH13}. Namely, we
show that it is always possible to cut $\O$ along some family of slits
$L_k$ into two parts $\O'_k$ and $\O''_k$ (containing $[ab]_\O$ and
$[cd]_\O$, resp.) so that, for any $k$, one has
\begin{eqnarray*}
\ZRWsign_\O&\asymp& \ZRWsign_{\O'_k}\bigl([ab]_\O;
\rL_k\bigr) \ZRWsign_{\O''_k}\bigl(\rL_k;[cd]_\O
\bigr) \quad\mbox{and}\\
 \ZRWsign_{\O'_k}\bigl([ab]_\O;
\rL_k\bigr) &\asymp& k \ZRWsign_{\O''_k}\bigl(
\rL_k;[cd]_\O\bigr) ;
\end{eqnarray*}
see Theorem~\ref{Thm:Separator} for details. Using discrete
cross-ratios techniques, we prove this result, which is quite natural
from a geometric point of view, without any reference to the actual
geometry of $\O$. As always in our paper, double-sided estimates given
above are {uniform} with respect to $(\O;a,b,c,d)$ and $k$.

Another application, given in Section~\ref{Sect:Estimates}, allows one
to control the discrete harmonic measure $\omega_\mathrm{disc}:=\hm\O{u}{[ab]_\O}$ of a ``far'' boundary arc
\mbox{$[ab]_\O\ss\pa\O$}
via an
appropriate discrete extremal length $\ELsign_\mathrm{disc}$ in $\O$;
see Section~\ref{Sect:Estimates} and Theorem~\ref{Thm:LogHm=ELdiscr}
for details. This should be considered as an analogue of the famous
Ahlfors--Beurling--Carleman estimate; see \cite{GM05}, Theorem IV.5.2,
and \cite{GM05}, page 150, for historical notes. Again, we get a uniform
double-sided bound which, as a byproduct, implies that
\[
\log\bigl(1+\omega_\mathrm{disc}^{-1}\bigr) \asymp
\ELsign_\mathrm{disc} \asymp \ELsign_\mathrm{cont} \asymp\log
\bigl(1+\omega_\mathrm{cont}^{-1}\bigr)
\]
\emph{uniformly} over all possible configurations $(\O;u,a,b)$, where
$\omega_\mathrm{cont}$ denotes the classical harmonic measure of the
boundary arc $[ab]$ seen from $u$ in the polygonal representation of
$\O
$; see Corollary~\ref{Cor:LogHmDisc=LogHmCont} for details. Note that
results of this sort seem to be hardly available by any kind of
coupling arguments. Indeed, dealing with thin fiords we are mostly
focused on exponentially rare events for both discrete random walks and
the (continuous) Brownian motion, which are highly sensitive to widths
of those fiords.

\subsection{Organization of the paper} \label{SubSect:Organization} In
Section~\ref{Sect:NotationAndPreliminaries} we formulate
assumptions (a)--(d) on the embedding $\G\subset\C$ (Section~\ref
{SubSect:GraphAssumptions}), fix the notation for discrete domains $\O$
(Section~\ref{SubSect:DiscreteDomains}), introduce the {partition
functions} $\ZRWsign_\O$ of the simple random walk in $\O$ and discuss
its relation to the standard notions of discrete harmonic measure and
discrete Green function (Section~\ref{SubSect:PartitionFctRW}).
Further, in Section~\ref{SubSect:RandomWalkAssumptions} we formulate two
crucial properties \ref{PropertyS} and \ref{PropertyT} of the random
walk on $\G$ (namely, uniform estimates for hitting probabilities and
expected exit times for discrete approximations of {Euclidean discs}).
We also list several basic facts of the discrete potential theory
(elliptic Harnack inequality, weak Beurling-type estimates, some
uniform estimates for Green functions) in Section~\ref{SubSect:BasicFacts}.

Section~\ref{Sect:Factorization} is devoted to a uniform (up to
multiplicative constants) \emph{factorization of the three-point
partition function} $\ZRW\O a {[bc]_\O}$ via two-point functions
$\ZRW
\O ab$, $\ZRW\O ac$ and $\ZRW\O bc$. Namely, we prove that (see
{Theorem~\ref{Thm:Zfact}})
\[
\ZRWsign_{\O}\bigl(a;[bc]_\O\bigr)
\asymp\bigl[ {\ZRW\O ab \ZRW\O ac}/{\ZRW\O bc}
\bigr]^{{1}/{2}}
\]
uniformly over all configurations $(\O;a,b,c)$. This is the
{cornerstone} of our paper and the only one place where we involve some
geometric considerations in the proofs.

In Section~\ref{Sect:XYandZ}, we introduce \emph{discrete
cross-ratios} $\xXsign_\O$, $\xYsign_\O$ for a simply connected domain
$\O$ with four marked boundary points $a,b,c,d$ (see Definition~\ref
{Def:CrossRatio}) and deduce from Theorem~\ref{Thm:Zfact} several
double-sided estimates relating $\xXsign_\O$, $\xYsign_\O$ and
$\ZRWsign
_\O$. In particular, we prove that $\xXsign_\O^{-1}\asymp1+\xYsign
_\O
^{-1}$ (see Proposition~\ref{Prop:XwtXestim}), which is an analogue of
the well-known identity for classical cross-ratios, and $\ZRWsign_\O
\asymp\log(1+\xYsign_\O)$ (see Theorem~\ref{Thm:ZlogYestim}),
which is
a precursor of the exponential-type estimate relating $\ZRWsign_\O$ and
$\ELsign_\O$.

Section~\ref{Sect:Surgery} is independent of the rest of the
paper. It shows how one can use Theorem~\ref{Thm:Zfact} and discrete
cross-ratios introduced in Section~\ref{Sect:XYandZ} in order to build
a sort of ``\emph{surgery technique}'' which allows one to effectively
``decouple'' dependence $\ZRWsign_\O$ of the boundary arcs $[ab]_\O$
and $[cd]_\O$ by finding nice discrete cross-cuts in $\O$.

In Section~\ref{Sect:ExtremalLength}, the notion of \emph
{discrete extremal length} $\EL\O{[ab]}{[cd]}$ comes into play. We
recall its definition and prove that $\ELsign_\O$ is always uniformly
comparable to its continuous counterpart, extremal length of the family
of curves connecting $[ab]$ and $[cd]$ in the polygonal representation
of $\O$. In particular, this fact implies the very important \emph
{duality estimate} for discrete extremal lengths; see Corollary~\ref
{Cor:ELduality}. We also prove some simple inequalities relating
$\ZRWsign_\O$ and $\ELsign_\O^{-1}$; see Proposition~\ref{Prop:ZleL-1}.

Section~\ref{Sect:Estimates} summarizes all the estimates for
$\xYsign_\O$, $\ZRWsign_\O$ and $\ELsign_\O$ obtained before into
single Theorem~\ref{Thm:DblSidedEstimates} which is the culmination of
our paper. Then we show how to fit a \emph{discrete harmonic measure}
$\hm\O{u}{[ab]_\O}$ into this context (as $\O\setminus\{u\}$ is not
simply connected, a reduction similar to \cite{GM05}, page 144, is
needed). The result [double-sided estimate of $\hm\O{u}{[ab]_\O}$ via
an appropriate extremal length] is given by Theorem~\ref
{Thm:LogHm=ELdiscr}. As a simple byproduct, we prove Corollary~\ref
{Cor:LogHmDisc=LogHmCont}, which says that the logarithm of a discrete
harmonic measure is uniformly comparable to its continuous counterpart.

In order to make the whole presentation self-contained, in the
\hyperref[app]{Appendix} we derive all the basic facts of the discrete
potential
theory listed in Section~\ref{SubSect:BasicFacts} from properties \ref
{PropertyS} and \ref{PropertyT} of the underlying random walk. In some
sense, our paper uses these properties, formulated for the simplest
possible discrete domains (approximations of Euclidean discs), as
``black box assumptions'' that turn out to be enough to develop uniform
estimates relating $\ZRWsign_\O$, $\xYsign_\O$ and $\ELsign_\O$
for all
simply connected $\O$'s; see also Remark~\ref{Rem:ad=>ST}.

\section{Notation, assumptions and preliminaries}
\label{Sect:NotationAndPreliminaries}

\subsection{Graph notation and assumptions}
\label{SubSect:GraphAssumptions} Throughout this paper we work with an
{infinite} undirected weighted planar graph $(\G;\mathrm{E}^\G)$
embedded into
a complex plane $\C$ so that all of its edges are straight segments
(see Figure~\ref{Fig: grid-irreg}), which is assumed to satisfy
assumptions (a)--(d) given below. The notation $\Gamma\ss\C$ is fixed
for the set of vertices which are understood as points in the complex
plane (so $|u-v|$ means the Euclidean distance between $u,v\in\G$), and
$\mathrm{E}^\Gamma$ denotes the corresponding set of edges. Each edge
$e\in\mathrm{E}
^\G$ is equipped with a positive \emph{weight} $\mathrm{w}_e$. Note
that, in general, these weights are \emph{not} related to the way how
$\G$ is embedded into $\C$.
We assume that $\G$ satisfies:

\begin{figure}

\includegraphics{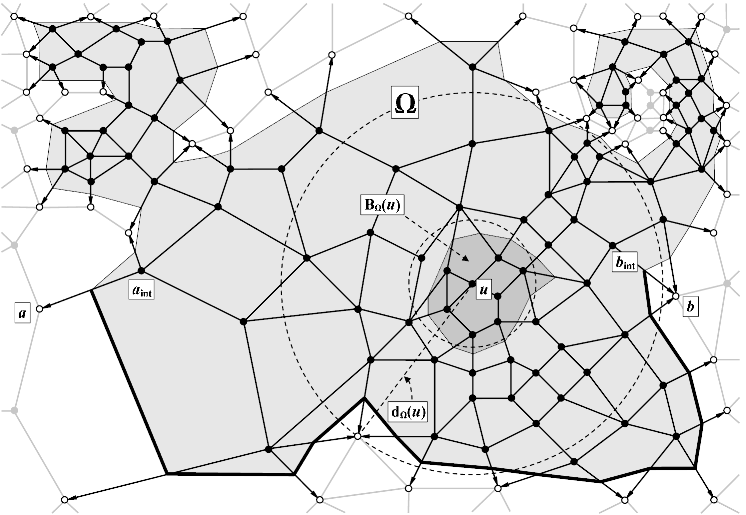}

\caption{An example of a graph $\G$ and a simply connected discrete
domain $\O\ss
\G$. The inner vertices of $\O$ are colored black, the boundary ones
are white. For two boundary edges $(aa_\mathrm{int})$ and $(bb_\mathrm
{int})$, the corresponding counterclockwise boundary arc $[ab]_\O$ is
marked. For an inner vertex $u\in\Int\G$, the distance $\dd\O
{u}=\dist
(u;\O)$ from $u$ to $\pa\O$ and the discrete disc $\BB\O
{u}=\BBsign^\G
_{r}(u)$ of radius $r=\frac{1}{3}\dd\O{u}$ are shown.} \label{Fig:
grid-irreg}
\end{figure}

\begin{longlist}[(a)]
\item[(a)]\emph{uniformly bounded degrees}: there exists a
constant $\varpi_0>0$ such that \mbox{$\mathrm{w}_e\ge\varpi_0$} for
all edges $e\in\mathrm{E}^\G$ and $\mu_v:=\sum_{(vv')\in\mathrm
{E}^\G}\mathrm
{w}_{vv'}\le\varpi_0^{-1}$ for all vertices $v\in\G$.
\end{longlist}
Clearly, this is equivalent to saying that all edge weights $\mathrm
{w}_\mathrm{e}$ are uniformly bounded away from $0$ and $\infty$, and
all degrees of vertices of $\Gamma$ are uniformly bounded as well. We
then denote \emph{random walk transition probabilities} by
%
\begin{equation}
\label{RWtransprobaDef} {\varpi}_{vv'}:=\frac{\mathrm{w}_{vv'}}{\mu_v}=\frac{\mathrm
{w}_{vv'}}{\sum_{(vv')\in\mathrm{E}^\G}\mathrm{w}_{vv'}}.
\end{equation}
Note that 
the probabilities $\varpi_{vv'}$ are uniformly bounded below by
$\varpi_0^2>0$.
We now describe the way that $\Gamma$ is \emph{embedded} into $\C$. We
assume that:
\begin{longlist}[(b)]
\item[(b)] \emph{there are no flat angles}: there exists a
constant $\eta_0>0$ such that, for each vertex $v\in\G$, all angles
between neighboring edges of $\G$ incident to $v$ do not exceed $\pi
-\eta_0$;

\item[(c)] \emph{edge lengths are locally comparable}: there
exists a constant $\varkappa_0\ge1$ such that, for each vertex $v\in
\G
$, one has
%
\begin{equation}
\label{RvDef} \max_{(vv')\in\mathrm{E}^\G}\bigl|v'-v\bigr|\le
\varkappa_0r_v \qquad\mbox{where } r_v:=\min
_{(vv')\in\mathrm{E}^\G}\bigl|v'-v\bigr|
\end{equation}
(below we sometimes call $r_v$ the \emph{local scale size});

\item[(d)] $\G$ \emph{is locally finite} (i.e., it does not
have accumulation points in $\C$).
\end{longlist}
It is worth noting that (b) and (c) also imply that all degrees of
{faces} of $\G$ are uniformly bounded, and all angles between
neighboring edges are uniformly bounded away from $0$. In particular,
the radius of isolation $\min_{v'\in\G}|v' - v|$ of a vertex $v\in
\G$
is always uniformly comparable to $r_v$. Let us emphasize that we do
\emph{not} assume that $r_v$'s are comparable to each other: the local
scale sizes \emph{can} significantly vary from place to place; see
Figure~\ref{Fig: grid-irreg}. Also, we do not assume any quantitative
bound in condition (d).

\begin{remark}
\label{Rem:LasymptDist}
It is easy to see that for some constant \mbox{$\nu_0=\nu_0(\eta
_0,\varkappa_0)\ge1$} and all $u,v\in\G$, there exists a
nearest-neighbor path $\rL_{uv}=(u_0u_1\cdots u_n)$, $(u_su_{s+1})\in
\mathrm{E}^\G
$, between $u=u_0$ and $v=u_n$ such that
%
\begin{equation}
\label{LuvBound} \Length(\rL_{uv})=\sum_{s=0}^{n-1}|u_{s+1}
- u_s|\le\nu _0|v - u|.
\end{equation}
\end{remark}

In particular, one can use the following construction (see \cite{ABGGN}
for details). Let $[u;v]\ss\mathbb{C}$ denote a straight segment
between $u$ and $v$ in the plane, $f_1,\ldots,f_m$ be consecutive
faces of
$\G$ that are intersected by $[u;v]$ and let
$[z_{s-1};z_{s}]:=[u;v]\cap f_{s}$. It follows from (b) and (c) that
one can replace each of the subsegments $[z_{s-1};z_{s}]$ by a path
$\ell_{s}$ running along the boundary of $f_{s}$ so that the length of
$\ell_s$ is bounded by $\nu_0|z_{s} - z_{s-1}|$. Concatenating these
$\ell_s$ and erasing repetitions, if necessarily, one gets a proper
path $\rL_{uv}$. It might happen that the result is not the shortest
path between $u$ and $v$ in $\G$. Nevertheless, it has an important
feature which will be used below:
%
\begin{equation}
\mbox{all vertices of $\rL _{uv}$ belong to faces crossed by the
segment $[u;v]$}.
\end{equation}
In particular, \emph{this} $\rL_{uv}$ does not cross the straight line
passing through $u$ and $v$ outside of $[u;v]$ (note that the shortest
path joining $u$ and $v$ along edges of $\G$ could do so).

\begin{remark}
\label{Rem:RvBound} Let $u\ne v$ be two vertices of $\Gamma$. It
immediately follows from~(\ref{LuvBound}) that $r_v\le\nu_0|v - u|$.
Moreover, for all edges $(vv')\in\mathrm{E}^\Gamma$, one has $|v' -
v|\le
\varkappa_0 r_v \le\varkappa_0\nu_0|v - u|$. In particular, it
cannot happen that $|v' - u|> (\varkappa_0\nu_0 + 1)\cdot|v - u|$.
\end{remark}

\subsection{Bounded discrete domains and discrete discs} \label
{SubSect:DiscreteDomains}
We start with a definition of a (bounded) {discrete domain} $\O$; see
Figure~\ref{Fig: grid-irreg}.
Let $(V^\O;\mathrm{E}^\O_\mathrm{int})$ be a bounded \emph
{connected} subgraph
of $(\G;\mathrm{E}^\G)$. In order to make the presentation simpler
and not to
overload the notation, we always assume that $(vv')\in\mathrm{E}^\O
_\mathrm
{int}$ for any two neighboring (in $\G$) vertices $v,v'\in V^\O$ (one
can easily remove this assumption, if necessary). Denote by $E^\O
_\mathrm{bd}$ the set of all \emph{oriented} edges $(a_\mathrm
{int}a)\notin E^\O_\mathrm{int}$ such that $a_\mathrm{int}\in V^\O$
(and $a\notin V^\O$).
We set $\O:=\Int\O\cup\pa\O$, where
\[
\Int\O:=V^\O, \qquad\pa\O:=\bigl\{\bigl(a;(a_{\mathrm{int}}a)\bigr)
\dvtx(a_{\mathrm
{int}}a)\in E^\O_{\mathrm{bd}}\bigr\}.
\]
Formally, the boundary $\pa\O$ of a discrete domain $\O$ should be
treated as the set of oriented edges $(a_{\mathrm{int}}a)$, but we
usually identify it with the set of corresponding vertices $a$, and
think about $\Int\O$ and $\pa\O$ as subsets of $\G$, if no
confusion arises.

We say that a discrete domain $\O$ is \emph{simply connected} if, for
any cycle in $E^\O_\mathrm{int}$, all edges of $\G$ surrounded by this
cycle also belong to $E^\O_\mathrm{int}$. If $\O$ is simply connected,
then its boundary vertices (or, more precisely, boundary edges) are
naturally cyclically ordered, exactly as in the continuous setting. For
two boundary vertices $a,b\in\pa\O$ of a simply connected $\O$, we
denote a \emph{boundary arc} $[ab]_\O\ss\pa\O$ as the
set of all
boundary vertices lying between $a$ and $b$ (including those two) when
one goes along $\pa\O$ in the \emph{counterclockwise} direction (so
$[ab]_\O\cup[ba]_\O=\pa\O$ and $[ab]_\O\cap[ba]_\O=\{a,b\}$); see
Figure~\ref{Fig: grid-irreg}. We also use the notation $[ab)_\O
:=[ab]_\O
\setminus\{b\}$, $(ab]_\O:=[ab]_\O\setminus\{a\}$, etc.

For a given vertex $u\in\G$ and $r>0$, we denote by $\BBsign
^\G
_r(u)$ the \emph{discrete disc} of radius $r$ around $u$. Namely,
$\Int
\BBsign^\G_r(u)$ is the set of all vertices $v\in\G$ lying in the
connected component of $\G\cap\{v\dvtx|v - u|<r\}$ containing $u$ (e.g.,
$\Int\BBsign^\G_{r_u}(u)=\{u\}$), and $\pa\BBsign^\G_r(u)$ is the set
of their neighbors; see Figure~\ref{Fig: grid-irreg}.

\begin{remark}
\label{Rem:FiniteVsSmall} Let $u\in\Gamma$ and $r>0$. The following
fact immediately follows from (\ref{LuvBound}):
\[
\mbox{if $v\in\Gamma$ is such that $|v - u|< \nu_0^{-1}r$,
then $v\in \Int\BBsign^\G_r(u)$.}
\]
Combining this with Remark~\ref{Rem:RvBound}, one easily concludes
that, for all $u\in\Gamma$ and $r\ge r_u$,
%
\begin{equation}
\label{SumRvBound} \sum_{v\in\Int\BBsign^\G_r(u)}r_v^2
\asymp r^2,
\end{equation}
where constants in $\asymp$ depend on $\eta_0,\varkappa_0$ and $\nu
_0$ only.
\end{remark}

Below we also need a stronger version of (\ref{SumRvBound}). Given an
interval $I\subset\R/(2\pi\Z)$ of length $\pi- \eta_0$, let $\Int
_{[I]}\BBsign^\G_{r}(u)$ denote the set of all vertices $v\in\G$ that
can be connected to $u$ by a nearest-neighbor path $(u_0u_1\cdots u_n)$
such that all $u_s$ (including $v=u_n$) satisfy $|u_s-u|<r$ and $\arg
(u_s-u)\in I$. In other words, we restrict ourselves to those $v\in
\BBsign^\G_{r}(u)$ that are connected to $u$ by nearest-neighbor paths
running in a given sector
\[
S(u,r,I):=\bigl\{z\in\C\dvtx|z - u|<r, \arg{(z - u)}\in I\bigr\}.
\]

\begin{lemma}
\label{Lemma:SumRvBoundI}
For all $u\in\Gamma$, $r\ge r_u$ and intervals $I\ss\R/(2\pi\Z)$ of
length $\pi- \eta_0$, one has
%
\begin{equation}
\label{SumRvBoundI} \sum_{v\in\Int_{[I]}\BBsign^\G_r(u)}r_v^2
\asymp r^2,
\end{equation}
where constants in $\asymp$ depend on $\eta_0,\varkappa_0$ and $\nu
_0$ only.
\end{lemma}

\begin{pf} The upper bound follows from (\ref{SumRvBound}). To prove
the lower bound, note that if $r_v$ is comparable to $r$ for at least
one vertex $v\in\Int_{[I]}\BBsign^\G_r(u)$, then we are done as the
corresponding term $r_v^2$ of the sum is comparable to $r^2$. On the
other hand, if $r_v\ll r$ for all $v\in\Int_{[I]}\BBsign^\G_r(u)$, then
one can use assumption (b) step by step in order to find a path running
from $u$ in the bulk of the sector $S(u,r,I)$.
In particular, in this case there exists a vertex $u'\in\Gamma$ such
that $\Int\BBsign^\G_{r'}(u')\subset\Int_{[I]}\BBsign^\G_r(u)$, where
$r':=r\cdot\sin(\frac{1}{2}(\pi-\eta_0))/2$. Then the lower bound in
(\ref{SumRvBoundI}) follows from (\ref{SumRvBound}) applied to the disc
$\BBsign^\G_{r'}(u')$.
\end{pf}

\subsection{Green's function, exit probabilities and partition
functions of the random walk in a discrete domain} \label
{SubSect:PartitionFctRW}

Let $\Omega$ be a (simply connected) discrete domain.
For a real function $H\dvtx\O\to\R$, we define its \emph{discrete
Laplacian} by
\[
[\Delta H](v):=\sum_{(vv')\in\mathrm{E}^\G} \varpi_{vv'}
\bigl(H\bigl(v'\bigr) - H(v)\bigr),\qquad v\in\Int\O,
\]
where the sum is taken over all neighbors of $v$, and $\varpi_{vv'}$
are given by (\ref{RWtransprobaDef}). We say that $H$ is \emph{discrete
harmonic in ${\O}$} if $[\Delta H](v)=0$ for all $v\in\Int\O$.

Below we often use two basic notions of discrete potential theory. The
first is the discrete \emph{harmonic measure} ${\hm\O uE}$
of a
boundary set \mbox{$E\ss\pa\O$} seen from an (inner) vertex $u\in
\O$.
It can be defined as the unique function which is discrete harmonic
in $\O$ and coincides with $\1_E(\ccdot)$ on $\pa\O$. At the same time,
${\hm\O uE}$ admits a simple probabilistic interpretation: it is the
probability of the event that the random walk (\ref{RWtransprobaDef})
on $\Gamma$ started at $u$ first hits $\pa\O$ on $E$. The second notion
is the (positive) \emph{Green function} $G_\O(v;u)$. It is the
unique function which is discrete harmonic everywhere in $\O$ except at
$u$, vanishes on the boundary $\pa\O$ and such that
\[
\bigl[\Delta G_\O(\ccdot;u)\bigr](u)=-\mu_u^{-1}.
\]
From the probabilistic point of view, $G_\O(v;u)$ is the expected
number of visits at $u$ (divided by $\mu_u$) of random walk (\ref
{RWtransprobaDef}) started at $v$ and stopped when reaching $\pa\O$.
Note that $G_\O$ is symmetric, that is, $G_\O(u;v)\equiv G_\O(v;u)$;
for example, see Remark~\ref{Rem:Z=Harm}(ii). The following notation
generalizes both discrete harmonic measure and Green's function.

\begin{definition}
Let $\O\ss\G$ be a bounded discrete domain and $x,y\in\O$. We
denote by
${\ZRW{\O}{x}{y}}$ the \emph{partition function of the
random walk}
joining $x$
and $y$ inside $\O$. Namely,
%
\begin{equation}
\label{ZRWdef} \ZRW{\O}xy:=\sum_{\gamma\in S_\O(x;y)}\mathrm {w}(
\gamma),
\end{equation}
where
\[
\mathrm{w}(\gamma):=\frac{\prod_{s=0}^{n(\gamma)-1}\mathrm
{w}_{u_su_{s+1}}}{\prod_{s=0}^{n(\gamma)}\mu_{u_s}}= \mu _{y}^{-1}\prod
_{s=0}^{n(\gamma)-1}\varpi_{u_su_{s+1}}
\]
and $S_\O(x;y)=\{\gamma= (u_0 \sim u_1 \sim\cdots \sim u_{n(\gamma
)})\dvtx u_0=x; u_1,\ldots,u_{n(\gamma)-1}\in
\Int\O;\break u_{n(\gamma)}=y\}$ is the set of all nearest-neighbor paths
connecting $x$ and $y$ inside
$\O$. Further, for $A,B\ss\O$, we define
\[
\ZRW\O AB:=\sum_{x\in A, y\in B}\ZRW\O xy,
\]
and by $\RW{\O}AB$ we denote a random nearest-neighbor path
$\gamma$ chosen from the set $S_\O(A;B):=\bigcup_{x\in A,y\in B}S_\O(x;y)$
with probabilities 
proportional to the weights $\mathrm{w}(\gamma)$.
\end{definition}


\begin{remark} \label{Rem:Z=Harm}
It is easy to see that:

\begin{longlist}[(ii)]

\item[(i)] if $u\in\Int\O$ and $b\in\pa\O$, then $\ZRW\O{u}b=\mu
_b^{-1}\hm\O{u}b$;

\item[(ii)] if both $u,v\in\Int\O$, then $\ZRW\O{v}{u}=G_\O(v;u)$.
\end{longlist}
\end{remark}

\begin{pf}
(i) Focusing on the first step of $\gamma\in S_\O(u;b)$ in (\ref{ZRWdef}),
one immediately concludes that the function
\[
H(u):= %
\cases{\ZRW\O{u}b, &\quad  $u\in\Int\O$,
\vspace*{2pt}\cr
\mu_b^{-1}
\1[u=b],& \quad $u\in\pa\O,$ } %
\]
is discrete harmonic in $\O$ and coincides with $\mu_b^{-1}\hm\O
\ccdot
b$ on the boundary $\pa\O$. Thus, $\ZRW\O{u}b=H(u)=\mu_b^{-1}\hm
\O{u}b$
for all $u\in\Int\O$.

(ii) As above, it immediately follows from (\ref{ZRWdef})
that the
function
\[
H(v):= %
\cases{\ZRW\O{v} {u}, &\quad $ v\in\Int\O$,
\vspace*{2pt}\cr
0, &\quad $v\in\pa\O$, }
\]
is discrete harmonic everywhere in $\O$, except at $u$ and
\[
H(u)=\mu_u^{-1}+\sum_{(uu')\in E^\G}
\varpi_{uu'}H\bigl(u'\bigr),
\]
where the first term $\mu_u^{-1}$ corresponds to the trivial trajectory
consisting of a single point $u$. Thus $[\Delta H](u)=-\mu_u^{-1}$ and
$\ZRW\O{v}{u}=H(v)=G_\O(v;u)$ for all $v\in\Int\O$.
\end{pf}

\subsection{Properties \texorpdfstring{\protect\ref{PropertyS}}{(S)} and
\texorpdfstring{\protect\ref{PropertyT}}{(T)} of the random walk on ${\Gamma}$}
\label{SubSect:RandomWalkAssumptions}
Our paper is based on two crucial properties, \ref{PropertyS}, \ref
{PropertyT}, of random walk (\ref{RWtransprobaDef}) on $\G$ that are
formulated below.


\renewcommand{\theproperty}{(S)}
\begin{property}[(``Space''; see \cite{ABGGN}, Theorem~1.4)]
\label{PropertyS}
There exists a constant $c_0=c_0(\varpi_0,\eta_0,\varkappa_0)>0$ such
that, uniformly over all vertices $u\in\G$, radii $r>0$ and intervals
$I\subset\R/(2\pi\Z)$ of length $\pi- \eta_0$, the following is fulfilled:
\[
\omega_{\BBsign^\G_r(u)}\bigl(u;\bigl\{a\in\pa\BBsign^\G_r(u)
\dvtx\arg (a - u)\in I\bigr\}\bigr) \ge c_0.
\]
\end{property}

In other words, the random walk started at the center of
{any} discrete disc $\BBsign^\G_r(u)$ can exit this disc through any
given boundary arc of the angle $\pi- \eta_0$ with probability
uniformly bounded away from $0$. Note that, if $r\le r_u$, then $\Int
\BBsign^\G_r(u)=\{u\}$, and the claim rephrases assumption (b).

\renewcommand{\theproperty}{(T)}
\begin{property}[(``Time''; see \cite{ABGGN}, Theorem~1.5)]
\label{PropertyT}
There exists a constant $C_0=C_0(\varpi_0,\eta_0,\varkappa_0)>1$ such
that, uniformly over all vertices $u\in\G$ and radii $r\ge r_u$, the
following is fulfilled:
%
\begin{equation}
\label{BoundPropertyT} C_0^{-1}r^2 \le\sum
_{v\in\Int\BBsign^\G_r(u)}r_v^2 G_{\BBsign^\G_r(u)}(v;u)
\le C_0r^2.
\end{equation}
\end{property}

Despite the fact that \ref{PropertyT} is formulated in terms of
discrete harmonic functions only (which do not depend on a particular
time parametrization of the underlying random walk), it is natural to
mention the following interpretation: let us consider some time
parametrization such that the (expected) time spent by the walk at a
vertex $v$ before it jumps is of order $r_v^2$ (recall that local
scales $r_v$ can be quite different for different $v$'s). Then we ask
the expected time spent in a discrete disc $\BBsign_r(u)$ by the random
walk started at $u$ before it hits $\pa\BBsign^\G_r(u)$ to be of order
$r^2$, uniformly over all possible discrete discs.

\begin{remark}
\label{Rem:ad=>ST}
In the first version of this paper, \ref{PropertyS} and \ref{PropertyT}
were presented as additional {``black box assumptions''} and the
following question was posed: do they hold true for any embedding
satisfying (a)--(d) [with some ``quantitative'' version of (d) which
the author, at the time, thought to be necessary] or not? Very
recently, the positive answer to this question was given in \cite{ABGGN},
\[
\mbox{(a)--(d) always imply \ref{PropertyS} and \ref{PropertyT}}.
\]
The proofs in \cite{ABGGN} are based on heat kernel estimates and the
parabolic Harnack inequality; see also a useful discussion in \cite{Koz07},
Section~2.1. We are grateful to the authors of \cite{ABGGN}
for helpful conversations on the subject. Also, it is worth noting that
in some ``integrable'' cases (e.g., for simple random walks on regular
lattices or special random walks on isoradial graphs \cite{CS11}) \ref
{PropertyS} and \ref{PropertyT} can be easily obtained due to nice
``local approximation properties'' of the random walk (\ref
{RWtransprobaDef}). In those cases, all the results of our paper can be
obtained without any further references. In some sense, we consider
\ref
{PropertyS} and \ref{PropertyT} as a ``pointe de la jonction'': being
formulated for simplest possible discrete domains (approximations of
Euclidean discs), they provide a starting point for our toolbox which
is more adapted for very rough $\O$'s.
\end{remark}


\subsection{Basic facts: Elliptic Harnack
inequality, Green's function estimates and Beurling-type estimates}
\label{SubSect:BasicFacts}
In this section we collect several basic facts about discrete harmonic
functions. These statements can be obtained using heat kernel estimates
\`a la \cite{ABGGN}, though to keep the whole presentation
self-contained we also provide direct proofs based on \ref{PropertyS}
and \ref{PropertyT} in the \hyperref[app]{Appendix}.

\begin{proposition}[(Elliptic Harnack inequality)]
\label{Proposition:Harnack} For each $\rho>1$, there exists a constant
$c(\rho)=c(\rho,\varpi_0,\eta_0,\varkappa_0)>0$ such that, for any
$u\in
\Gamma$, $r>0$ and any nonnegative harmonic function $H\dvtx\BBsign
^\G
_{\rho r}(u)\to\R_+$, one has
\[
\min_{v\in\Int\BBsign^\G_{r}(u)} H(v) \ge c(\rho)\max
_{v\in\Int
\BBsign
^\G_{r}(u)} H(v).
\]
%
\end{proposition}

\begin{pf}
This result appears in \cite{ABGGN}. In order to keep the presentation
self-contained, we also give a simple proof based on \ref{PropertyS} in
the \hyperref[app]{Appendix}.
\end{pf}

\begin{lemma}[(Green's function estimates)]
\label{Lemma:GreenInDiscs}For each $\rho>1$, there exist constants
$c_{1,2}(\rho)=c_{1,2}(\rho,\varpi_0,\eta_0,\varkappa_0)>0$ such that,
for any $u\in\Gamma$ and $r>0$, the following holds:
\begin{eqnarray*}
G_{\BBsign^\G_{\rho r}(u)}(v;u)&\ge& c_1(\rho) \qquad \mbox{for all } v\in\Int
\BBsign^\G_r(u);
\\
G_{\BBsign^\G_{\rho r}(u)}(v;u)&\le&
c_2(\rho) \qquad \mbox{for all } v\in\BBsign^\G_{\rho r}(u)
\setminus\Int\BBsign^\G_r(u).
\end{eqnarray*}
\end{lemma}

\begin{pf}
See the \hyperref[app]{Appendix}.
\end{pf}

\begin{lemma}[(Crossings of annuli)] \label{Lemma:AroundAnnulus}
There exist two constants $\rho_0=\rho_0(\varpi_0,\break \eta_0,\varkappa
_0)>1$ and $\delta_0=\delta_0(\varpi_0,\eta_0,\varkappa_0)>0$ such that
the following is fulfilled: for any $u\in\Gamma$, $r>0$ and any
nearest-neighbor path $\gamma\subset\Gamma$ crossing the annulus
\[
A\bigl(u,\rho_0^{-1}r,r\bigr)=\bigl\{z\in\C\dvtx
\rho_0^{-1}r<|z - u|<r\bigr\},
\]
the probability of the event that the random walk (\ref
{RWtransprobaDef}) crosses $A(u,\rho_0^{-1}r,r)$ without hitting the
path $\gamma$ is bounded from above by $1 - \delta_0$.
\end{lemma}

\begin{pf}
This easily follows from successive applications of \ref{PropertyS};
see the \hyperref[app]{Appendix} for details.
\end{pf}

\begin{lemma}[(Weak Beurling-type estimate)]
\label{Lemma:b0-weakBeurling}
Let $\beta_0:=-\frac{\log(1-\delta_0)}{\log\rho_0}$. Then, for any
simply connected discrete domain $\O$, an inner vertex $u\in\Int\O$ and
a set $E\subset\pa\O$, the following is fulfilled:
\[
\omega_\O(u;E) \le\biggl[\rho_0\cdot\frac{\dist(u;\partial\O)}{\dist
_{\O
}(u;E)}
\biggr]^{\b_0} \quad\mbox{and}\quad \omega_\O(u;E) \le\biggl[
\rho_0\cdot\frac{\diam E}{\dist_{\O
}(u;E)}\biggr]^{\b_0},
\]
where $\dist_\O(u;E):=\inf\{r \dvtx u \mbox{ and } E \mbox
{ are connected in } \O\cap B^\G_r(u)\}$.
Above we set $\diam E:=r_{x}$ if $E=\{x\}$ consists of a single
boundary vertex.
\end{lemma}

\begin{pf}
This immediately follows from Lemma~\ref{Lemma:AroundAnnulus}; see the
\hyperref[app]{Appendix} for details.
\end{pf}

For $u\in\O$ and $r>0$, we denote by $\BBsign^\O_r(u)$ the
${r}$-\emph{neighborhood of} ${u}$ \emph{in} ${\O}$.
More rigorously,
we set $\Int\BBsign^\O_r(u)$ to be the connected component of $\Int
\O
\cap\Int\BBsign^\G_r(u)$ containing $u$ if $u\in\Int\O$, and containing
$x_{\mathrm{int}}$ if $u=x\in\pa\O$. In particular, we set $\BBsign
^\O
_r(x)=\Int\BBsign^\O_r(x)=\es$ if $x\in\pa\O$ and $r\le
|x_{\mathrm
{int}} - x|$. The next lemma allows us to control the behavior of
positive harmonic functions near a part of $\pa\O$ where they satisfy
Dirichlet boundary conditions.

\begin{lemma}[(Boundary behavior)]
\label{Lemma:BoundaryDecay}
Let $\Omega$ be a simply connected discrete domain, $u\in\Int\Omega$,
$r:=\dist(u;\partial\O)$ and $x\in\partial\Omega$ be the closest
boundary vertex to $u$ (so that $r=|u - x|$) and $\rL_{ux}$ denote
the path running from $u$ to $x$ constructed in Remark~\ref
{Rem:LasymptDist}. Let a vertex $u'\in\rL_{ux}$ be such that $|u' -
x|\le r':=\rho_0^{-1}r$ and $\rL_{ux}^{uu'}\subset\Int\O$, where
$\rL
^{uu'}_{ux}$ denotes the portion of $\rL_{ux}$ from $u$ to $u'$. Then,
for any nonnegative harmonic function $H\dvtx\BBsign^\O_{r}(x)\to\R_+$
vanishing on $\partial\Omega\cap\partial\BBsign^\O_{r}(x)$, one has
\[
H\bigl(v'\bigr) \le\delta_0^{-1}
\rho_0^{2\b_0}\cdot\bigl[ \bigl|v' - x\bigr|/r
\bigr]^{\b
_0}\cdot \max_{v\in\rL^{uu'}_{ux}}H(v) \qquad\mbox{for all }
v'\in{\BBsign^\O _{r'}(x)}.
\]
\end{lemma}

\begin{pf}
This follows from (a version of) Lemma~\ref{Lemma:AroundAnnulus}; see
the \hyperref[app]{Appendix} for details.
\end{pf}

The last fact that we use below is the following uniform bound for the
Green function $G_\O$ in an arbitrary $\Omega$ in terms of Green's
functions in the appropriate discs.

\begin{lemma} \label{Lemma:GreenEstimatesViaG} Let an integer $n_0$ be
chosen so that $(1 - \delta_0)^{n_0}\le\frac{1}{3}$ , $\O$ be a
simply connected discrete domain, $u\in\Int\Omega$, $r:=\dist(u;\pa
\O
)$ and $R:=\rho_0^{2n_0}r$. Then
\[
G_{\BBsign^\G_r(u)}(v;u)\le G_\O(v;u)\le2G_{\BBsign^\G_R(u)}(v;u)
\qquad\mbox{for all } v\in\Int\BBsign^\G_{r}(u).
\]
\end{lemma}

\begin{pf}
This also follows from Lemma~\ref{Lemma:AroundAnnulus}; see the
\hyperref[app]{Appendix} for details.
\end{pf}

\begin{remark}
From now on, we think about the constants $\varpi_0,\eta_0,\varkappa_0$
used in assumptions (a)--(c) and all other constants that appeared in
this section [like $c_0=c_0(\varpi_0,\eta_0,\varkappa_0)$ and
$C_0=C_0(\varpi_0,\eta_0,\varkappa_0)$ in Properties \ref
{PropertyS}, \ref{PropertyT}, etc.] as \emph{fixed once forever}.
Thus, below we say, for example, ``with some uniform constants $\const
_1$ and $\const_2$,'' meaning that $\const_{1,2}$ may, in general,
depend on $\varpi_0,\eta_0,\varkappa_0$, but are independent of all
other parameters involved (like domain shape, location of boundary
points or particular graph structure).
\end{remark}

\section{Factorization theorem for the function \texorpdfstring{$\ZRW{\O}a{[bc]_\O}$}{$\ZRW{Omega}a{[bc]_{Omega}}$}}
\label{Sect:Factorization}

The main result of this section is Theorem~\ref{Thm:Zfact}. It deals
with a simply connected discrete domain $\O$ and three marked boundary
points $a,b,c\in\pa\O$ [no assumptions about actual geometry of $(\O
;a,b,c)$ are used] and provides a uniform up-to-constant factorization
of the three-point function $\ZRW\O{a}{[bc]_\O}$ via $\ZRW\O{a}{b}$,
$\ZRW\O{a}{c}$ and $\ZRW\O{b}{c}$. Actually, our proof is based on a
factorization of the latter two-point functions via some inner point
$u\in\Int\O$ which is ``not too close'' to any of the boundary arcs
$[ab]_\O$, $[bc]_\O$ and $[ca]_\O$. Thus our strategy to prove
Theorem~\ref{Thm:Zfact} can be described as follows:
\begin{itemize}
\item prove that the ratio $\ZRW{\O}au\ZRW{\O}ub/\ZRW{\O}ab$ is
uniformly comparable with the probability of the event that $\RW{\O}ab$
passes ``not very far'' from $u$ [namely, at distance less than $\frac
{1}{3}\dist(u;\pa\O)$]; see Proposition~\ref{Prop:ZuaZub/Zab=P};
\item prove that this probability is bounded below if $u$ is ``not too
close'' to any of the boundary arcs $[ab]_\O$ and $[ba]_\O$; see
Lemma~\ref{Lemma:LuCrossing} and Proposition~\ref{Prop:RWintersectsBu};
\item find an inner vertex $u$ which is ``not too close'' to any of
$[ab]_\O$, $[bc]_\O$ and $[ca]_\O$ (see Lemma~\ref{Lemma:HMuabc}) and
factorize all $\ZRWsign_\O$'s using this $u$.
\end{itemize}
Below we use the following notation. For a discrete domain $\O$ and
$u\in\Int\O$, let
%
\begin{equation}
\label{BBdef} \ddsign_\O(u) :=\dist(u;\pa\O)=\min
_{x\in\pa\O}|u - x|,\qquad \BBsign_\O(u) :=
\BBsign^\G_{\dd{\O}u/3}(u).
\end{equation}
Recall that (\ref{BBdef}) means $\Int\BB{\O}u=\{v\in\G\dvtx
|v-u|<\textfrac
{1}{3}\dist(u;\pa\O)\}$ (or, more accurately, a connected component of
this set; see Figure~\ref{Fig: grid-irreg}), and \mbox{$\pa\BB{\O
}u\ss\O
$} is the set of all vertices neighboring to $\Int\BB{\O}u$. We also
generalize notation (\ref{ZRWdef}) in the following way: for a given
subdomain $U\ss\O$ and a random walk path $\gamma=(u_0\sim u_1\sim
\cdots\sim
u_{n(\gamma)})$, let
\[
\TTsign_U(\gamma) :=\sum_{s=0}^{n(\gamma)}r_{u_s}^2
\1[u_s\in\Int U].
\]
%
Then, for $A,B\ss\O$, we define
\[
 \ZRWsign_\O[\TTsign_U](A;B) :=\sum
_{\gamma\in S_\O(A;B)}\mathrm {w}(\gamma ) \TTsign_U(\gamma) .
\]
Note that
\[
\frac{
\ZRWsign_\O[\TTsign_U](A;B)
}{\ZRW{\O}AB}=\E\bigl[ \TTsign_U\bigl(\RWsign_{\O}(A;B)
\bigr) \bigr]
\]
is the expected time spent in $U$ by a (properly parameterized) random
walk $\RW{\O}AB$.

\begin{proposition}
\label{Prop:ZuaZub/Zab=P} Let $\O$ be a simply connected discrete
domain, $a,b\in\pa\O$, and
\mbox{$u\in\Int\O$}. Then the following double-sided estimate is fulfilled:
%
\begin{equation}
\label{ZuaZab/Zab=P} \frac{\ZRW{\O}ua\ZRW{\O}ub}{\ZRW{\O
}ab}\asymp \P\bigl[ \RWsign_{\O}(a;b)
\cap\Int \BBsign_\O(u)\ne\es\bigr] ,
\end{equation}
with some uniform (i.e., independent of $\O,a,b,u$) constants.
\end{proposition}

\begin{pf}
Recall that both functions $\ZRW{\O}{\ccdot}a$ 
and $\ZRW{\O}{\ccdot}b$ 
are discrete harmonic and positive inside $\O$. Therefore, Harnack's
principle (see Proposition~\ref{Proposition:Harnack}) gives
%
\begin{equation}
\label{xZuaZub=SumZvaZvb} \ZRW{\O}ua\ZRW{\O}ub\asymp \frac{\sum_{v\in\Int
\BBsign_\O(u)}r_v^2 \ZRW{\O}va\ZRW{\O}vb}{\sum_{v\in\Int
\BBsign_\O(u)}r_v^2}.
\end{equation}
Recall that $\sum_{v\in\Int\BB{\O}u}r_v^2\asymp(\dd{\O}u)^2$
due to
(\ref{SumRvBound}).

Joining two random walk paths $\gamma_{av}$ (from $a$ to $v$) and
$\gamma_{vb}$ (from $v$ to $b$), and taking into account $\mathrm
{w}(\gamma_{av}\gamma_{vb})=\mu_v\cdot\mathrm{w}(\gamma
_{av})\mathrm
{w}(\gamma_{vb})\asymp\mathrm{w}(\gamma_{av})\mathrm{w}(\gamma_{vb})$,
it is easy to see
that
%
\begin{equation}
\label{xSumZvaZvb=ZTab} \sum_{v\in\Int
\BBsign_\O(u)} r_v^2
\ZRW{\O}va\ZRW{\O}vb\asymp { \ZRWsign_\O[\TTsign_{\BB{\O}u}](a;b)}
\end{equation}
[indeed, each of the vertices $u_s\in\RW{\O}ab$ contributing to
$\TTsign
_{{\BB{\O}u}}$ can be chosen as $v$ in order to split $\RW{\O}ab$
into two halves $\gamma_{av}$ and $\gamma_{vb}$].

Further, let $w$ denote the \emph{first} vertex $u_s\in\Int\BB{\O
}u$ of
$\RW{\O}ab$, if such a vertex exists. Since on the right-hand side of
(\ref{xSumZvaZvb=ZTab}) we do not count those paths which do not intersect
$\BB{\O}u$, by splitting $\RW{\O}ab$ into two halves at $w$, it can be
rewritten as
%
\begin{equation}
\label{zZTab=SumZawZTwb} { \ZRWsign_\O[\TTsign_{\BB{\O}u}](a;b)} \asymp
\sum_{w\in\Int\BB{\O}u} \ZRW{\O\setminus \BBsign_\O(u)}aw
\ZRWsign_\O[\TTsign_{\BB{\O}u}](w;b),
\end{equation}
where a (generally, doubly connected) discrete domain $\O':=\O
\setminus
\BB{\O}u$ should be understood so that $\Int\O'=\Int\O\setminus
\Int\BB
{\O}u$. 
It immediately follows from our definition of $\ZRWsign[\TTsign]$ and
Harnack's principle applied to the discrete harmonic function $\ZRW{\O
}{\ccdot}b$ that
%
\begin{eqnarray}
\label{zZTwb=Zwb} \ZRWsign_\O[\TTsign_{\BB{\O}u}](w;b) &\asymp&\sum
_{v\in\Int\BB{\O}u} r_v^2 \ZRW{\O}wv
\ZRW{\O}vb
\nonumber
\\[-8pt]
\\[-8pt]
\nonumber
& \asymp&\sum_{v\in\Int\BB
{\O}u} r_v^2
\ZRW{\O}wv\cdot\ZRW{\O}wb
\end{eqnarray}
(indeed, for each $u_s$ contributing to
$\TTsign_{{\BB{\O}u}}
=\sum_{s=0}^{n(\gamma)}r_{u_s}^2 \1[u_s\in\Int{
\BBsign_\O(u)}]$,
split the random path $\RW{\O}wb$ into two halves $\gamma
_{wv},\gamma_{vb}$ at
the point $v=u_s$ and use the up-to-constant multiplicativity $\mathrm
{w}(\gamma_{wv}\gamma_{vb})\asymp\mathrm{w}(\gamma_{wv})\mathrm
{w}(\gamma_{vb})$ once more). Moreover, it is easy to conclude that
%
\begin{equation}
\label{zSumG=const} \sum_{v\in\Int\BB{\O}u} r_v^2
\ZRW{\O}wv= \sum_{v\in
\Int\BB{\O}u} r_v^2
G_{\O}(v;w)\asymp\bigl( \ddsign_\O(u)\bigr)^2
\end{equation}
for any $w\in\Int\BB{\O}u$. Indeed, the upper bound follows from the
estimates $\frac{2}{3}\dd{\O}u\le\dd\O{w}\le\frac{4}{3}\dd
{\O}u$, the
inclusion $\BBsign_\O(u)\subset\BBsign^\G_{\dd\O{w}}(w)$, the upper
bound in Lemma~\ref{Lemma:GreenEstimatesViaG} and the upper bound in
(\ref{BoundPropertyT}). The lower bound is trivial if $r_w$ is
comparable to $\dd{\O}u$, and is guaranteed by Lemma~\ref
{Lemma:SumRvBoundI} and the lower bounds in Lemmas \ref
{Lemma:GreenInDiscs} and \ref{Lemma:GreenEstimatesViaG} if $r_w\ll\dd
{\O}u$. Combining (\ref{zSumG=const}) with (\ref
{xZuaZub=SumZvaZvb})--(\ref{zZTwb=Zwb}), one obtains
%
\begin{eqnarray}
\label{xZuaZub/Zab=} \frac{\ZRW{\O}ua\ZRW{\O}ub}{\ZRW{\O}ab}&\asymp&\bigl( \ddsign_\O(u)
\bigr)^{-2}\frac
{\ZRWT{\O}{{\BB{\O}u}}ab}{\ZRW{\O}ab}
\nonumber
\\[-8pt]
\\[-8pt]
\nonumber
&\asymp&\frac{\sum_{w\in\Int\BB{\O}u} \ZRW{\O\setminus
\BBsign_\O(u)}aw
\ZRW{\O}wb}{\ZRW{\O}ab}.
\end{eqnarray}
Finally, the numerator can be rewritten as
\[
\sum_{w\in\Int\BB{\O}u} \ZRW{\O\setminus \BBsign_\O(u)}aw
\ZRW{\O}wb \asymp\sum_{\gamma\in S_\O(a;b)\dvtx\gamma\cap\Int
\BB{\O
}u\ne\es
}\mathrm{w}(\gamma)
\]
[as above, denote by $w$ the first vertex $u_s\in\Int\BB{\O}u$ of
$\gamma$,
if it exists]. Thus (\ref{xZuaZub/Zab=}) is comparable to the
probability of the event $\gamma\cap\BB{\O}u\ne\es$.
\end{pf}

Let $u\in\Int\O$ be an inner vertex, $x\in\pa\O$ be the closest
boundary vertex to $u$ and $\rL_{ux}$ be the nearest-neighbor path from
$u$ to $x$ constructed in Remark~\ref{Rem:LasymptDist}. 
For $v\in\rL_{ux}$, let $\rL_{ux}^{vx}$ denote the portion of $\rL
_{ux}$ from $v$ to $x$, and let $\Length(\rL_{ux}^{vx})$ be the
Euclidean length of $\rL_{ux}^{vx}$. It is easy to see that
%
\begin{equation}
\label{DistToBdAlongL} \Length\bigl(\rL_{ux}^{vx}\bigr)\le\const
\cdot\, \ddsign_\O(v) \qquad\mbox{for al } v\in\rL_{ux}\cap\Int\O.
\end{equation}
Indeed, let $v$ belong to a face $f$ and $[z;z']:=[u;x]\cap f\ne\es$;
see Remark~\ref{Rem:LasymptDist}. Then
\begin{eqnarray*}
\Length\bigl(\rL_{ux}^{vx}\bigr)&\le&\const\cdot\,|z - x|=
\const\cdot\,\dist (z;\pa\O),
\\
\dist(z;\pa\O)&\le&|z - v|+ \ddsign_\O(v)\quad \mbox{and}\quad |z - v|\le \const
\cdot\, r_v \le\const\cdot\, \ddsign_\O(v).
\end{eqnarray*}

We denote by $\rL_\O(u)\ss\Int\O$ the portion of $\rL
_{ux}$ from
$u$ \emph{to the first hit of} $\pa\O$; see the top-left picture in
Figure~\ref{Fig: domain-truncated}.
Below we also use the notation
\[
\P_\O^{a,b}\bigl[\rL_\O(u)\bigr]:=\P
\bigl[ \RWsign_{\O}(a;b) \cap\rL_{\O
}(u)\ne\es\bigr]
\]
and the similar notation
\[
\P_\O^{a,b}\bigl[ \BBsign_\O(u)\bigr]:=
\P\bigl[ \RWsign_{\O}(a;b) \cap\Int \BBsign_\O(u) \ne\es
\bigr]
\]
for the right-hand side of (\ref{ZuaZab/Zab=P}).

\begin{lemma}
\label{Lemma:LuCrossing} Let $\O$ be a simply connected discrete
domain, $u\in\Int\O$, $x\in\pa\O$ be the closest boundary vertex
to $u$
[so that $\dd{\O}u=|u - x|$] and $a,b\in\pa\O$ be such that
$a,b\notin\BBsign^\O_{\dd{\O}u}(x)$. Then, for a path $\rL_\O
(u)$ defined
above, one has
\[
\P_\O^{a,b}\bigl[\rL_\O(u)\bigr]\le\const\cdot\,
\P_\O^{a,b}\bigl[ \BBsign_\O(u)\bigr].
\]
\end{lemma}

\begin{pf}
Let a sequence of vertices $u=v_0,v_1,\ldots,v_n\in\Int\O$ be defined
inductively by the
following rule: \mbox{$v_{k+1}\in\rL_\O(u)$} is the first vertex on
$\rL
_\O(u)$ after $v_k$ (when going toward $\pa\O$) which does not belong
to $\Int\BB\O{v_k}$. Thus each $v_{k+1}\in\pa\BB\O{v_k}$ and
\[
\rL_\O(u)\ss\bigcup_{k=0}^n
\Int \BBsign_\O(v_k)
\]
[note that the local finiteness assumption (d) guarantees $n<\infty$,
but we do not have any quantitative bound for this number]. Further,
let $u':=v_m$ be the first of those vertices such that $|v_k - x|\le
\rho
_0^{-1}\dd{\O}u$ for all $k\ge m$, where $\rho_0$ is the constant
used in
Lemma~\ref{Lemma:BoundaryDecay} [we set $m:=n$, if $|v_n - x|>\rho
_0^{-1}\dd{\O}u$]. It immediately follows from (\ref{DistToBdAlongL})
that $\Length(\rL_{ ux}^{v_kx})$ decays exponentially as $k$
grows. In particular, this implies a uniform estimate $m\le\const$.

\begin{figure}

\includegraphics{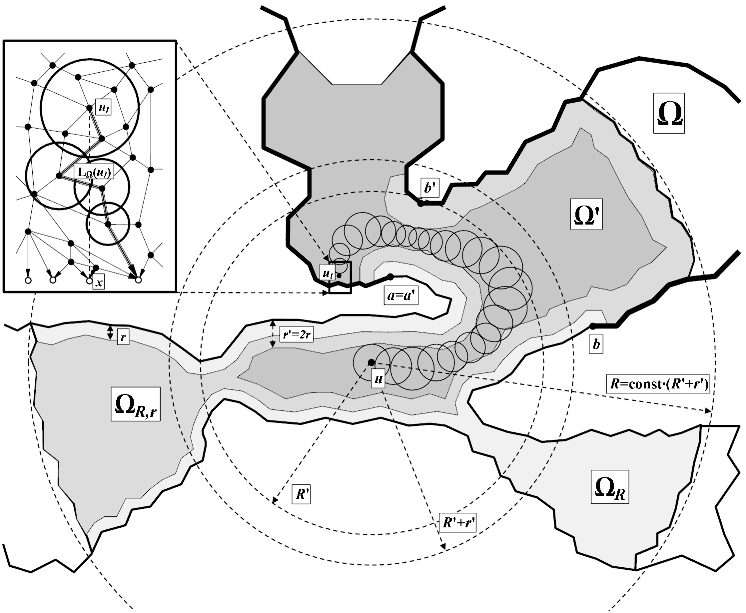}

\caption{The notation from the proof of
Proposition \protect\ref{Prop:RWintersectsBu}: a simply connected
domain $\O$
and its subdomains $\O_{R,r}\subset\O_R\subset\O$. The boundary arc
$[ba]_\O\subset\pa\O$ and a uniformly bounded number of discrete discs
$\BB\O{u_k}$ that cover a path $L'$ running from $u$ to $[ba]_\O$ in
$\O
'\subset\O_{R,r}$ are shown. The top-left picture: a vertex $u_l$, the
closest to $u_l$ boundary vertex $x$, the path $\rL_\O(u_l)\subset
\rL
_{u_lx}$ running from $u_l$ to $\pa\O$ and the sequence of discrete
discs $\BB\O{v_k}$ with exponentially decaying radii that cover $\rL
_\O(u_l)$.}
\label{Fig: domain-truncated}
\end{figure}

Let $H=\ZRW{\O}{\ccdot}a$ or $H=\ZRW{\O}{\ccdot}b$. The Harnack
principle (Proposition~\ref{Proposition:Harnack}) gives
\[
H(u)=H(v_0)\asymp H(v_1)\asymp\cdots\asymp
H(v_m).
\]
Moreover, by our assumption $a,b\notin\BBsign^\O_{\dd{\O}u}(x)$.
Thus Lemma~\ref{Lemma:BoundaryDecay} yields
\[
H(v_k)\le\const\cdot\,\bigl(|v_k - x|/
\ddsign_\O(u)\bigr)^{\beta_0}\cdot H(u),\qquad k\ge m.
\]
Then Proposition~\ref{Prop:ZuaZub/Zab=P} applied to each of the balls
$\BB{\O}{v_k}$ allows us to conclude that
%
\begin{eqnarray}
\label{xP[L]estimate} \P_\O^{a,b}\bigl[\rL_\O(u)
\bigr]&\le&\sum_{k=0}^{n}
\P_\O^{a,b}\bigl[ \BBsign_\O(v_k)
\bigr] \asymp\sum_{k=0}^{n}
\frac{\ZRW{\O}{v_k}a\ZRW{\O}{v_k}b}{\ZRW
{\O}ab}\nonumber
\\
& \le&\const\cdot\,\frac{\ZRW{\O}{u}a\ZRW{\O}{u}b}{\ZRW{\O
}ab}\cdot \Biggl[m+\sum
_{k=m}^n \biggl(\frac{|v_k-x|}{\dd{\O}u}
\biggr)^{ 2\beta
_0}\Biggr]
\\
& \le&\const\cdot\,\frac{\ZRW{\O}{u}a\ZRW{\O}{u}b}{\ZRW{\O
}ab}\asymp\P
_\O^{a,b}\bigl[ \BBsign_\O(u)\bigr]\nonumber
\end{eqnarray}
[recall that the distances $|v_k-x|\le\Length(\rL_{ ux}^{v_kx})$
decay exponentially for $k\ge m$, so the final bound does \emph
{not} depend on $n$].
\end{pf}

\begin{proposition}
\label{Prop:RWintersectsBu} Let $\O$ be a simply connected discrete
domain, $a,b\in\pa\O$,
$u\in\Int\O$, and $\sigma>0$ be such that both
$\hm{\O}u{[ab]_\O}, \hm{\O}u{[ba]_\O}\ge\sigma$. 
Then the uniform estimate
%
\begin{equation}
\label{RWintersectsBu} \P_\O^{a,b}\bigl[ \BBsign_\O(u)
\bigr]\ge\const(\sigma)
\end{equation}
holds true, with some $\const(\sigma)>0$ independent of $\O,a,b,u$.
\end{proposition}

\begin{pf}
For simplicity, let us rescale the underlying graph $\G$ so that  \mbox{$\dd
{\O}u=1$}.
We begin the proof with the following claim that is a corollary of the
weak Beurling estimates (Lemma~\ref{Lemma:b0-weakBeurling}) and our
assumption on the harmonic measures of $[ab]_\O$ and $[ba]_\O$: there
exist two constants $R=R(\sigma)>0$ and $r=r(\sigma)>0$ such that $u$
remains connected to $[ba]_\O$ in a ``truncated'' domain $\Omega_{R,r}$
defined as
\[
\Int\Omega_{R,r}:=\Int\BBsign^\O_{R}(u)
{}\Big\backslash{}\bigcup_{x\in
[ab]_{\O}}\Int\BBsign^\O_{r}(x)
\]
%
(more rigorously, $\Int\Omega_{R,r}$ is a connected component of this
set containing $u$; see Figure~\ref{Fig: domain-truncated}) and vice
versa with $[ba]_\O$ and $[ab]_\O$ interchanged. Let us emphasize that
$R(\s)$ \emph{and} $r(\s)$ \emph{can be chosen uniformly for all} $\O,a,b$
\emph{and} $u$.

Indeed, the first estimate in Lemma~\ref{Lemma:b0-weakBeurling} implies
that one can find a constant $R'=R'(\sigma)>0$ (independently of $\O
,a,b$ and $u$) so that
\[
\omega_{\O_{R'}}\bigl(u;\pa\BBsign^\G_{R'}(u)\cap
\pa\O_{R'}\bigr) \le\tfrac
{1}{2}\sigma \qquad\mbox{where }
\O_{R'}:=\BBsign^\O_{R'}(u)
\]
[note that, by definition, $\pa\O_{R'}\subset\pa\O\cup\pa\BBsign
^\G
_{R'}(u)$]. Let $a',b'\in\pa\O$ be chosen so that $[b'a']_\O\subset
[ba]_\O$ is the minimal boundary arc of $\O$ satisfying
\[
\bigl[b'a'\bigr]_\O\cap\pa
\O_{R'}=[ba]_\O\cap\pa\O_{R'};
\]
see Figure~\ref{Fig: domain-truncated}. Then, we set $R:=\const\cdot\,
(R' + r')$, where $r'\le1$ will be fixed later and the (uniform)
multiplicative constant is chosen according to Remark~\ref{Rem:RvBound}
so that no face of $\G$ crosses both boundaries of the annulus $A(u,R'
+ r',R)$. As above, denote $\O_{R}:=\BBsign^\O_{R}(u)$. It is easy to
see that $a',b'\in\pa\O_{R}$ and
\begin{eqnarray*}
\omega_{\O_{R}}\bigl(u;\bigl[b'a'
\bigr]_{\O_{R}}\bigr) & \ge& \omega_{\O_{R}}\bigl(u;
\bigl[b'a'\bigr]_{\O
}\cap\pa\O
_{R}\bigr) 
\ge \omega_{\O_{R'}}\bigl(u;[ba]_{\O}
\cap\pa\O_{R'}\bigr)
\\
& \ge& \omega_\O
\bigl(u;[ba]_{\O
}\bigr) - \omega_{\O_{R'}}\bigl(u;\pa
\BBsign^\G_{R'}(u)\cap\pa\O_{R'}\bigr) \ge
\tfrac
{1}{2}\sigma.
\end{eqnarray*}
In particular, $u$ is connected to the boundary arc $[b'a']_{\O_{R}}$
in $\O_{R}$.

The next step is to remove a thin neighborhood of the complementary arc
$[a'b']_{\O_{R}}$ from $\O_{R}$ so as to keep $u$ connected to
$[b'a']_{\O_R}$ in the remaining domain. Let
\[
\Int\O':=\Int\O_{R}{}\Big\backslash{}\bigcup
_{x\in[a'b']_{\O_{R}}}\Int \BBsign^{\O_{R}}_{r'}(x)
\]
(more rigorously, $\Int\O'$ is the connected component of this set
containing $u$; see Figure~\ref{Fig: domain-truncated}). Assume that
$[b'a']_{\O_R}\cap\pa\O'=\varnothing$. Then there exist two vertices
$x_1,x_2\in[a'b']_{\O_R}$ such that the set
\[
E:=\Int\BBsign^{\O_R}_{r'}(x_1)\cup\Int
\BBsign^{\O_R}_{r'}(x_2)
\]
separates $u$ from $[b'a']_{\O_R}$ in $\O_R$ (here we use the fact that
$[b'a']_{\O_R}$ is a boundary arc of a simply connected domain $\O_R$
and not just a subset of $\pa\O_R$; see also Figure~\ref{Fig:
domain-truncated}). Then $\Int\BBsign^{\O_R}_{r'}(x_1)$ and $\Int
\BBsign
^{\O_R}_{r'}(x_2)$ have to share a face which implies
$\diam E \le\const\cdot\, r'$. Provided that $r'=r'(\sigma)>0$ is
chosen small enough (independently of $\O,a,b$ and $u$), we arrive at
the contradiction between the lower bound $\hm{{\O_{R}}}u{[b'a']_{\O
_{R}}}\ge\frac{1}{2}\sigma$ and the second estimate in Lemma~\ref
{Lemma:b0-weakBeurling} [recall that we have rescaled $\G$ so that
$\dd
\O{u}=1$].

Further, if we set $r:=\frac{1}{2}r'$, then
\[
\Int\O'\subset\Int\O_{R,r}.
\]
Since $[b'a']_{\O_R}\cap\pa\O'\ne\varnothing$ and all faces of
$\G$
intersecting $\pa\BBsign^\G_{R}(u)$ are at distance at least $R'+r'$
from $u$, we conclude that $[b'a']_{\O}\cap\pa\O'\ne\varnothing$: indeed,
once reaching the set $[b'a']_{\O_R}\setminus[b'a']_\O\subset\pa
\BBsign
^\G_{R}$ (this is the upper boundary arc on Figure~\ref{Fig:
domain-truncated}) inside of $\O'$, one can continue walking along
faces touching $\pa\BBsign^\G_{R}$ and reach the arc $[b'a']_\O$
staying inside $\O'$. Thus $u$ remains connected to the arc $[ba]_\O
\supset[b'a']_\O$ in the truncated domain $\O_{R,r}\supset\O'$.

Now let $L$ be a discrete path running from $u$ to $[ba]_\O$ inside
$\O
'$. We define a sequence of vertices $u=u_0,u_1,\ldots,u_n\in L\cap
\Int\O
'$ inductively by the following rule: $u_{k+1}\in\Int\O'$ is the first
vertex on $L$ after $u_k$ (when going toward $[ba]_\O$) which does not
belong to $\bigcup_{s\le k}\Int\BB\O{u_s}$. Let $u_l$ be the first of
those $u_k$ satisfying $\dd{\O}{u_l}<\n_0^{-1}\cdot r$ (if such a
vertex $u_l$ exists, otherwise we set $l:=n$),
and let $L'$ denote the portion of $L$ from $u$ to $u_l$.

It follows from Remark~\ref{Rem:LasymptDist} that $|u_k - u_s|\ge\n
_0^{-1}\cdot\frac{1}{3}\dd\O{u_s}\ge\frac{1}{3}\n_0^{-2}r$ for all
$0\le s<k\le l$. As all $u_k$ lie inside $\BBsign^\G_{R}(u)$, this
implies that $l$ is uniformly bounded. Applying Harnack's principle and
Proposition~\ref{Prop:ZuaZub/Zab=P} similarly to (\ref{xP[L]estimate}),
we arrive at
\[
\sum_{k=0}^l \P_\O^{a,b}
\bigl[ \BBsign_\O(u_k)\bigr]\le\const\cdot\,
\P_\O ^{a,b}\bigl[ \BBsign_\O(u) \bigr].
\]
If $l=n$ [which means $L=L'\subset\bigcup_{k=0}^l \Int\BB\O{u_k}$],
this immediately gives the estimate $\P[\RW{\O}ab\cap L\ne\es]\le
\const\cdot\,\P_\O^{a,b}[\BB{\O}{u}]$. Otherwise, our definition
of $\O
_{R,r}$ guarantees that if $x$ is the closest boundary vertex to $u_l$,
then $x\in(ba)_\O$ and $a,b\notin\BBsign^{\O}_{\dd\O{u_l}}(x)$.
Together with Lemma~\ref{Lemma:LuCrossing}, this yields
\begin{eqnarray*}
\P\bigl[ \RWsign_{\O}(a;b) \cap\bigl[L' \cup
\rL_\O(u_l)\bigr]\ne\es\bigr] &\le&\sum
_{k=0}^l \P_\O^{a,b}\bigl[
\BBsign_\O(u_k)\bigr]+ \P_\O^{a,b}
\bigl[\rL_{\O}(u_l)\bigr]
\\
&\le& \const \cdot\,
\P_\O^{a,b}\bigl[ \BBsign_\O(u)\bigr].
\end{eqnarray*}
Clearly, one can repeat the same arguments for the other boundary arc
$[ab]_\O$. We complete the proof by saying that, due to topological
reasons, $\RW{\O}ab$ should cross at least one of those two paths
(connecting $u$ to $[ba]_\O$ and $[ab]_\O$, resp.).
\end{pf}

The last ingredient of the proof of Theorem~\ref{Thm:Zfact} is the
following simple lemma:

\begin{lemma}
\label{Lemma:HMuabc} There exists a constant $\sigma_0>0$ such that,
for any simply
connected discrete domain $\O$ and three boundary points $a,b,c\in\pa
\O
$ listed counterclockwise, one can find a vertex $u\in\Int\O$ so
that all
$\hm{\O}u{[ab]_\O}$, $\hm{\O}u{[bc]_\O}$, $\hm{\O}u{[ca]_\O
}\ge\sigma_0$.
\end{lemma}

\begin{figure}

\includegraphics{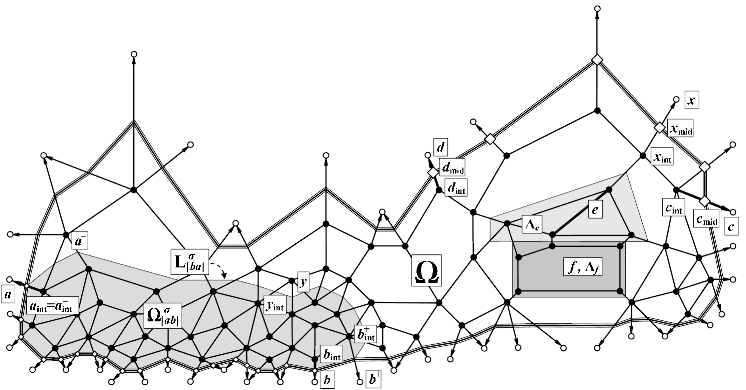}

\caption{An example of a simply connected discrete domain $\O$ and
its polygonal
representation (see Section \protect\ref{Sect:ExtremalLength}) with four
boundary points $a,b,c,d$ listed counterclockwise. Along the boundary
arcs $[ab]_\O$ and $[cd]_\O$, the midpoints $x_\mathrm{mid}$ of edges
$(x_\mathrm{int}x)$ are marked by small rhombii. In the left part of
$\O
$, the notation from the proof of Lemma \protect\ref{Lemma:HMuabc} is
shown: a
subdomain $\O_{[ab]}^{ \sigma}\ss\O$, the path $\rL_{[ba]}^{
\sigma}$
and the vertices $b^+,y,a^-$ on this path. In the right part of $\O$,
the notation from the proof of Proposition \protect\ref
{Prop:ELdiscr=ELcont} is
shown: the neighborhoods $\Lambda_e$, $\Lambda_f$ of an edge $e$ and a
face $f$, respectively.}
\label{Fig: domain-abcd+L}
\end{figure}

\begin{pf}
Recall that the ``no flat angles'' assumption (see Section~\ref
{SubSect:GraphAssumptions}) guarantees that all degrees of faces of $\G
$ are uniformly bounded. Let 
\[
\Int\O_{[ab]}^{ \sigma}:=\bigl\{u\in\Int\O\dvtx
\omega_\O\bigl(u;[ab]_\O\bigr) \ge\sigma\bigr\}.
\]
If $\sigma$ is chosen small enough (independently of $\O,a,b$ and $c$),
then $\O_{[ab]}^{ \sigma}$ contains all the vertices of faces touching
$[ab]_\O$ and hence is connected (which means that $\Int\O_{[ab]}^{
\sigma}$ is a connected subgraph of $\Gamma$). Moreover $\O_{[ab]}^{
\sigma}$ is always simply connected due to the maximum principle. Let
\[
L_{[ba]}^{ \sigma}:=\pa\O_{[ab]}^{ \sigma}
\setminus[ab]_\O =(ba)_{\O
_{[ab]}^{ \sigma}}=\bigl[b^+a^-
\bigr]_{\O_{[ab]}^{ \sigma}} ,
\]
where $b^+\in L_{[ba]}^{ \sigma}$ denotes the next vertex on $\pa\O
_{[ab]}^{ \sigma}$ after $b$, and $a^-\in L_{[ba]}^{ \sigma}$ is the
vertex just before $a$ when going along $\pa\O_{[ab]}^{ \sigma}$
counterclockwise; see Figure~\ref{Fig: domain-abcd+L}. For $y\in
L_{[ba]}^{ \sigma}$, let $y_\mathrm{int}\in\Int\O_{[ab]}^{ \sigma}$
be the corresponding inner vertex. Then, for all $y\in L_{[ba]}^{
\sigma}$, one has
\begin{eqnarray*}
\omega_\O\bigl(y_\mathrm{int};[bc]_\O\bigr) +
\omega_\O\bigl(y_\mathrm{int};[ca]_\O\bigr) & \ge&
\omega_\O\bigl(y_{\mathrm{int}};(ba)_\O\bigr) \ge \const
\cdot\, \omega_\O\bigl(y;(ba)_\O\bigr)
\\
&=&\const\cdot\,
\bigl(1- \omega_\O\bigl(y;[ab]_\O\bigr) \bigr)\ge\const
\cdot\,(1 - \s)
\end{eqnarray*}
since, by definition, $y\notin\Int\O_{[ab]}^{ \sigma}$ implies $\hm
{\O}{y}{[ab]_\O}<\s$. Further, for any two consecutive vertices
$y,y'\in L_{[ba]}^{ \sigma}$, the corresponding vertices $y_\mathrm
{int}$ and $y'_\mathrm{int}$ share a face of $\G$. This implies
\[
\omega_\O\bigl(y'_\mathrm{int};[bc]_\O
\bigr) \asymp \omega_\O\bigl(y_\mathrm {int};[bc]_\O
\bigr) \quad\mbox{and}\quad \omega_\O\bigl(y'_\mathrm{int};[ca]_\O
\bigr) \asymp \omega_\O\bigl(y_\mathrm {int};[ca]_\O
\bigr).
\]
On the other hand, $\hm\O{b^+_\mathrm{int}}{[bc]_\O}\ge\const$
and $\hm
\O{a^-_\mathrm{int}}{[ca]_\O}\ge\const$ due to the same argument (e.g.,
$b^+_\mathrm{int}$ shares a face with $b_\mathrm{int}$). Therefore,
observing $L_{[ba]}^{ \sigma}$ step by step, one can find $y\in
L_{[ba]}^{ \sigma}$ such that both
$\hm{\O}{y_\mathrm{int}}{[bc]_\O}$ and $\hm{\O}{y_\mathrm
{int}}{[ca]_\O
}$ are bounded below by some constant independent of $\O,a,b$ and $c$.
Let $u:=y_\mathrm{int}$. To complete the proof, note that $\hm{\O
}u{[ab]_\O}\ge\s$ as $u\in\Int\O_{[ab]}^{ \sigma}$.
\end{pf}

\begin{theorem}
\label{Thm:Zfact} Let $\O$ be a simply connected discrete domain and
boundary points $a,b,c\in\pa\O$ be listed
counterclockwise. Then, the following double-sided estimate is fulfilled:
%
\begin{equation}
\label{Zfact}
\ZRWsign_{\O}\bigl(a;[bc]_\O\bigr)
\asymp\biggl[\frac{\ZRW{\O}ab
\ZRW{\O
}ac}{\ZRW{\O}bc}
\biggr]^{{1} /{2}},
\end{equation}
with some uniform (i.e., independent of $\O,a,b,c$) constants.
\end{theorem}

\begin{pf} Due to Lemma~\ref{Lemma:HMuabc}, one can find an inner
vertex $u\in\Int\O$ such that
all $\hm{\O}u{[ab]_\O},\hm{\O}u{[bc]_\O},\hm{\O}u{[ca]_\O}\ge
\sigma_0$,
where the constant \mbox{$\sigma_0>0$} is independent of $\O,a,b$ and
$c$. Note that, for any $x\in
[bc]_\O$, one has
\[
\omega_\O\bigl(u;[ax]_\O\bigr) \ge \omega_\O
\bigl(u;[ab]_\O\bigr) \ge\sigma_0 \quad\mbox{and}\quad
\omega_\O\bigl(u;[xa]_\O\bigr) \ge \omega_\O
\bigl(u;[ca]_\O\bigr) \ge\sigma_0.
\]
Therefore, Propositions \ref{Prop:ZuaZub/Zab=P} and \ref
{Prop:RWintersectsBu} imply
\begin{eqnarray*}
{\ZRWsign_{\O}\bigl(a;[bc]_\O\bigr)}
&=&\sum_{x\in[bc]_\O}\ZRW{\O}ax
 \asymp \sum_{x\in[bc]_\O}\ZRW{\O}ua\ZRW{\O}ux
\\
&= &\ZRW{
\O}ua
\ZRWsign_{\O}\bigl(u;[bc]_\O\bigr)
\asymp\ZRW{\O}ua,
\end{eqnarray*}
where we have used $\ZRW{\O}u{[bc]_\O}\asymp\hm{\O}u{[bc]_\O
}\asymp1$.
Similarly,
\begin{eqnarray*}
\biggl[\frac{\ZRW{\O}ab \ZRW{\O}ac}{\ZRW{\O}bc} \biggr]^{ {1}/ {2}}&\asymp& \biggl[
\frac{\ZRW{\O}ua\ZRW{\O}ub\cdot
\ZRW{\O}ua\ZRW{\O}uc}{\ZRW{\O}ub\ZRW{\O}uc}\biggr]^{{1}/ {2}}\\
&=& \ZRW{\O}ua.
\end{eqnarray*}
Thus, both parts of (\ref{Zfact}) are uniformly comparable to $\ZRW
{\O}ua$.
\end{pf}

\section{Discrete cross-ratios}
\label{Sect:XYandZ} 

The main purpose of this section is to obtain a uniform double-sided
estimate (\ref{ZlogYestim}) relating discrete analogues of two
conformal invariants defined for a simply connected discrete domain $\O
$ with four marked boundary points $a,b,c,d$: discrete cross-ratio $\xY
\O abcd$ (see Definition~\ref{Def:CrossRatio}) and the total partition
function $\ZRW\O{[ab]_\O}{[cd]_\O}$ of random walks connecting two
opposite boundary arcs. Note that the cross-ratio $\xYsign_\O$ changes
to its reciprocal when replacing boundary arcs $[ab]_\O$ and $[cd]_\O$
by ``dual'' ones ($[bc]_\O$ and $[da]_\O$), while the corresponding
change of $\ZRWsign_\O$ is more sophisticated; see (\ref{ZlogYestim}).


Let two points $a,b$ (or, more generally, two disjoint arcs
$A=[a_1a_2]_\O$, $B=[b_1b_2]_\O$) on the boundary of a simply connected
discrete domain $\O$ be fixed. Then one can use the ratio $\ZRW\O xa /
\ZRW\O xb$ in order to ``track'' the position of $x$ with respect to
$a,b$. Being considered on $\pa\O$, this ratio has a monotonicity
property (see Lemma~\ref{Lemma:R-monotone} below), which allows one to
use it as a ``parametrization'' of $\pa\O$ between $A$ and $B$. Namely,
for $x\in\pa\O$, denote
\[
{ \RRsign_\O(x;A,B) }:=\frac{\ZRW\O{x}A}{\ZRW\O{x}B}.
\]
%

\begin{lemma}
\label{Lemma:R-monotone} Let $\O$ be a simply connected discrete domain
and $A=[a_1a_2]_\O$, \mbox{$B=[b_1b_2]_\O$} denote two
disjoint boundary arcs of $\O$. Then the ratio $\RR\O\ccdot AB$
decreases along the boundary arc
$[a_2b_1]_\O$ and increases along the boundary arc $[b_2a_1]_\O$.
\end{lemma}

\begin{remark}
In particular, if $A=\{a\}$ and $B=\{b\}$ are just single boundary
points, then $\RR\O\ccdot
ab$ attains its maximal and minimal values on $\pa\O$ at $a$ and $b$,
respectively, being
monotone on both boundary arcs $[ab]_\O$ and $[ba]_\O$.
\end{remark}

\begin{pf}
Similar to the proof of Remark~\ref{Rem:Z=Harm}(i), for any given
$t>0$, we define a discrete harmonic (in $\O$) function
\[
H_t(u):= %
\cases{ \ZRW\O{u}A-t\ZRW\O{u}B, &\quad $u\in\Int\O$,
\vspace*{2pt}\cr
\mu_u^{-1}\bigl(\1 _A(u)-t
\1_B(u)\bigr), &\quad $u\in\pa\O$.} %
\]
Note that, for any $x\in\pa\O$, one has
\[
\ZRW\O{x}A-t\ZRW\O{x}B=H_t(x)+\varpi_{xx_\mathrm
{int}}H_t(x_\mathrm{int}).
\]

For a given boundary point $x\in(a_2b_1]_\O$, let $t_x>0$ be chosen
so that
\mbox{$H_{t_x}(x_\mathrm{int})=0$} [if $x\in(a_2b_1)_\O$, this means
$\RR\O xAB=t_x$ as $H_{t_x}(x)=0$, while $\RR\O
{b_1}AB<t_{b_1}$].

The function $H_{t_x}$ is discrete harmonic in $\O$, vanishes on
$\pa\O\setminus(A\cup B)$, is strictly positive on $A$ and strictly
negative on $B$. Therefore, there exists
a nearest-neighbor path $\gamma_{xA}$ running from $x_\mathrm{int}$ to
$A$ such that
\mbox{$H_{t_x}\ge0$} along $\gamma_{xA}$. Due to the maximum
principle, this implies
$H_{t_x}(y_\mathrm{int})\ge0$ for all intermediate boundary points
$y\in[a_2x)_\O$. In other
words,
\begin{eqnarray}
\ZRW\O{y}A-t_x\ZRW\O{y}B= \mu_{a_2}^{-1}\1[y =
a_2]+\varpi _{yy_\mathrm
{int}}H_{t_x}(y_\mathrm{int})
\ge 0 \nonumber\\
\eqntext{\mbox{for all } y\in[a_2x)_\O.}
\end{eqnarray}
Thus $\RR\O y AB \ge t_x \ge\RR\O xAB$ for all $y\in[a_2x)_\O$, which
means that $\RR\O\ccdot AB$ decreases
along $[a_2b_1]_\O$. The proof for the other boundary arc $[b_2a_1]_\O$
is similar.
\end{pf}

\begin{definition}
\label{Def:CrossRatio} Let $\O$ be a simply connected discrete domain
and boundary points $a,b,c,d\in\pa\O$ be
listed counterclockwise. We define their \emph{discrete cross-ratios} by
\begin{eqnarray*}
\xXsign_{\O}(a,b;c,d)&:=&\biggl[\frac{\ZRW{\O}ac\cdot\ZRW{\O
}bd}{\ZRW{\O}ab\cdot
\ZRW{\O}cd}
\biggr]^{{1} /{2}};
\\
\xYsign_\O(a,b;c,d) &:=&\biggl[\frac{\ZRW{\O}ad\cdot\ZRW{\O
}bc}{\ZRW{\O}ab\cdot
\ZRW{\O}cd}
\biggr]^{{1}/ {2}}.
\end{eqnarray*}
\end{definition}

\begin{remark}
Since $a,b,c,d$ are listed counterclockwise, Lemma~\ref
{Lemma:R-monotone} implies
\begin{eqnarray*}
\xXsign_{\O}(a,b;c,d)=\biggl[\frac{\RR{\O}acb}{\RR{\O}dcb}\biggr]^
{{1}/{2}}& \le&1 \quad\mbox{and}\\
 \frac{\xX{\O}abcd}{\xY{\O}abcd}=\biggl[\frac{\RR{\O}acd}{\RR{\O
}bcd}
\biggr]^{{1}/ {2}}& \le&1.
\end{eqnarray*}
Note that the cross-ratio $\xX{\O}abcd$ admits the following
probabilistic interpretation:
\[
\bigl(\xXsign_\O(a,b;c,d)\bigr)^2= \P\bigl[
\RWsign_{\O}(a;d) \cap \RWsign_{\O}(b;c) \ne\es\bigr].
\]
Indeed, any random walks running from $a$ to $c$ and from $b$ to $d$ in
$\O$ have to intersect for topological reasons. Rearranging the tails
of those walks after they meet, it is easy to see that $\ZRW{\O
}ac\cdot
\ZRW{\O}bd$ can be rewritten as a partition function of pairs of random
walks running from $a$ to $d$ and from $b$ to $c$ in $\O$ that
intersect each other.
\end{remark}

We include the exponent $\frac{1}{2}$ in Definition~\ref
{Def:CrossRatio} for two
(clearly related) reasons: first, it simplifies several double-sided
estimates given below,
and second, it makes the notation closer to the standard continuous
setup. Indeed, the
continuous analogue of the partition function $\ZRW{\O}ab$ for the
upper half-plane $\H$ (up
to a multiplicative constant) is given by $(b - a)^{-2}$, so the
quantities $\xXsign_\O$ and
$\xYsign_\O$ introduced above are ``discrete versions in $\O$'' of the
usual cross-ratios
\[
x_{\H}(a,b;c,d):=\frac{(b - a)(d - c)}{(c - a)(d - b)} \quad\mbox{and}\quad
 y_{\H}(a,b;c,d):=
\frac{(b - a)(d - c)}{(d - a)(c - b)}.
\]
In the continuous setup, the following is fulfilled:
$(x_{\H}(a,b;c,d))^{-1} \equiv1+(y_{\H}(a,b;c,d))^{-1}$.
One clearly cannot hope that the same \emph{identity} remains valid on
the discrete level for
\emph{all} $\O$'s (even, say, if $\G$ is the standard square grid).
Nevertheless, below we
prove that the similar \emph{uniform double-sided estimate}
holds true for the discrete cross-ratios, with constants, in general,
depending on parameters fixed in assumptions (a)--(d) but \emph{not} on
the configuration $(\O;a,b,c,d)$ or the underlying graph $\G$ structure.

\begin{proposition}
\label{Prop:XwtXestim} Let $\O$ be a simply connected discrete domain
and $a,b,c,d\in\pa\O$ be
listed counterclockwise. Then, the following double-sided estimate
holds true:
%
\begin{equation}
\label{XwtXestim} \bigl( \xXsign_{\O}(a,b;c,d)\bigr)^{-1}
\asymp1+\bigl( \xYsign_\O(a,b;c,d) \bigr)^{-1},
\end{equation}
with some uniform (i.e., independent of $\O,a,b,c,d$) constants.
\end{proposition}

\begin{pf}
We apply factorization (\ref{Zfact}) to both sides of the trivial estimate
\[
\ZRWsign_{\O}\bigl(a;[bd]_\O\bigr)
\asymp
\ZRWsign_{\O}\bigl(a;[bc]_\O\bigr)
+
\ZRWsign_{\O}\bigl(a;[cd]_\O\bigr)
,
\]
which is almost an identity besides the term $\ZRW{\O}ac$, counted once
on the left-hand side
and twice on the right-hand side. It is easy to check that, dividing by
$[\ZRW{\O}ab\ZRW{\O}ac\ZRW{\O}ad]^{1/2}$, one obtains the following
double-sided estimate:
%
\begin{eqnarray}
\label{xXwtXestim} &&\biggl[\frac{1}{\ZRW{\O}ac\ZRW{\O}bd}\biggr]^{ {1}/ {2}}
\nonumber
\\[-8pt]
\\[-8pt]
\nonumber
&&\qquad\asymp
\biggl[\frac{1}{\ZRW{\O}ad\ZRW{\O}bc}\biggr]^{{1}/ {2}} + \biggl[\frac{1}{\ZRW{\O}ab\ZRW{\O}cd}
\biggr]^{{1} /{2}} ,
\end{eqnarray}
which is equivalent to (\ref{XwtXestim}).
\end{pf}

\begin{remark}
\label{Rem:XY(Z)comparable}
It immediately follows from (\ref{XwtXestim}) that $\xX{\O
}abcd\asymp\xY
{\O}abcd$, if $\xYsign_{\O}\le\const$ (which means that arcs
$[ab]_\O$
and $[cd]_\O$ are ``not too close'' in $\O$). Moreover, the next
Proposition shows that, in this case, $\ZRW{\O}{[ab]_\O}{[cd]_\O
}\asymp
\xY{\O}abcd$ as well, since $\ZRWsign_\O$ is always squeezed (up to
multiplicative constants) by $\xXsign_\O$ and $\xYsign_\O$.
\end{remark}

\begin{proposition}
\label{Prop:ZabcdXwtXestim} Let $\O$ be a simply connected discrete
domain and boundary points
$a,b,c,d\in\pa\O$ be listed counterclockwise. Then the following
estimates are
fulfilled:
%
\begin{equation}
\label{ZabcdXwtXestim} \const\cdot\, \xXsign_\O(a,b;c,d)\le
\ZRWsign_{\O}\bigl([ab]_\O;[cd]_\O\bigr)
\le\const\cdot\, \xYsign_\O(a,b;c,d)
,
\end{equation}
with some uniform (i.e., independent of $\O,a,b,c,d$) constants.
\end{proposition}

\begin{pf}
Due to Theorem~\ref{Thm:Zfact}, one has
\begin{eqnarray*}
\ZRWsign_{\O}\bigl([ab]_\O;[cd]_\O\bigr)
&=&\sum
_{x\in[ab]_\O}
\ZRWsign_{\O}\bigl(x;[cd]_\O\bigr)
\\
& \asymp&
\frac{1}{(\ZRW{\O}cd)^{{1}/{2}}}\sum_{x\in[ab]_\O}\bigl(\ZRW{\O
}xc\bigr)^{ {1} /{2}}\bigl(\ZRW{\O}xd\bigr)^{{1}/ {2}}.
\end{eqnarray*}
It follows from Lemma~\ref{Lemma:R-monotone} that, for any $x\in
[ab]_\O$,
\begin{eqnarray*}
\bigl(\ZRW{\O}xc\bigr)^{{1} /{2}}\bigl(\ZRW{\O}xd\bigr)^{{1} /{2}}&=&
\frac{(\ZRW{\O}xc)^{{1}/{2}}}{(\ZRW{\O}xd)^{{1}/{2}}}\cdot \ZRW{\O }xd\\
&\ge &
\frac{(\ZRW{\O}ac)^{{1}/{2}}}{(\ZRW{\O}ad)^{{1}/{2}}}\cdot \ZRW{\O }xd.
\end{eqnarray*}
Therefore, summing and applying Theorem~\ref{Thm:Zfact} once more, one obtains
\begin{eqnarray*}
\ZRWsign_{\O}\bigl([ab]_\O;[cd]_\O\bigr)
&\ge&\const\cdot\,
\frac{(\ZRW{\O
}ac)^{{1}/{2}}}{
(\ZRW{\O}cd)^{{1}/{2}}(\ZRW{\O}ad)^{{1}/{2}}}\cdot
\ZRWsign_{\O}\bigl([ab]_\O;d\bigr)
\\
& \asymp&
\frac{(\ZRW{\O}ac)^{{1}/{2}}(\ZRW{\O}bd)^{{1}/{2}}}{
(\ZRW{\O}cd)^{{1}/{2}}(\ZRW{\O}ab)^{{1}/{2}}}= \xXsign_{\O}(a,b;c,d).
\end{eqnarray*}
On the other hand, Cauchy's inequality (and Theorem~\ref{Thm:Zfact}
again) gives
\begin{eqnarray*}
\bigl(
\ZRWsign_{\O}\bigl([ab]_\O;[cd]_\O\bigr)\bigr)^2 &
\le&\const\cdot\,\frac{\ZRW{\O
}{[ab]_\O
}c\ZRW{\O}{[ab]_\O}d}{\ZRW{\O}cd}
\\
&\asymp& \frac{(\ZRW{\O}ac\ZRW{\O}bc\ZRW{\O}ad\ZRW{\O}bd)^{
{1}/{2}}}{\ZRW{\O
}cd\ZRW{\O}ab}
\\
& =&
\xXsign_{\O}(a,b;c,d) \xYsign_\O(a,b;c,d) \le\bigl(
\xYsign_\O(a,b;c,d) \bigr)^2.
\end{eqnarray*}
\upqed\end{pf}

\begin{theorem}
\label{Thm:ZlogYestim} Let $\O$ be a simply connected discrete domain
and boundary points $a,b,c,d\in\pa\O$ be
listed counterclockwise. Then the following double-sided estimate holds true:
%
\begin{equation}
\label{ZlogYestim}
\ZRWsign_{\O}\bigl([ab]_\O;[cd]_\O\bigr)
\asymp\log
\bigl(1 + \xYsign_\O(a,b;c,d) \bigr) ,
\end{equation}
with some uniform (i.e., independent of $\O,a,b,c,d$) constants.
\end{theorem}

\begin{pf}
Denote $\xYsign_\O:=\xY{\O}abcd$, $\xXsign_\O:=\xX{\O}abcd$,
and let a
constant $M$ be chosen big enough [independently of $(\O;a,b,c,d)$]. If
$\xYsign_\O\le M$, Propositions \ref{Prop:XwtXestim}, \ref
{Prop:ZabcdXwtXestim} imply
%
\begin{eqnarray}
\label{xZlogYestim} %
\qquad\ZRWsign_{\O}\bigl([ab];[cd]\bigr)
 &\ge&\const
\cdot\,\xXsign_\O\asymp{(1 + \xYsign _\O)^{-1}{
\xYsign_\O}} \ge(1 + M)^{-1} \cdot{\log(1 +
\xYsign_\O )},
\nonumber
\\[-8pt]
\\[-8pt]
\nonumber
\ZRWsign_{\O}\bigl([ab];[cd]\bigr)
&\le&\const\cdot\,\xYsign_\O\le\const\cdot\, M\bigl[{\log (1 + M)}\bigr]^{-1}
\cdot\log(1 + \xYsign_\O)
\end{eqnarray}
(with constants independent of $M$). Thus, without loss of generality,
we can assume that $\xYsign_\O\ge M$
(i.e., $[ab]_\O$ and $[cd]_\O$ are ``very close'' to each other in
$\O
$). Let
\[
\RRsign_\O(x):= \RRsign_{\O}(x;c,d) =\frac{\ZRW{\O}xc}{\ZRW{\O}xd} ,\qquad
x\in [ab]_\O.
\]
Due to Lemma~\ref{Lemma:R-monotone}, $\RRsign_\O$ increases on
$[ab]_\O
$. Moreover, it follows
from Proposition~\ref{Prop:XwtXestim} [or directly from (\ref
{xXwtXestim})] that
\begin{eqnarray*}
\biggl[\frac{\RRsign_{\O}(b)}{\RRsign_{\O}(a)}\biggr]^{{1} /{2}}&=& \biggl[
\frac
{\ZRW{\O}bc\ZRW{\O}ad}{
\ZRW{\O}bd\ZRW{\O}ac}\biggr]^{{1}/ {2}}\asymp1+ \biggl[\frac{\ZRW
{\O}bc\ZRW
{\O}ad}{
\ZRW{\O}ab\ZRW{\O}cd}
\biggr]^{{1} /{2}}\\
&=&1+\xYsign_\O\asymp \xYsign_\O.
\end{eqnarray*}
%
As any two consecutive boundary vertices $x,x'\in[ab]_\O$ belong to
the same face of $\G$, one has
$\ZRW\O{x}c\asymp\ZRW\O{x'}c$, $\ZRW\O{x}d\asymp\ZRW\O{x'}d$ and
\[
1\le\frac{\RRsign_\O(x')}{\RRsign_\O(x)}\le\const\!.
\]
Therefore, provided that $\xYsign_\O\ge M$ is big enough, one can find
a number $n\asymp\log\xYsign_\O$ and a sequence of
boundary points $a=a_0,a_1,\ldots,a_n=b$ such that
\[
4\le\frac{\RRsign_\O(a_{k+1})}{\RRsign_\O(a_k)}\le\const
\]
for all $k=0,\ldots,n - 1$. This can be easily rewritten as
\[
\const\le\biggl[\frac{\RRsign_\O(a_k)}{\RRsign_\O(a_{k+1})} \biggr]^{ {1}/ {2}}=
\xXsign_\O(a_k,a_{k+1};c,d) \le
\frac{1}{2} ,
\]
or, due to Proposition~\ref{Prop:XwtXestim}, as
$\xY\O{a_k}{a_{k+1}}cd \asymp1$. Hence if the constant $M$ was chosen
big enough, estimate (\ref{xZlogYestim}) implies
\[
\ZRWsign_{\O}\bigl([a_ka_{k+1}]_\O;[cd]_\O\bigr)
\asymp1
\]
for all $k=0,\ldots,n - 1$. This easily gives
%
\begin{equation}
\label{xxZlogYestim}
\ZRWsign_{\O}\bigl([ab]_\O;[cd]_\O\bigr)
\asymp\sum
_{k=0}^{n-1}
\ZRWsign_{\O}\bigl([a_ka_{k+1}]_\O;[cd]_\O\bigr)
\asymp n\asymp\log\xYsign_\O.
\end{equation}
Combining estimate (\ref{xZlogYestim}) with $\xYsign_\O\le M$ and
(\ref
{xxZlogYestim}) with $\xYsign_\O\ge M$, one arrives at~(\ref{ZlogYestim}).
\end{pf}

\section{Surgery technique}
\label{Sect:Surgery} 

The main purpose of this section is to illustrate how the tools
developed above can be used to construct cross-cuts of a simply
connected discrete domain $\O$ having some nice ``separation''
properties, without any reference to the actual geometry of $\O$. The
main result is Theorem~\ref{Thm:Separator} which claims the existence
of these ``separators.'' In Proposition~\ref{Prop:InclSeparators}, we
also give some simple monotonicity properties of such cross-cuts.

More precisely, let $A=[a_1a_2]_\O$ and $B=[b_1b_2]_\O$ be two disjoint
boundary arcs of a simply connected $\O$. We are interested in the
following question: is it possible to cut $\O$ along some cross-cut $L$
into two simply connected parts $\O_A,\O_B$, one containing $A$ and
the other
containing $B$, so that
%
\begin{equation}
\label{ZABasympZAG.ZGB} \ZRW\O AB \asymp\ZRW{\O_A} AL \ZRW{\O_B}
LB ?
\end{equation}
Moreover, we are interested not only in a \emph{single} cross-cut $L$,
but rather in a \emph{family} $L_k=L_A^B[k]$ such that, in addition to
factorization
(\ref{ZABasympZAG.ZGB}), one has
%
\begin{equation}
\label{ZAG/ZGB=k} \ZRW{\O_A} A{\rL_k} / \ZRW{
\O_B} {\rL_k} B \asymp k.
\end{equation}

Note that both $\ZRW{\O_A} A{\rL_k},\ZRW{\O_B}{\rL_k} B\ge\ZRW
\O AB$.
Thus (\ref{ZABasympZAG.ZGB}) certainly fails if $\ZRW\O AB\gg1$. For
a similar reason, one cannot hope for
(\ref{ZAG/ZGB=k}) if $k\ll\ZRW\O AB$ or $k\gg(\ZRW\O AB)^{-1}$.
However, being motivated by the continuous setup, one certainly hopes
for the positive answer in all other situations, and indeed,
Theorem~\ref{Thm:Separator} given below claims the existence of a
``separator'' $\rL_A^B[k]$ and provides a natural construction of this
slit for any given $\O,A,B$ and $k$.

Namely, let discrete domains $\O_A^B[k]$ and $\O_B^A(k^{-1})$ be
defined by
\begin{eqnarray*}
\Int\O_A^B[k]&:=&\biggl\{u\in\Int\O\dvtx\frac{\ZRW\O{u}A}{\ZRW\O
{u}B}
\ge k \biggr\},
\\
\Int\O_B^A\bigl(k^{-1}\bigr)&:=&\biggl\{u\in\Int
\O\dvtx\frac{\ZRW\O
{u}B}{\ZRW\O{u}A}> k^{-1}\biggr\}
\end{eqnarray*}
(we use square and round brackets to abbreviate $\ge$ and $>$
inequalities, resp.). Below we always work with $k$'s which are not
extremely big or extremely small, so that $\Int\O_A^B[k]$ contains all
vertices of faces touching $A$, while $\Int\O_B^A(k^{-1})$ contains
all vertices
near $B$. Then both $\O_A^B[k]$ and $\O_B^A(k^{-1})$ are connected and
simply connected (due to the maximum principle applied to the function
$\ZRW\O{\ccdot}A-k\ZRW\O{\ccdot}B$). Further, we denote the set of
edges 
\[
L_A^B[k]=L_B^A
\bigl(k^{-1}\bigr):=\bigl\{(u_Au_B)\in
E^\O_\mathrm{int}\dvtx u_A\in \Int\O
_A^B[k], u_B\in\Int\O_B^A
\bigl(k^{-1}\bigr)\bigr\};
\]
see Figure~\ref{Fig: domain+slit}(A). According to our conventions
concerning the boundary of a discrete domain, this set can be
interpreted as a part of $\pa\O_A^B[k]$, as well as a part of $\pa\O
_B^A(k^{-1})$.

\begin{figure}

\includegraphics{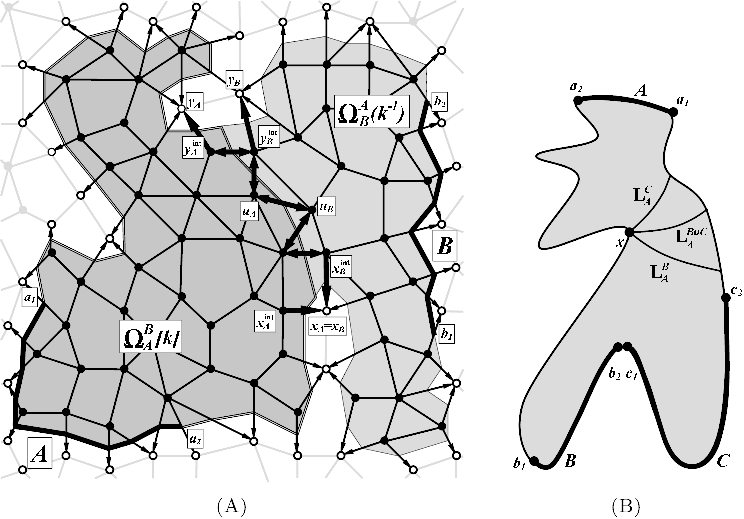}

\caption{\textup{(A)} A simply connected discrete
domain split into two parts, $\O_A^B[k]$ and $\O_B^A(k^{-1})$,
according to the ratio of harmonic measures of two marked boundary
arcs, $A=[a_1a_2]_\O$ and $B=[b_1b_2]_\O$. All edges $(u_Au_B)$ that
cross the slit $L_A^B[k]$ are marked, as well as four boundary edges
$(x_A^\mathrm{int}x_A),(x_B^\mathrm{int}x_B),(y_B^\mathrm
{int}y_B),(y_A^\mathrm{int}y_A)\in\pa\O$ neighboring $L_A^B[k]$.
\mbox{\textup{(B)} Notation} used in Proposition \protect\ref
{Prop:InclSeparators} and
schematic drawing of the monotonicity property $\O_A^C[x]\ss\O
_A^{B\cup
C}[x]\ss\O_A^B[x]$ for $x\in(a_2b_1)$.}
\label{Fig: domain+slit}
\end{figure}

%
%
%

\begin{theorem}
\label{Thm:Separator} Let $\O$ be a simply connected discrete domain,
$A,B\ss\pa\O$ be two
disjoint boundary arcs, $\ZRWsign:=\ZRW\O AB$ and $k>0$ be chosen so
that both
$\O_A:=\O_A^B[k]$ and $\O_B:=\O_B^A(k^{-1})$ are connected (i.e.,
$\O
_A$ contains all inner
vertices around $A$ while $\O_B$ contains all inner vertices around
$B$). Then:
\begin{longlist}[(ii)]
\item[(i)] for any fixed (big) constant $K\ge1$, the
following is fulfilled:
if $\ZRWsign\le K$ and $K^{-1} \le k\le K$, then the cross-cut
$L_k:=\rL_A^B[k]$
satisfies conditions (\ref{ZABasympZAG.ZGB}), (\ref{ZAG/ZGB=k}), with
constants depending
on $K$ but independent of $\O,A,B,k$;

\item[(ii)] there exists a (small) constant $\k_0>0$ such that the
following is
fulfilled: if $\ZRWsign\le\k_0$ and $\k_0^{-1}\ZRWsign\le k\le\k
_0\ZRWsign^{-1}$, then the cross-cut $L_k$ satisfies conditions (\ref
{ZABasympZAG.ZGB}), (\ref{ZAG/ZGB=k}) with some uniform constants.
Moreover, in this case, both $\O_A$ and $\O_B$ are always connected.
\end{longlist}
\end{theorem}

\begin{pf}
Since $\ZRW\O{u_A}\ccdot\asymp\ZRW\O{u_B}\ccdot$, it is clear that
%
\begin{equation}
\label{HmA/HmB=k} \frac{\ZRW\O{u_A}A}{\ZRW\O{u_A}B}\asymp\frac
{\ZRW\O
{u_B}A}{\ZRW\O{u_B}B}\asymp k \qquad\mbox{for all } u=(u_Au_B)\in\rL_k.
\end{equation}

Let $\pa\O_A\cap\pa\O=[y_Ax_A]_\O$ and $\pa\O_B\cap\pa\O
=[x_By_B]_\O$
[see Figure~\ref{Fig: domain+slit}(A)], and let
%
\begin{equation}
\label{ZaZbDef} \ZRWsign_A:=
\ZRWsign_{\O}\bigl(A;[x_By_B]_\O\bigr)
,\qquad
\ZRWsign_B:=
\ZRWsign_{\O}\bigl(B;[y_Ax_A]_\O\bigr),
\end{equation}
where these partition functions are considered \emph{in the original
domain $\O$}.
Then
\begin{eqnarray*}
\ZRW{\O_A} {A} {\rL_k}&=&\sum
_{u\in\rL_k}\ZRW{\O_A} {A} {u_B}\asymp \sum
_{u\in
\rL_k}
\ZRWsign_{\O_A}(A;u_B)
\ZRWsign_{\O}\bigl(u_B;\pa\O\bigr)
\\
&=&\sum_{u\in\rL_k}
\ZRW{\O_A} {A} {u_B}\cdot\bigl(
\ZRWsign_{\O}\bigl(u_B;[x_By_B]_\O\bigr)
+
\ZRWsign_{\O}\bigl(u_B;[y_Ax_A]_\O\bigr)
\bigr)
\end{eqnarray*}
since $\ZRW{\O}{u_B}{\pa\O}\asymp1$ for any $u_B\in\Int\O$.
Note that
the sum of first terms can be rewritten as
\[
\sum_{u\in\rL_k}\ZRW{\O_A} {A}
{u_B}
\ZRWsign_{\O}\bigl(u_B;[x_By_B]_\O\bigr)
\asymp
\ZRWsign_{\O}\bigl(A;[x_By_B]_\O\bigr)
=
\ZRWsign_A.
\]
Indeed, each random walk path running from $A$ to $[x_By_B]_\O$ inside
$\O$ should pass through $\rL_k$ for
topological reasons, so denoting by $u$ the \emph{first} crossing, one
obtains the result.
Similarly, the second sum is comparable to the total partition
functions of those random walks,
which start from $A$, cross $\rL_k$ (possibly many times) and finish
back at $[y_Ax_A]_\O$. Denoting by
$v$ the \emph{last} crossing of $\rL_k$ and using (\ref{HmA/HmB=k}),
one obtains
\begin{eqnarray*}
\sum_{u\in\rL_k}\ZRW{\O_A} {A}
{u_B}
\ZRWsign_{\O}\bigl(u_B;[y_Ax_A]_\O\bigr)
&\asymp &\sum_{v\in\rL_k}\ZRW{\O} {A} {v_B}
\ZRWsign_{\O_A}\bigl(v_B;[y_Ax_A]_\O\bigr)
\\
&\asymp & k\sum_{v\in\rL_k}
\ZRWsign_{\O}(B;v_B)
\ZRWsign_{\O_A}\bigl(v_B;[y_Ax_A]_\O\bigr)
\\
&\asymp & k\ZRWsign_B ,
\end{eqnarray*}
since each random walk path running from $B$ to $[y_Ax_A]_\O$ inside
$\O
$ should cross~$\rL_k$.
Thus we arrive at the double-sided estimates
\[
\ZRW{\O_A} {A} {\rL_k}\asymp\ZRWsign_A+k
\ZRWsign_B ,
\]
and similarly, $\ZRW{\O_B}{\rL_k}{B}\asymp k^{-1}\ZRWsign
_A+\ZRWsign
_B$. Therefore, it is sufficient
to prove that
%
\begin{equation}
\label{ZaZbSufficient} \ZRWsign_A/\ZRWsign_B\asymp k \quad\mbox{and}\quad
\ZRWsign_A\ZRWsign_B\asymp\ZRWsign.
\end{equation}

It directly follows from (\ref{HmA/HmB=k}) that
%
\begin{equation}
\label{ZxA/ZxB=ZyA/ZyB=k} \frac{\ZRW\O{x}A}{\ZRW\O{x}B}\asymp k\asymp \frac{\ZRW\O{y}A}{\ZRW\O{y}B}
\end{equation}
(here and below we omit subscripts of $x $ and $y$, all the claims hold
true for both $x=x_A,x_B$ and, similarly, $y=y_A,y_B$, since the values of
$\ZRW\O{x_A}{\cdot}$ and $\ZRW\O{x_B}{\cdot}$ are uniformly
comparable). Let $A=[a_1a_2]_\O$, $B=[b_1b_2]_\O$ and denote
\begin{eqnarray*}
\xYsign_A&:=& \xYsign_\O(a_1,a_2;x,y)
, \qquad \xYsign_B:= \xYsign_\O(b_1,b_2;y,x)
,
\\
\xXsign_A&:=& \xXsign_\O(a_1,a_2;x,y),
\qquad\xXsign_B:= \xXsign_\O(b_1,b_2;y,x),
\end{eqnarray*}
where all discrete cross-ratios are considered \emph{in the original
domain $\O$}. Using
Theorem~\ref{Thm:Zfact} and (\ref{ZxA/ZxB=ZyA/ZyB=k}), it is easy to
check that
%
\begin{equation}
\label{YaXa/YbXb=k} \biggl[\frac{\xYsign_A\xXsign_A}{\xYsign
_B\xXsign_B} \biggr]^{{1}/ {2}}\asymp
\biggl[\frac{\ZRW\O{x}A\ZRW\O{y}A}{\ZRW\O{x}B\ZRW\O{y}B} \biggr]^{ {1} /{2}}\asymp k.
\end{equation}
The rest of the proof is divided into three steps:
\begin{itemize}
\item First, we prove (\ref{ZaZbSufficient}) assuming that both
$\ZRWsign_A,\ZRWsign_B$ are bounded above by some absolute constant
(roughly speaking, this means that $x$ and $y$ are ``not too close'' to
both $A,B$). In some sense this is the most conceptual step, based on
discrete cross-ratios techniques from Section~\ref{Sect:XYandZ}.
\item Second, we use discrete cross-ratios once again to show that,
indeed, one has $\ZRWsign_A,\ZRWsign_B\le\const$ if $k\asymp1$ [in
particular, this implies (i)].
\item Finally, we analyze general case in (ii) by starting with $k=1$
and then increasing it until $\ZRWsign_A$ becomes $\asymp1$, which, as
we show, cannot happen before $k\asymp\ZRWsign^{-1}$.
\end{itemize}

\emph{Step} (1) \emph{The proof of} (\ref{ZaZbSufficient}) \emph{under
assumption} $\ZRWsign_A,\ZRWsign_B\le\const$. In this case
Theorem~\ref{Thm:ZlogYestim} guarantees that $\xYsign_A,\xYsign_B\le
\const$ as well, and Remark~\ref{Rem:XY(Z)comparable} says that
\[
\ZRWsign_A\asymp[\xYsign_A\xXsign_A]^{1/2}\quad
\mbox{and}\quad \ZRWsign _B\asymp[\xYsign_B
\xXsign_B]^{1/2}.
\]
Therefore, (\ref{YaXa/YbXb=k}) immediately gives the first part of
(\ref
{ZaZbSufficient}). Moreover, one has $\xXsign_A\asymp\xYsign_A$ and
$\xXsign_B\asymp\xYsign_B$, which is equivalent to saying that
%
\begin{equation}
\label{Zxaa/Zyaa=Zxbb/Zybb} \frac{\ZRW\O x{a_1}}{\ZRW\O y{a_1}}\asymp\frac{\ZRW\O
x{a_2}}{\ZRW\O
y{a_2}} \quad\mbox{and}\quad
\frac{\ZRW\O x{b_1}}{\ZRW\O
y{b_1}}\asymp\frac{\ZRW\O x{b_2}}{\ZRW\O y{b_2}}.
\end{equation}
In addition, Theorem~\ref{Thm:Zfact} applied to (\ref
{ZxA/ZxB=ZyA/ZyB=k}) gives
\[
\frac{\ZRW\O{x}{a_1}\ZRW\O{x}{a_2}}{\ZRW\O{y}{a_1}\ZRW\O
{y}{a_2}}\asymp \frac{\ZRW\O{x}{b_1}\ZRW\O{x}{b_2}}{\ZRW\O{y}{b_1}\ZRW\O{y}{b_2}},
\]
thus upgrading (\ref{Zxaa/Zyaa=Zxbb/Zybb}) to
%
\begin{equation}
\label{Zxaabb/Zyaabb=} \frac{\ZRW\O x{a_1}}{\ZRW\O y{a_1}}\asymp \frac
{\ZRW\O x{a_2}}{\ZRW\O
y{a_2}}\asymp
\frac{\ZRW\O x{b_1}}{\ZRW\O y{b_1}}\asymp\frac{\ZRW
\O
x{b_2}}{\ZRW\O y{b_2}}.
\end{equation}
%

As $\ZRWsign\le\const$, we also have $\ZRWsign\asymp\xX\O
{a_1}{a_2}{b_1}{b_2}$. Rearranging
factors, one obtains
\[
\frac{\ZRWsign_A\ZRWsign_B}{\ZRWsign}\asymp \frac{[\xYsign_A\xXsign_A\xYsign_B\xXsign_B]
^{{1}/{2}}}{
\xXsign_\O(a_1,a_2;b_1,b_2)}\asymp [\RRsign_1
\RRsign_2]^{{1}/ {4}},
\]
where
\[
\RRsign_j:=\frac{\ZRW\O{a_j}{x}\ZRW\O{x}{b_j}\ZRW\O{b_j}{y}\ZRW
\O
{y}{a_j}}{(\ZRW\O{a_j}{b_j}\ZRW\O{x}{y})^2}.
\]
Finally, it follows from (\ref{Zxaabb/Zyaabb=}) that $\xY\O
{a_j}{x}{b_j}{y}\asymp1$. 
Due to Proposition~\ref{Prop:XwtXestim}, this also implies
$\xX\O{a_j}{x}{b_j}{y}\asymp1$ and, similarly, $\xX\O
{x}{b_j}{y}{a_j}\asymp1$. Therefore,
\[
R_j= \bigl[ \xXsign_\O(a_j,x;b_j,y)
\xXsign_\O(x,b_j;y,a_j)
\bigr]^{-
{1}/{2}}\asymp1,
\]
that is, ${\ZRWsign_A\ZRWsign_B}\asymp{\ZRWsign}$ [which is the second
part of (\ref{ZaZbSufficient})], and we are done.

\emph{Step} 2. \emph{Proof of} $\ZRWsign_A,\ZRWsign_B\le\const$,
\emph{if} $k\asymp1$.
In this case, Proposition~\ref{Prop:XwtXestim} and (\ref
{YaXa/YbXb=k}) give
\[
{\xYsign_A^2}({1 + \xYsign_A})^{-1}
\asymp\xYsign_A\xXsign_A\asymp \xYsign_B
\xXsign_B\asymp {\xYsign_B^2}({1 +
\xYsign_B})^{-1}.
\]
Thus if, say, $\xYsign_A\le\const$, then $\xYsign_B\le\const$ as
well, and
$\ZRWsign_A,\ZRWsign_B\le\const$ due to Theorem~\ref{Thm:ZlogYestim}.
Hence, without
loss of generality, we may assume that \emph{both} $\xYsign_A,\xYsign_B$
\emph{are bounded away from zero}, which is equivalent to saying that both
$\xXsign_A,\xXsign_B\asymp1$, that is,
\begin{eqnarray*}
\ZRW\O{x} {a_1}\ZRW\O{y} {a_2}&\asymp&\ZRW
\O{a_1} {a_2}\ZRW\O{x} {y} ,
\\
\ZRW\O{x} {b_2}\ZRW\O{y} {b_1}&\asymp&\ZRW
\O{b_1} {b_2}\ZRW\O{x} {y}.
\end{eqnarray*}
Using Theorem~\ref{Thm:Zfact} and (\ref{ZxA/ZxB=ZyA/ZyB=k}), we
obtain
\begin{eqnarray*}
\frac{\ZRW\O{x}{a_2}}{\ZRW\O{y}{a_2}}&\asymp& \frac{\ZRW\O{x}{a_2}\ZRW\O{x}{a_1}}{\ZRW\O{a_1}{a_2}\ZRW\O
{x}{y}}\asymp \frac{(\ZRW\O{x}{A})^2}{\ZRW\O{x}{y}}
\\
&
\asymp&\frac{(\ZRW\O
{x}{B})^2}{\ZRW\O{x}{y}} \asymp \frac{\ZRW\O{x}{b_1}\ZRW\O{x}{b_2}}{\ZRW\O{b_1}{b_2}\ZRW\O
{x}{y}}\asymp \frac{\ZRW\O{x}{b_1}}{\ZRW\O{y}{b_1}} ,
\end{eqnarray*}
which means $\xY{\O}{a_2}x{b_1}y\asymp1$. Then, Remark~\ref
{Rem:XY(Z)comparable} applied to the quadrilateral $(\O;a_2,x;b_1,y)$
gives $1\asymp\xX{\O}{a_2}x{b_1}y\asymp\xY{\O}{a_2}x{b_1}y$ which can
be rewritten as
\[
{\ZRW\O{x} {a_2}} {\ZRW\O{y} {b_1}}\asymp{\ZRW\O{x} {y}}
{\ZRW\O {a_2} {b_1}}\asymp {\ZRW\O{x} {b_1}}
{\ZRW\O{y} {a_2}}.
\]
Similarly, one has
\[
{\ZRW\O{x} {a_1}} {\ZRW\O{y} {b_2}}\asymp{\ZRW\O{x} {y}}
{\ZRW\O {a_1} {b_2}}\asymp {\ZRW\O{x} {b_2}}
{\ZRW\O{y} {a_1}}.
\]
Then, using $\xXsign_A,\xXsign_B\asymp1$ and rearranging factors, one
arrives at
\[
\xYsign_A\xYsign_B\asymp \xYsign_A
\xXsign_A\xYsign_B\xXsign_B\asymp
\frac{\ZRW\O
{a_1}{b_2}\ZRW\O
{a_2}{b_1}}{
\ZRW\O{a_1}{a_2}\ZRW\O{b_2}{b_1}}= \xYsign_\O(a_1,a_2;b_1,b_2)
.
\]
As $\ZRWsign$ is bounded above, Theorem~\ref{Thm:ZlogYestim} ensures
that $\xY\O{a_1}{a_2}{b_1}{b_2}\le\const$.
Taking into account $\xYsign_A,\xYsign_B\ge\const$, we get $\xYsign
_A,\xYsign_B\asymp1$, and so $\ZRWsign_A,\ZRWsign_B\asymp1$.

\emph{Step} 3. \emph{Proof of the general case in} (ii). Let $\ZRWsign
_A(k)$ and
$\ZRWsign_B(k)$ be defined by~(\ref{ZaZbDef}) for a given $k$. Note that
$\ZRWsign_A(k)$, $\ZRWsign_B(k)$ are piecewise-constant left-continuous
functions of $k$ which jump no more than
by some constant factor $\varpi_0^{-2}>1$ [see assumption (a) in
Section~\ref{SubSect:GraphAssumptions}], when domain $\O_A^B[k]$ [and,
simultaneously, $\O_B^A(k^{-1})$] changes.

We will fix $\k_0$ at the end of the proof, but in any case it will be
less than $1$.
Since $\ZRWsign\le1$, step 2 ensures that $\ZRWsign_A(1),\ZRWsign
_B(1)\le\zeta_0$ for some
absolute constant $\zeta_0$ [actually, $\ZRWsign_A(1)$ and $\ZRWsign_A(1)$
are much smaller, being of order
$\ZRWsign^{1/2}$]. Now let us start to increase the parameter $k$.
Since $\O_A^B[k']\subset\O_A^B[k]$ for $k'>k$, the partition function
$\ZRWsign_A(k)$ increases, while $\ZRWsign_B(k)$ decreases. Let
\[
k_\mathrm{max}:=\max\bigl\{k\ge1\dvtx\ZRWsign_A(k)\le
\zeta_0\bigr\}.
\]
Due to step 1, there exists a positive constant $c_0\le1$ such that
the following is fulfilled:
\[
c_0 k\le\ZRWsign_A(k)/\ZRWsign_B(k)\le
c_0^{-1} k \quad\mbox{and}\quad c_0\ZRWsign\le
\ZRWsign_A(k)\ZRWsign_B(k)\le c_0^{-1}
\ZRWsign
\]
for any $k\in[1,k_\mathrm{max}]$. Moreover, one has $\ZRWsign
_A(k_\mathrm{max})\ge\varpi_0^{2}\zeta_0$,
since the function $\ZRWsign_A(\ccdot)$ cannot jump too much at the
point $k_\mathrm{max}$.
Therefore, we obtain the estimate
\[
k_\mathrm{max}\ge c_0\cdot\frac{\ZRWsign_A(k_\mathrm
{max})}{\ZRWsign
_B(k_\mathrm{max})}\ge
c_0^2\cdot\frac{(\ZRWsign_A(k_\mathrm{max}))^2}{\ZRWsign}\ge \varpi
_0^{4}\zeta_0^2c_0^2
\cdot\ZRWsign^{-1}.
\]
Thus, if $\k_0\le\min\{1,\varpi_0^{4}\zeta_0^2c_0^2\}$, then (ii) holds
true for all
\mbox{$k\in[1;\k_0\ZRWsign^{-1}]$} (and similar arguments can be
applied for $k\in[\k_0^{-1}\ZRWsign;1]$).

Finally, for all vertices near $A$, one has
\[
\ZRW\O\ccdot A\ge\const\quad\mbox{and}\quad \ZRW\O\ccdot B \le\const \cdot\,\ZRWsign.
\]
Thus, choosing $\k_0$ small enough (independently of $\O,A,B$), one
ensures that $\O_A^B[\k_0\ZRWsign^{-1}]$ is connected (and so $\O
_A^B[k]$ is connected for all $k\ge\k_0\ZRWsign^{-1}$).
\end{pf}

Dealing with more involved configurations (e.g., simply connected
discrete domains with many marked boundary points), in addition to
Theorem~\ref{Thm:Separator}, it is useful to have some information
concerning mutual ``topological'' properties of cross-cuts separating
$A$ and $B$, corresponding to \emph{different} pairs $A,B$. In order to
shorten the notation below, for $x\in\pa\O\setminus(A\cup B)$, we set
\[
\O_A^B[x]:=\O_A^B\bigl[
\RRsign_\O(x;A,B) \bigr]=\biggl\{u\in\O\dvtx\frac{\ZRW\O
{u}A}{\ZRW\O
{u}B} \ge
\frac{\ZRW\O{x}A}{\ZRW\O{x}B}\biggr\}.
\]
Roughly speaking, $\O_A^B[x]$ is the set of those $u\in\O$ which are
``not further in $\O$''
from $A$ compared to $B$ than a reference point $x$. Note that
since the function $\RR\O\ccdot AB$ is monotone on the boundary arcs
$(a_2b_1)_\O$ and $(b_2a_1)_\O$
(see Lemma~\ref{Lemma:R-monotone}), $\O_A^B[x]$ also behaves in a
monotone way when $x$ runs along
$\pa\O\setminus(A\cup B)$.

\begin{proposition}
\label{Prop:InclSeparators}
Let $\O$ be a simply connected discrete domain, disjoint boundary arcs
$A=[a_1a_2]_\O$, $B=[b_1b_2]_\O$ and $C=[c_1c_2]_\O$ be
listed counterclockwise, and $B\cup C=[b_1c_2]_\O$ [i.e., $b_2$ and
$c_1$ are consecutive points of $\pa\O$; see Figure~\ref{Fig:
domain+slit}\textup{(B)}]. Then
\begin{eqnarray*}
\O_A^C[x]&\ss&
\O_A^{B\cup C}[x]\ss\O_A^B[x] \qquad \mbox{for
any } x\in(a_2b_1)_\O,
\\
\O_A^B[y]&\ss& \O _A^{B\cup C}[y]\ss
\O_A^C[y] \qquad \mbox{for any } y\in(b_2a_1)_\O.
\end{eqnarray*}
\end{proposition}

\begin{pf}
Let $x\in(a_2b_1)_\O$ (the second case is similar) and $u\in\Int\O
_A^C[x]$ which, by definition, means
%
\begin{equation}
\label{UinOABx1} {\ZRW\O{u}A}\cdot{\ZRW\O{x}C}\ge{\ZRW\O{x}A}\cdot{\ZRW\O{u}C}.
\end{equation}
We need to check that $u\in\Int\O_A^{B\cup C}[x]$ which is
equivalent to
%
\begin{equation}
\label{UinOABx2} {\ZRW\O{u}A}\cdot{\ZRW\O{x} {B \cup C}}\ge{\ZRW\O{x}A}\cdot {\ZRW
\O {u} {B \cup C}}.
\end{equation}
%
Since $\ZRW\O\ccdot{B \cup C}=\ZRW\O\ccdot B + \ZRW\O\ccdot C$, it
is sufficient to prove
that, for any $b\in B$,
\[
\frac{\ZRW\O{x}b}{\ZRW\O{u}b}=\frac{\ZRW\O{x}{b_\mathrm
{int}}}{\ZRW\O
{u}{b_\mathrm{int}}} \ge\frac{\ZRW\O{x}A}{\ZRW\O{u}A}.
\]

For $v\in\O$, denote
\[
H(v):= %
\cases{ {\ZRW\O{u}A}\cdot{\ZRW\O{x}v} - {\ZRW\O{x}A}\cdot{\ZRW
\O {u}v}, & \quad $v\in\Int\O$,
\vspace*{2pt}\cr
\mu_x^{-1}\1[v = x], &\quad $v\in
\pa\O$.} %
\]
Suppose that, on the contrary, $H(b_\mathrm{int})<0$ for some $b\in B$.
Since the function $H$
is harmonic everywhere in $\O$ except $u$ (where it is subharmonic),
and vanishes on $\pa\O$ everywhere except $x$ (where it is strictly
positive), there exists a
nearest-neighbor path $\gamma_{bu}$ running from $b_\mathrm{int}$ to $u$
such that $H<0$ along
$\gamma_{bu}$. On the other hand, $H(c_\mathrm{int})\ge0$ for at least
one $c\in C$
[otherwise, summation along the arc $C$ gives a contradiction with
(\ref
{UinOABx1})]. Hence,
there exists a nearest-neighbor path $\gamma_{cx}$ running from
$c_\mathrm
{int}$ to
$x$ such that $H\ge0$ along $\gamma_{cx}$. Since these two paths cannot
cross each
other, and $\O$ is simply connected, $\gamma_{cx}$ should separate
$u$ and
$A$. Then the maximum principle implies
$H(a_\mathrm{int})>0$ for any $a\in A$. Summing along the arc~$A$, one
arrives at the
inequality
%
\begin{equation}
\label{HmAcontradiction} \ZRW\O uA\cdot\ZRW\O{x}A > \ZRW\O {x}A\cdot\ZRW \O uA,
\end{equation}
which is a contradiction. Thus $\O_A^C[x]\subset\O_A^{B\cup C}[x]$.

Now let $u\in\O_A^{B\cup C}[x]$. Arguing as above, in order to
deduce $u\in\O_A^B[x]$ from~(\ref{UinOABx2}), it is sufficient to prove
that, for all $c\in C$,
\[
\frac{\ZRW\O{x}c}{\ZRW\O{u}c}=\frac{\ZRW\O{x}{c_\mathrm
{int}}}{\ZRW\O
{u}{c_\mathrm{int}}} \le\frac{\ZRW\O{x}A}{\ZRW\O{u}A}.
\]
Suppose, on the contrary, that $H(c_{\mathrm{int}})>0$ for some $c\in
C$. Then there exists a path $\gamma_{cx}$ running from
$c_{\mathrm{int}}$ to $x$ such that $H>0$ along $\gamma_{cx}$. Now there
are two
cases. If $\gamma_{cx}$ separates $u$ and $A$, then the maximum
principle implies
$H(a_{\mathrm{int}})>0$ for all $a\in A$, which leads to the same contradiction
(\ref{HmAcontradiction}). But if $\gamma_{cx}$ does not separate $u$ and
$A$, then it separates
$u$ and $B$. Therefore, $H(b_\mathrm{int})>0$ for all $b\in B$, which
directly gives
$u\in\O_A^B[x]$ by summation along $B$.
\end{pf}

\section{Extremal lengths}
\label{Sect:ExtremalLength}

In this section we recall the notion of a discrete extremal length $\EL
\O{[ab]_\O}{[cd]_\O}$ between two opposite boundary arcs of a discrete
simply connected domain $\O$ (which is nothing but the resistance of
the corresponding electrical network), first discussed by Duffin in
\cite{Duf62}. Note that $\ELsign_\O$ can be defined in two \emph
{equivalent} ways: (a) via some extremal problem (see Definition~\ref
{Def:ExtLengthVar}) and (b) via solution to a Dirichlet--Neumann
boundary value problem; see Proposition~\ref{Prop:EL=U/I} and
Remark~\ref{Rem:DualFct}.
The most important feature of (a) is that it allows one to estimate
$\ELsign_\O$ ``in geometric terms.'' In particular, we show that
$\ELsign_\O$ is uniformly comparable to its continuous counterpart,
extremal length of the corresponding polygonal quadrilateral; see
Proposition~\ref{Prop:ELdiscr=ELcont} and Corollary~\ref{Cor:ELduality}
for details. At the same time, approach (b) allows us to relate
$\ELsign
_\O$ to the random walk partition function $\ZRWsign_\O$ discussed
above; see Proposition~\ref{Prop:ZleL-1}. Note that this connection is
of crucial importance for the next section, which starts with the
complete set of uniform double-sided estimates relating $\xYsign_\O
,\ZRWsign_\O$ and $\ELsign_\O$; see Theorem~\ref{Thm:DblSidedEstimates}.

Let $\O$ be a discrete domain and $E^\O=E^\O_\mathrm{int}\cup E^\O
_\mathrm{bd}$ be the set of edges of $\O$. For a given
function (``discrete metric'') $g\dvtx E^\O\to[0;+\iy)$, we define the
``$g$-area'' of $\O$ by
\[
A_g(\O):=\sum_{e\in E^\O}
\mathrm{w}_eg_e^2,
\]
where $\mathrm{w}_e$ denote weights of edges of $\G$; see
Section~\ref{SubSect:GraphAssumptions}. Further, for a given subset
$\gamma\ss E^\O$
(e.g., a nearest-neighbor path running in $\O$), we define its
``$g$-length'' by
\[
L_g(\gamma):=\sum_{e\in\gamma}g_e.
\]
Finally, for a family $\cE$ of lattice paths in $\O$, we set $L_g(\cE
):=\inf_{\gamma\in\cE}L_g(\gamma)$.

\begin{definition}
\label{Def:ExtLengthVar} The \emph{discrete extremal length} of the
family $\cE$ is given by
%
\begin{equation}
\label{ELdisc=sup} \ELsign[\cE]:=\sup_{g\dvtx E^\O\to[0;+\iy)}\frac{[L_g(\cE
)]^2}{A_g(\O)},
\end{equation}
where the supremum is taken over all $g$'s such that $0<A_g(\O
)<+\infty
$. In particular, if
$\O$ is simply connected, $a,b,c,d\in\pa\O$ are listed
counterclockwise, and $b\ne c$, $d\ne a$, then
we define $\EL{\O}{[ab]_\O}{[cd]_\O}$ as the extremal
length of
the family
$(\O;[ab]_\O\leftrightarrow[cd]_\O)$ of all lattice
paths connecting the
boundary arcs $[ab]_\O$ and $[cd]_\O$ inside
$\O$.
\end{definition}

Note that the discrete extremal metric $g^\mathrm{max}$ [that provides
a maximal value in the right-hand side of (\ref{ELdisc=sup})] always
exists and is unique up to a multiplicative constant. Indeed, by
homogeneity, it is enough to consider only those $g$ that satisfy the
additional assumption $A_g(\O)=1$ and the set of all such discrete
metrics is compact in the natural topology (as $E^\O$ is finite).
Moreover, if $g,g'$ are two extremal metrics such that $A_g(\O
)=A_{g'}(\O)=1$, then the metric $g'':=\frac{1}{2}(g+g')$ satisfies
$L_{g''}(\mathcal{E})\ge L_g(\mathcal{E})=L_{g'}(\mathcal{E})$ and we
have $A_{g''}(\O)<1$ unless $g=g'$. Thus if $g\ne g'$, then $g''$
provides a larger value in (\ref{ELdisc=sup}).

Definition~\ref{Def:ExtLengthVar} easily allows one to estimate $\EL
\O
{[ab]_\O}{[cd]_\O}$ {from below}, since for this purpose it is
sufficient to take any ``discrete metric'' $g$ in $\O$ and estimate
$A_{g}(\O)$ and $L_{g}(\O;[ab]_\O\leftrightarrow[cd]_\O)$ for
this particular $g$.
Note that the most natural way to give an {upper} bound is to use (some
form of) the duality between the extremal lengths $\EL{\O}{[ab]_\O
}{[cd]_\O}$ and $\EL{\O}{[bc]_\O}{[da]_\O}$; see Corollary~\ref
{Cor:ELduality} below.

For a (simply connected) discrete domain $\O\ss\G$, we denote its
\emph{
polygonal representation} as the open (simply connected) set $\O^\C
\ss\C$ bounded by the polyline $x^0_\mathrm{mid}x^1_\mathrm
{mid}\cdots x^n_\mathrm{mid}x^0_\mathrm{mid}$ passing through all \emph
{middle points} $x^k_\mathrm{mid}:=\frac{1}{2}(x^k + x^k_\mathrm
{int})$ of boundary edges $(x^k_\mathrm{int}x^k)\in\pa\O$ in their
natural order (counterclockwise with respect to $\O$); see Figure~\ref
{Fig: domain-abcd+L}. For $a,b\in\pa\O$, $a\ne b$, we denote by
$[ab]_\O^\C\ss\pa\O^\C$ the part of this polyline from
$a_\mathrm
{mid}$ to $b_\mathrm{mid}$, viewed as a boundary arc of $\O^\C$. In
case $a=b$, we slightly modify this definition, setting, say, $[aa]_\O^\C:=[\frac{1}{2}(a^-_\mathrm{mid}+a_\mathrm
{mid});a_\mathrm
{mid}]\cup[a_\mathrm{mid};\frac{1}{2}(a^+_\mathrm{mid}+a_\mathrm
{mid})]$, where $a^\mp$ denote the boundary points of $\O$ just before
and next to $a$.

Let $\ELsign_\O^\C:=\ELsign_\O^\C([ab]_\O^\C;[cd]_\O^\C)$
denote the
classical extremal distance between the opposite arcs of a topological
quadrilateral $(\O^\C;a_\mathrm{mid},b_\mathrm{mid},c_\mathrm
{mid},d_\mathrm{mid})$ in the complex plane; for example, see \cite{Ahl73},
Chapter~4, or \cite{GM05}, Chapter IV. Note that our
Definition~\ref{Def:ExtLengthVar} replicates the classical one, which says
%
\begin{equation}
\label{ELcont=sup} \ELsign_\O^\C\bigl([ab]_\O^\C;[cd]_\O^\C
\bigr) =\sup_{g:\O^\C\to
[0;+\iy)}\frac
{ [\inf_{\gamma:[ab]_\O^\C\leftrightarrow[cd]_\O^\C}\int_\gamma g\,ds
]^2}{\iint_\O g^2\,dx\,dy},
\end{equation}
where the supremum is taken over all $g$ such that $0<\iint_\O
g^2\,dx\,dy<+\infty$, and the infimum is over all curves connecting
$[ab]_\O
^\C$ and $[cd]_\O^\C$ inside $\O^\C$; see \cite{Ahl73,GM05}. It
is well
known that the extremal metric $g^\mathrm{max}$ [providing a maximal
value in the right-hand side of (\ref{ELcont=sup})] exists, is unique
up to a multiplicative constant and is given by $g^\mathrm
{max}(z)\equiv|\phi'(z)|$ where $\phi$ conformally maps $\O^\C$ onto
the rectangle
%
\begin{eqnarray}
\label{xPhiO->Rect} %
\phi\dvtx\O^\C&\to&\bigl
\{z\dvtx0<\Re z< 1, 0<\Im x< \ELsign_\mathrm {cont}^{-1}\bigr\},
\nonumber
\\[-8pt]
\\[-8pt]
\nonumber
a&\mapsto& i\ELsign_\mathrm{cont}^{-1},\qquad b
\mapsto0,\qquad c\mapsto 1,\qquad d\mapsto1+i\ELsign_\mathrm{cont}^{-1}.
\end{eqnarray}

\begin{proposition}
\label{Prop:ELdiscr=ELcont}
Let $\O$ be a simply connected discrete domain and $a,b,c,d\in\pa\O$,
$b\ne c$, $d\ne a$, be listed in the counterclockwise order. Then
%
\begin{equation}
\label{Eldiscr=ELcont} \ELsign_{\O}\bigl([ab]_\O;[cd]_\O
\bigr) \asymp\ELsign _\O^\C \bigl([ab]_\O^\C;[cd]_\O^\C
\bigr)
\end{equation}
with some uniform (i.e., independent of $\O,a,b,c,d$) constants.
\end{proposition}

\begin{pf}
Let $\ELsign_\mathrm{disc}:=\EL\O{[ab]_\O}{[cd]_\O}$ and $\ELsign
_\mathrm{cont}:=\ELsign_\O^\C({[ab]_\O^\C};{[cd]_\O^\C})$. We
prove two
one-sided estimates separately, taking a solution to either discrete
(\ref{ELdisc=sup}) or continuous (\ref{ELcont=sup}) extremal problem,
and constructing some related metric for the other one, thus obtaining
a lower bound for the other (continuous or discrete) extremal length.

(i) $\ELsign_\mathrm{cont}\ge\const\cdot\,\ELsign
_\mathrm{disc}$. Let $g^\mathrm{max}_e$, $e\in E^\O$, be the extremal
metric in (\ref{ELdisc=sup}). For a face $f$ of $\G$ (considered as a
convex polygon in $\C$), let $\L_f\ss\G$ be defined by saying that
$\Int
\L_f$ consists of all vertices incident to $f$, and $\L_f^\C$ be the
polygonal representation of $\L_f$. Further, for an edge $e\in E^\G$
separating two faces $f$ and $f'$, let $\Int\L_e:=\Int\L_f\cup\Int
\L
_{f'}$ and $\L_e^\C$ be the polygonal representation of $\L_e$; see
Figure~\ref{Fig: domain-abcd+L}. We set
\[
g(z):=\sum_{e\in E^\O} g^\mathrm{max}_er_e^{-1}
\1_{\L_e^\C}(z),\qquad z\in\O^\C,
\]
where $r_e$ denotes the length of $e$. Since each point in $\O^\C$
belongs to a uniformly bounded number of edge neighborhoods $\L_e^\C$
(recall that degrees of faces and vertices of $\G$ are uniformly
bounded), one has
%
\begin{equation}
\label{xAcont=Adisc} \iint_\O g^2 \,dx\,dy \asymp\sum
_{e\in E^\O} \bigl(g^\mathrm {max}_e
\bigr)^2r_e^{-2}\Area\bigl(
\L_e^\C\cap\O^\C\bigr)\asymp\sum
_{e\in E^\O
} \mathrm {w}_e \bigl(g^\mathrm{max}_e
\bigr)^2,
\end{equation}
as $r_e^{-2}\Area(\L_e^\C\cap\O^\C)\asymp1\asymp\mathrm{w}_e$;
see our
assumptions on $\G$ listed in Section~\ref{SubSect:GraphAssumptions}.

Now let $\gamma$ be any continuous curve crossing $\O^\C$ from
$[ab]_\O^\C$
to $[cd]_\O^\C$, $F^\gamma$ be the set of all (closed) \emph{faces}
touched by $\gamma$, and $E^\gamma\ss E^\O$ be the set of all
edges of $\O$
incident to those faces. It is clear that $E^\gamma$ contains a discrete
nearest-neighbor path running from $[ab]_\O$ to $[cd]_\O$. Thus it is
sufficient to estimate $\int_\gamma g\,ds$ (from below) via $\sum_{e\in E^\gamma
}g^\mathrm{max}_e$.
Note that, for any $f\in F^\gamma$,
\[
\gamma \mbox{ should cross the annulus type polygon } \L_f^\C
\setminus f
\mbox{ at least once}.
\]
Let $\gamma_f$ denote this crossing (there is one exceptional
situation: if, say, $b$ and $c$ are two consecutive boundary points,
and $f$ is a boundary face between them, then $\gamma$ may not cross the
annulus $\L_f^\C\setminus f$, so we denote by $\gamma_f$ the corresponding
crossing of $\L_f^\C$ itself). As degrees of vertices and faces of
$\G$
are uniformly bounded, each piece of $\gamma$ belongs to a bounded number
of $\gamma_f$. Since $\Length(\gamma_f)\ge\const\cdot\, r_e$ for
any $e\sim f$
(all those $r_e$ are comparable to each other due to our assumptions),
we arrive at
\begin{eqnarray*}
\int_\gamma g\,ds&\ge&\const\cdot\,\sum
_{f\in F^\gamma}\int_{\gamma_f} g\,ds
\\
&\ge& \const\cdot\,
\sum_{e\sim f\in F^\gamma}\Length(\gamma _f)g^\mathrm{max}_er_e^{-1}
\ge\const\cdot\,\sum_{e\in E^\gamma} g^\mathrm{max}_e.
\end{eqnarray*}
Together with (\ref{xAcont=Adisc}), this allows us to conclude that
\[
\ELsign_\mathrm{cont}\ge\frac{[\inf_{\gamma}\int_\gamma
g\,ds]^2}{\iint_\O
g^2\,dx\,dy} \ge\const\cdot\,
\frac{ [\inf_{\gamma}L_{g^\mathrm
{max}}(E^\gamma
)]^2}{A_{g^\mathrm{max}}(\O)}\ge\const\cdot\,\ELsign_\mathrm{disc}.
\]

(ii) $\ELsign_\mathrm{disc}\ge\const\cdot\,\ELsign
_\mathrm{cont}$. Let $g^\mathrm{max}\dvtx\O^\C\to\R^+$ be the extremal
metric in (\ref{ELcont=sup}). Recall that $g^\mathrm{max}(z)\equiv
|\phi
'(z)|$, where the conformal mapping $\phi$ is given by (\ref{xPhiO->Rect}).
We set
\[
g_e:=\int_{\O^\C\cap e} g^\mathrm{max}\,ds,\qquad e\in
E^\O.
\]
Since we have
$\sum_{e\in\gamma}g_e=\int_{\gamma} g^\mathrm{max}\,ds$ for each
nearest-neighbor
path $\gamma$ in $\O$, it is sufficient to estimate $\sum_{e\in\O
_e}\mathrm{w}_e g_e^2$ (from above) via $\iint(g^\mathrm{max})^2\,dx\,dy$.

Let $z_e$ denote the mid-point of an \emph{inner} edge $e$. As $\phi$
is a univalent holomorphic function (in $\L_e^\C\cap\O^\C$), all values
$|\phi'(z)|$ for $z\in e$ are uniformly comparable to each other (and
comparable to all other values $|\phi'(z)|$ for $z$ near $z_e$); for
example, see \cite{Ahl73}, Chapter~5, Theorem~5-3, or \cite{GM05}, Chapter~1,
Theorem~4.5. In particular, this implies
\[
g_e^2\asymp r_e^2\bigl|
\phi'(z_e)\bigr|^2\le\const\cdot\,
\iint_{\L_e^\C\cap
\O^\C
}\bigl|\phi'\bigr|^2\,dx\,dy.
\]
It is easy to see that the same holds true for \emph{boundary} edges:
if $\O^\C$ has an inner angle $\theta_x\in(\eta_0;2\pi]$ at the
boundary point $x_\mathrm{mid}\in\pa\O$, then $\phi$ behaves like
$(z-x_\mathrm{mid})^{\pi/\theta_x}$ near $x$ [or $(z-x_\mathrm
{mid})^{\pi/2\theta_x}$, if $x$ is one of the corners $a,b,c,d$]. Hence
$|\phi'|$ blows up not faster than $|z - x_\mathrm{mid}|^{-1/2}$ (or
$|z - x_\mathrm{mid}|^{-3/4}$, resp.) when $z$ approaches $x_\mathrm
{mid}$, which means $g_e\asymp r_e|\phi'(x_\mathrm{int})|$.

As each point in $\O^\C$ belongs to a uniformly bounded number of $\L
_e^\C$, we obtain
\[
\sum_{e\in E^\O}\mathrm{w}_eg_e^2
\le\const\cdot\,\iint_{\O^\C
}\bigl|\phi'\bigr|^2\,dx\,dy.
\]
Therefore,
\[
\ELsign_\mathrm{disc}\ge\frac{[\inf_{\gamma}\sum_{e\in\gamma
}g_e]^2}{\sum_{e\in E^\O}\mathrm{w}_eg_e^2} \ge\const\cdot\,
\frac{[\inf_\gamma
\int_{\gamma}
g^\mathrm{max}\,ds]^2}{\iint_{\O^\C}(g^\mathrm{max})^2\,dx\,dy}\ge \const\cdot\, \ELsign_\mathrm{cont}.
\]
\upqed\end{pf}

\begin{corollary}
\label{Cor:ELduality} Let $\O$ be a simply connected discrete domain
and four distinct boundary points $a,b,c,d\in\pa\O$ be listed
counterclockwise. Then
%
\begin{equation}
\label{ELduality} \ELsign_{\O}\bigl([ab]_\O;[cd]_\O
\bigr) \cdot \ELsign_{\O}\bigl([bc]_\O;[da]_\O
\bigr) \asymp1
\end{equation}
with some uniform (i.e., independent of $\O,a,b,c,d$) constants.
\end{corollary}

\begin{pf}
The proof directly follows from (\ref{Eldiscr=ELcont}) applied to both
factors and the \emph{exact} duality
\[
\ELsign_\O^\C\bigl([ab]_\O^\C;[cd]_\O^\C
\bigr)\cdot\ELsign_\O^\C \bigl([bc]_\O^\C
;[da]_\O^\C\bigr)=1
\]
of continuous extremal lengths.
\end{pf}

We now move on to the second approach, the notion of extremal length
via solution to the following Dirichlet--Neumann boundary value problem
[which corresponds to the real part $\Re\phi$ of the uniformization map
(\ref{xPhiO->Rect})].

Let $\O$ be simply connected and $a,b,c,d\in\pa\O$, $b\ne c$, $d\ne a$,
be listed
counterclockwise. Denote by $V=V_{(\O;[ab]_\O,[cd]_\O)}\dvtx\O\to[0;1]$
the unique discrete harmonic in $\O$ function (electric potential) such
that $V\equiv0$ on $[ab]_\O$, $V\equiv1$ on $[cd]_\O$, and $V$
satisfies Neumann boundary conditions [i.e., $V(x_\mathrm{int})=V(x)$]
for $x\in\pa\O\setminus([ab]_\O\cup[cd]_\O)$. We also set
\[
I(V):=\sum_{x\in[ab]_\O} \mathrm{w}_{xx_\mathrm
{int}}V(x_\mathrm{int})=
\sum_{x\in[cd]_\O} \mathrm{w}_{xx_\mathrm
{int}}\bigl(1 -
V(x_\mathrm{int})\bigr)
\]
[note that $\sum_{x\in\pa\O} \mathrm{w}_{xx_\mathrm{int}}(V(x) -
V(x_\mathrm{int}))=\sum_{u\in\Int\O}\mu_u[\D V](u)=0$].

The next proposition rephrases $\EL\O{[ab]_\O}{[cd]_\O}$ via $I(V)$
(which is nothing but the electric current in the corresponding
network). Note that contrary to the classical setup, this identity does
\emph{not} allow to replace double-sided estimate (\ref{ELduality}) by
an equality. Indeed, mimicking the continuous case, one can pass from
$V$ to its harmonic conjugate function $V^*$ that solves the similar
boundary value problem for dual arcs, but this $V^*$ is defined on a
dual graph $\G^*$, leading to the extremal length of some \emph{other}
discrete quadrilateral (drawn on $\G^*$) rather than $\O\ss\G$ itself;
see also Remark~\ref{Rem:DualFct}.

\begin{proposition}
\label{Prop:EL=U/I}
For any simply connected discrete domain $\O$ and any $a,b,c,d\in\pa
\O
$, $b\ne c$, $d\ne a$, listed counterclockwise, the following is fulfilled:
%
\begin{equation}
\label{EL=I^-1} \ELsign_{\O}\bigl([ab]_\O;[cd]_\O
\bigr) = \bigl[I(V_{(\O;[ab]_\O,[cd]_\O)})\bigr]^{-1}.
\end{equation}
\end{proposition}

\begin{pf}
See \cite{Duf62}. The core idea is to construct the function $V$
explicitly in terms of the extremal discrete metric $g^\mathrm{max}$ of
the family $(\O;[ab]_\O\leftrightarrow[cd]_\O)$. Namely, let $(\O
;u \leftrightarrow[ab]_\O)$
denote the family of all discrete paths running from $u\in\O$ to the
boundary arc $[ab]_\O$ inside $\O$, and
\[
V(u):=L_{g^\mathrm{max}}{\bigl(\O;u \leftrightarrow[ab]_\O\bigr)}.
\]
Then $V$ is constant on $[cd]_\O$ and satisfies Neumann boundary
conditions on both $(bc)_\O$ and $(da)_\O$ [if one of these properties
fails, then one can improve $g^\mathrm{max}$ on the corresponding
boundary edge so that $L_g(\cE)$ does not change while $A_g(\O)$
decreases]. In particular, one can normalize $g^\mathrm{max}$ so that
$V\equiv1$ on $[cd]_\O$.

Moreover, $V$ is discrete harmonic in $\O$. Indeed, note that $V(u')
- V(u)=\pm g^\mathrm{max}_{uu'}$ for any $(uu')\in E^\O$ (otherwise,
one can improve $g^\mathrm{max}_{uu'}$). Then, for a given $u\in\Int
\O
$, replacing $g^\mathrm{max}_{uu'}$ by $g^\mathrm{max}_{uu'} + \ve$
on all edges $(uu')\in E^\O$ such that $V(u') > V(u)$ and,
simultaneously, replacing $g^\mathrm{max}_{uu'}$ by $g^\mathrm
{max}_{uu'} - \ve$ on all $(uu')\in E^\O$ such that $V(u') < V(u)$,
one does not change global distances [and, in particular, does not
change $L_g(\cE)$], while the area $A_g(\O)$ changes by $\ve\mu
_u[\D
V](u) + O(\ve^2)$.

Finally, using discrete integration by parts and $[\D V](u)\equiv0$,
one concludes that
\begin{eqnarray*}
\ELsign_\O^{-1} & = &A_{g^\mathrm{max}}(\O)= \sum
_{e=(uu')\in
E^\O}\mathrm{w}_{e}\bigl(V\bigl(u'
\bigr) - V(u)\bigr)^2 
\\
& =& -\sum
_{u\in\Int\O} \mu_u [\D V](u)V(u) - \sum
_{x\in
\pa\O} \mathrm{w}_{xx_\mathrm{int}}\bigl(V(x_\mathrm{int})
- V(x)\bigr)V(x)
\\
& =&\sum_{x\in[cd]_\O}
\mathrm{w}_{xx_\mathrm{int}}\bigl(1 - V(x_\mathrm{int})\bigr)=I(V).
\end{eqnarray*}
%
\upqed\end{pf}

Note that, for any discrete harmonic in $\O$ function $V$, one can
construct a discrete \emph{harmonic conjugate} function $V^*$ which is
uniquely defined (up to an additive constant) on faces of $\O$
(including boundary ones) by saying
%
\begin{equation}
\label{V*def} H\bigl(f_{vv'}^\mathrm{left}\bigr) - H
\bigl(f_{vv'}^\mathrm{right}\bigr):=\mathrm {w}_{vv'}
\cdot\bigl(H\bigl(v'\bigr) - H(v)\bigr)
\end{equation}
for any oriented edge $(vv')\in E^\O$, where $f_{vv'}^\mathrm{left}$
and $f_{vv'}^\mathrm{right}$ denote faces to the left and to the right
of $(vv')$, respectively. The function $V^*$ is well defined locally
(if and only if $\D V=0$), and hence well defined globally, as $\O$ is
simply connected. Moreover, for any inner face $f$ in $\O$, it
satisfies a discrete harmonicity condition
%
\begin{equation}
\label{V*harm} \sum_{f'\sim f}\mathrm{w}_{ff'}
\bigl(V^*\bigl(f'\bigr)-V^*(f)\bigr)=0,
\end{equation}
where $\mathrm{w}_{ff'}:=\mathrm{w}_{vv'}^{-1}$ for any couple of dual
edges $(ff')=(vv')^*$.

\begin{remark}
\label{Rem:DualFct}
If one takes $V=V_{(\O;[ab]_\O,[cd]_\O)}$, then the harmonic conjugate
function $V^*$ is constant along boundary arcs $(bc)_\O$ and $(da)_\O$
(since $V$ satisfies Neumann boundary conditions on these arcs). Fixing
an additive constant so that $V^*\equiv0$ on $(bc)_\O$ and tracking
the increment of $V^*$ along $[ab]_\O$, one obtains $V^*\equiv
I(V) \mbox{ on } (da)_\O$. Further, Dirichlet boundary conditions for
$V$ on $[ab]_\O$ and $[cd]_\O$ can be directly translated into Neumann
conditions for $V^*$ (one can easily see that $V^*$ satisfies (\ref
{V*harm}) with a smaller number of terms at all faces touching $[ab]_\O
$ or $[cd]_\O$). Thus $[I(V)]^{-1}\cdot V^*$ solves the same
Dirichlet--Neumann boundary value problem for the \emph{dual}
quadrilateral drawn on $\G^*$. Moreover,
\[
\sum_{(ff')^*\in E^\O}\mathrm{w}_{ff'}\bigl(V^*
\bigl(f'\bigr) - V^*(f)\bigr)^2= \sum
_{(vv')\in E^\O}\mathrm{w}_{vv'} \bigl(V\bigl(v'
\bigr) - V(v)\bigr)^2=\ELsign_\O^{-1},
\]
and hence the 
dual extremal length $\ELsign_\O^*$ is \emph{equal} to
$[I(V)^{-2}\ELsign_\O^{-1}]^{-1} =\ELsign_\O^{-1}$.
\end{remark}

The last proposition in this section gives an estimate for the
partition function $\ZRW\O{[ab]_\O}{[cd]_\O}$ of random walks joining
$[ab]_\O$ and $[cd]_\O$ in $\O$ via the extremal length $\EL\O
{[ab]_\O
}{[cd]_\O}$ [note that the latter can be thought about as the
(reciprocal of) similar partition function for random walks \emph
{reflecting} from the dual boundary arcs].

\begin{proposition}
\label{Prop:ZleL-1} Let $\O$ be a simply connected discrete domain, and
$a,b,c,d\in\pa\O$, $b\ne c$, $d\ne a$, be listed counterclockwise. Then
%
\[
\ZRWsign_{\O}\bigl([ab]_\O;[cd]_\O\bigr)
\le\const\cdot\,\bigl(
\ELsign_{\O}\bigl([ab]_\O;[cd]_\O\bigr)
\bigr)^{-1},
\]
where the constant does not depend on $\O,a,b,c,d$. Moreover, if we
additionally assume that $\EL\O{[ab]_\O}{[cd]_\O}\le\const$,
then
\[
\ZRWsign_{\O}\bigl([ab]_\O;[cd]_\O\bigr)
\asymp\bigl(
\ELsign_{\O}\bigl([ab]_\O;[cd]_\O\bigr)
\bigr)^{-1}
\]
(with constants in $\asymp$ depending on the upper bound for $\ELsign
_\O
$ but independent of $\O,a,b,c,d$).
\end{proposition}

\begin{pf}
It is easy to see that, for any $u\in\Int\O$, $V(u)$ is equal to the
probability of the event that the random walk started at $u$ and
reflecting from complementary arcs $(bc)_\O$,$(da)_\O$ exists $\O$
through $[cd]_\O$ (indeed, this probability is a discrete harmonic
function which satisfies the same boundary conditions as $V$). Hence,
for any $x\in[ab]_\O$,
\[
V(x_\mathrm{int})\ge\const\cdot
\,
\ZRWsign_{\O}\bigl(x_\mathrm{int};[cd]_\O\bigr)
\asymp
\ZRWsign_{\O}\bigl(x;[cd]_\O\bigr),
\]
since the right-hand side is (up to a constant) the same probability
for the random walk with absorbing boundary conditions on $(bc)_\O$ and
$(da)_\O$. Thus (\ref{EL=I^-1}) gives
\[
\bigl( \ELsign_{\O}\bigl([ab]_\O;[cd]_\O\bigr)
\bigr)^{-1}=\sum_{x\in[ab]_\O} \mathrm
{w}_{xx_\mathrm{int}}V(x_\mathrm{int}) \ge\const\cdot\,
\ZRWsign_{\O}\bigl([ab]_\O;[cd]_\O\bigr).
\]

Further, let $\ELsign_\O:=\EL\O{[ab]_\O}{[cd]_\O}\le\const$.
Due to
Corollary~\ref{Cor:ELduality}, it is equivalent to $\EL\O{[bc]_\O
}{[da]_\O}\ge\const$. We have seen above that this implies
$\ZRW\O{[bc]_\O}{[da]_\O}\le\const$ which is equivalent to $\xY
{\O
}bcda\le\const$ due to Theorem~\ref{Thm:ZlogYestim}. Therefore,
\begin{eqnarray*}
\ZRWsign_\O:=
\ZRWsign_{\O}\bigl([ab];[cd]\bigr)
&\asymp&\log\bigl(1 +
\xYsign_{\O}(a,b;c,d)\bigr)
\\
& =& \log\bigl(1 + \bigl(
\xYsign_\O(b,c;d,a) \bigr)^{-1}\bigr)\ge\const.
\end{eqnarray*}
Since $\ZRWsign_\O\le\const\cdot\,\ELsign_\O^{-1}$ in any case, this
implies $\ZRWsign_\O\asymp1$, if $\ELsign_\O\asymp1$.

Thus we are mostly interested in the situation when $\ELsign_\O$ is
very small (i.e., boundary arcs $[ab]_\O$ and $[cd]_\O$ are ``very
close'' to
each other in $\O$). Our strategy in this case is similar to the proof
of Theorem~\ref{Thm:ZlogYestim}: we split $[ab]_\O$ into several
smaller pieces $[a_ka_{k+1}]_\O$ such that $\EL\O{[a_ka_{k+1}]_\O
}{[cd]_\O}\asymp1$ and apply the result obtained above to each of
these smaller arcs.
Recall that
\[
\ELsign_\O^{-1}= \sum_{x\in[ab]_\O}
\mathrm{w}_{xx_\mathrm
{int}}V(x_\mathrm{int}).
\]
We construct boundary points $a=a_0,a_1,\ldots,a_{n+1}=b\in\pa\O$
inductively by the following procedure: if $a_k$ is already chosen, we
move $a_{k+1}$ further along the boundary arc $[ab]_\O$ step by step
until the first vertex $a_{k+1}$ such that
\[
\bigl( \ELsign_{\O}\bigl([a_ka_{k+1}]_\O;[cd]_\O\bigr)
\bigr)^{-1}= \sum_{x\in[a_ka_{k+1}]}
\mathrm{w}_{xx_\mathrm{int}}V_{(\O
;[a_ka_{k+1}]_\O,[cd]_\O)}(x_\mathrm{int})\ge1
\]
(or $a_{k+1}=b$). Note that this sum cannot increase by more than some
uniform constant on each step (as we increase the absorbing boundary
$[a_ka_{k+1}]_\O$, all terms decreases, while the new (last) term is no
greater than $\mathrm{w}_{xx_\mathrm{int}}\le\const$). Therefore,
$\EL
\O{[a_ka_{k+1}]_\O}{[cd]_\O}\asymp1$ for all $k$, possibly except the
last one (when we are forced to choose $a_{n+1}=b$ before the sum
becomes large). As we have seen above, this implies
\[
\ZRWsign_{\O}\bigl([a_ka_{k+1}]_\O;[cd]_\O\bigr)
\asymp1 \qquad\mbox {for all } k=0,1,\ldots,n - 1.
\]
Note that $V=V_{(\O;[ab]_\O,[cd]_\O)}\le V_{(\O;[a_ka_{k+1}]_\O
,[cd]_\O
)}$ due to monotonicity of boundary conditions (the absorbing boundary
is larger in the first case). This gives
\[
\sum_{x\in[a_ka_{k+1}]}\mathrm{w}_{xx_\mathrm{int}}V(x_\mathrm
{int})\le\sum_{x\in[a_ka_{k+1}]}\mathrm{w}_{xx_\mathrm{int}}V_{(\O
;[a_ka_{k+1}]_\O,[cd]_\O)}(x_\mathrm{int})
\le\const
\]
for all $k=0,1,\ldots,n$, thus $\ELsign^{-1}\le\const\cdot\,(n +
1)\asymp
n$, 
which implies the inverse estimate
\[
\ZRWsign_\O\asymp\sum_{k=0}^n
\ZRWsign_{\O}\bigl([a_ka_{k+1}]_\O;[cd]_\O\bigr)
\ge \const\cdot\, n \asymp\ELsign_\O^{-1}.
\]
\upqed\end{pf}

\section{Double-sided estimates of harmonic measure}
\label{Sect:Estimates} \setcounter{equation}{0}

We start this section with Theorem~\ref{Thm:DblSidedEstimates} which
combines uniform estimates obtained above for cross-ratios~$\xYsign_\O
$, partition functions $\ZRWsign_\O$ and extremal lengths $\ELsign
_\O$
of discrete quadrilaterals $(\O;a,b,c,d)$. Then we show how tools
developed in our paper can be used to obtain exponential double-sided
estimates in terms of appropriate extremal lengths for the {discrete}
harmonic measure $\hm\O{u}{[ab]_\O}$ of a ``far'' boundary arc (similar
to the classical ones due to Ahlfors, Beurling and going back to
Carleman; see \cite{Ahl73}, Sections 4-5 and 4-14, and \cite{GM05},
Sections IV.5, IV.6). The main result is given by Theorem~\ref
{Thm:LogHm=ELdiscr}. In particular, it allows us to obtain a \emph
{uniform} double-sided estimate of $\log\hm\O{u}{[ab]_\O}$ via
$\log\hm
{{\O^\C}}u{[ab]_\O^\C}$, where $\omega_{\O^\C}$ denotes the
\emph
{continuous} harmonic measure in a polygonal representation of $\O$;
see Corollary~\ref{Cor:LogHmDisc=LogHmCont}. Note that one cannot hope
to prove the similar estimate for $\hm\O{u}{[ab]_\O}$ itself: dealing
with thin fiords, one is faced with exponentially small harmonic
measures which are highly sensitive to the widths of those fiords.

\begin{theorem}
\label{Thm:DblSidedEstimates} Let $\O$ be a simply connected discrete
domain and distinct boundary points
$a,b,c,d\in\pa\O$ be listed counterclockwise. Denote
\begin{eqnarray*} \xYsign&:=& \xYsign_\O(a,b;c,d) , \qquad
\ZRWsign:=
\ZRWsign_{\O}\bigl([ab]_\O;[cd]_\O\bigr)
, \qquad \ELsign :=
\ELsign_{\O}\bigl([ab]_\O;[cd]_\O\bigr),
\\
\xYsign'&:=& \xYsign_\O(b,c;d,a) , \qquad \ZRWsign':=
\ZRWsign_{\O}\bigl([bc]_\O;[da]_\O\bigr)
, \qquad\ELsign ':=
\ELsign_{\O}\bigl([bc]_\O;[da]_\O\bigr).
\end{eqnarray*}

\textup{(i)} If at least one of the estimates
%
\begin{eqnarray}
\label{EstimEquiv} %
 \xYsign&\le&\const, \qquad \ZRWsign\le
\const, \qquad \ELsign\ge\const,
\nonumber
\\[-8pt]
\\[-8pt]
\nonumber
\xYsign'&\ge&\const, \qquad
\ZRWsign'\ge\const, \qquad \ELsign'\le\const
\end{eqnarray}
holds true, then all these estimates hold true (with
constants depending on the initial bound but independent of $\O
,a,b,c,d$). Moreover, if at least one of $\xYsign$, $\xYsign'$,
$\ZRWsign$, $\ZRWsign'$, $\ELsign$, $\ELsign'$ is of order $1$ (i.e.,
admits the double-sided estimate
$\asymp1$), then they all are of order $1$.

\textup{(ii)} If (\ref{EstimEquiv}) holds true, then the following
double-sided estimates
are fulfilled:
\[
\ZRWsign\asymp\xYsign\quad\mbox{and}\quad \log\bigl(1 + \xYsign ^{-1}
\bigr)\asymp \ELsign.
\]
%
In particular, there exist some constants $\b_{1,2},C_{1,2}>0$ such
that the uniform estimate
%
\begin{equation}
\label{BeurlingForZ} C_1 \cdot\exp[-\b_1\ELsign] \le\ZRWsign
\le C_2 \cdot\exp[-\b_2\ELsign]
\end{equation}
holds true for any discrete quadrilateral $(\O;a,b,c,d)$ satisfying
(\ref{EstimEquiv}).
\end{theorem}

\begin{pf}
(i) It follows from Theorem~\ref{Thm:ZlogYestim} and Proposition~\ref
{Prop:ZleL-1} that
\[
\log(1 + \xYsign)\asymp\ZRWsign\le\const\cdot\,\ELsign^{-1} \quad\mbox{and}\quad
\log\bigl(1 + \xYsign'\bigr)\asymp\ZRWsign'\le\const
\cdot\,\bigl(\ELsign'\bigr)^{-1}.
\]
Moreover, $\xYsign\xYsign'=1$ by definition, and $\ELsign\ELsign
'\asymp
1$ due to
Corollary~\ref{Cor:ELduality}. Therefore, one has
\[
%
\begin{array} {ccccc} \xYsign\le\const& \Lra& \ZRWsign\le
\const& \La& \ELsign\ge \const
\\
\Uda& & & & \Uda
\\
\xYsign'\ge\const& \Lra& \ZRWsign'\ge\const& \Ra&
\ELsign'\le \const, \end{array} %
\]
which gives the equivalence of all six bounds. Interchanging $\xYsign
$, $\ZRWsign$, $\ELsign$ and $\xYsign'$, $\ZRWsign'$, $\ELsign'$, one
obtains the same equivalence of inverse estimates. Thus, if at least
one of these quantities is $\asymp1$, then all others are $\asymp1$
as well.

(ii) Since $\xYsign\le\const$, Remark~\ref
{Rem:XY(Z)comparable} guarantees that $\ZRWsign\asymp\xYsign$. Further,
since $\ELsign'\le\const$, Proposition~\ref{Prop:ZleL-1} gives
$\ZRWsign
'\asymp(\ELsign')^{-1}$, and hence
\[
\log\bigl(1 + \xYsign^{-1}\bigr)=\log\bigl(1 + \xYsign'
\bigr) \asymp\ZRWsign'\asymp \bigl(\ELsign'
\bigr)^{-1}\asymp\ELsign.
\]
Thus, we have $\exp[\b_2\ELsign]\le1 + \xYsign^{-1}\le\exp[\b
_1\ELsign]$ for some $\b_{1,2}>0$, and we also know that $1 + \xYsign
^{-1}\asymp\xYsign^{-1}\asymp\ZRWsign^{-1}$.
\end{pf}

Now let $u\in\Int\O$ and $[ab]_\O\ss\O$ be some boundary arc of
$\O$
which should be thought about as lying ``very far'' from $u$ [so that
the harmonic measure $\hm\O{u}{[ab]_\O}$ is small]. In order to be able
to apply exponential estimate (\ref{BeurlingForZ}) to this harmonic
measure, one should first compare the partition function of random
walks running from $u$ to $[ab]_\O$ in $\O$ with a partition function
of random walks running between opposite sides of some quadrilateral.

Recall that we denote by $\dd\O{u}$ the (Euclidean) distance from $u$
to $\pa\O$ and
let a discrete domain $\Asign_\O=\A\O{u}$ be defined by
\[
\Int \Asign_{\O}(u):=\Int\O\setminus\Int\BBsign_{\varrho_0\dd\O
{u}}^\G(u),
\]
where $\varrho_0=\varrho_0(\varpi_0,\eta_0,\varkappa_0)>0$ is a fixed
constant. If $\varrho_0$ is chosen small enough, Remark~\ref
{Rem:RvBound} implies that for any $\O$ and $u\in\Int\O$:
\begin{itemize}
\item either $u$ belongs to a face touching $\pa\O$;
\item or $\A\O{u}$ is doubly connected [in other words, $\Int\A\O{u}$
contains a cycle surrounding $u$].
\end{itemize}

\begin{remark}
Throughout most of this section (until Theorem~\ref{Thm:LogHm=ELdiscr})
we assume that $\A\O{u}$ is {doubly connected}. Otherwise, one can
apply an appropriate version of Lemma~\ref{Lemma:HmAsympZa}, which
relates $\hm\O{u}{[ab]_\O}$ to the partition function of random walks
running in $\A\O{u}$, and directly estimate the latter partition
function by the corresponding discrete extremal length using (\ref
{BeurlingForZ}); see the proof of Corollary~\ref{Cor:LogHmDisc=LogHmCont}.
\end{remark}

Below we rely upon the following property of the Green function $G_\O
(\cdot;u)$, which is guaranteed by Lemmas \ref
{Lemma:GreenEstimatesViaG} and \ref{Lemma:GreenInDiscs}:
%
\begin{equation}
\label{GboundOnC} G_\O(v;u)\asymp1\qquad \mbox{for all } v\in
\Csign_\O=\Csign_\O (u):=\pa\Asign_\O(u)
\setminus\pa\O,
\end{equation}
where the constants in $\asymp$ are independent of $\O$, $u$ and $v$.
Note that $\Csign_\O(u)$ can be naturally identified with
$\pa\BBsign_{\varrho_0\dd\O{u}}^\G(u)$ if $\A\O{u}$ is doubly connected.

\begin{lemma}
\label{Lemma:HmAsympZa} Let a simply connected discrete domain $\O$ and
$u\in\Int\O$ be such that $\A{\O}u$ is doubly connected, and
$[ab]_\O\ss
\pa\O$. Then,
%
\begin{equation}
\label{HmAsympZa} \omega_\O\bigl(u;[ab]_\O\bigr) \asymp
\ZRWsign_{\Asign_\O(u)}\bigl(\Csign_\O (u);[ab]_\O\bigr).
\end{equation}
\end{lemma}

\begin{pf} For a random walk running from $u$ to $[ab]_\O$ in $\O$,
let $v$ denote its last vertex on $\Csign_\O$ (such a vertex exists due
to topological reasons). Splitting this path into two halves (before
$v$ and after $v$, resp.), one concludes that
\[
\omega_\O\bigl(u;[ab]_\O\bigr) \asymp
\ZRWsign_{\O}\bigl(u;[ab]_\O\bigr)
\asymp\sum_{v\in
\Csign_\O}
\ZRWsign_{\O}(u;v)
\ZRWsign_{\Asign_{\O}}\bigl(v;[ab]_\O\bigr).
\]
As $\ZRW{\O}uv=G_\O(u;v)\asymp1$ for any $v\in\Csign_\O(u)$, this
gives (\ref{HmAsympZa}).
\end{pf}

\begin{figure}[b]

\includegraphics{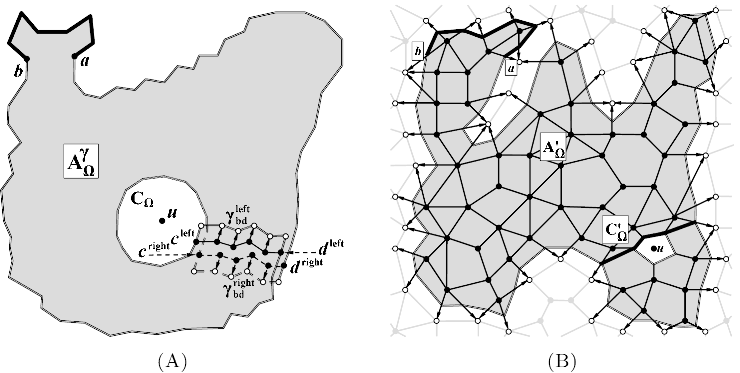}

\caption{\textup{(A)} In order to analyze the discrete
extremal length between $\Csign_\O$ and $[ab]_\O$, we cut a doubly
connected domain $\Asign_\O$ along some nearest-neighbor path $\gamma$
running from $c\in\Csign_\O$ to $d\in(ba)_\O$, so that \emph{two}
identical copies of $\gamma$ are included into a simply connected
domain $\Asign_\O^\gamma$ (which is drawn on the universal cover of
$\Asign
_\O$). Thus the boundary $\pa\Asign_\O^\gamma$ is formed by the
outer part
$(d^\mathrm{right}d^\mathrm{left})=\pa\O$, the inner part
$(c^\mathrm
{left}c^\mathrm{right})=\Csign_\O$ and two paths $\gamma^\mathrm{left}$
and $\gamma^\mathrm{right}$ consisting of vertices neighboring to~$\gamma$.
\textup{(B)} If a vertex $u$ is close to $\pa\O$, it might happen that
$\Asign_\O$ is simply connected or even disconnected. Then we denote by
$\Asign'_\O$ the proper connected component of $\Asign_\O$, and by
$\Csign'_\O$ the corresponding part of $\pa\Asign'_\O$.}
\label{Fig: annulus}
\end{figure}

%
%
%

In order to relate the partition function (\ref{HmAsympZa}) of random
walks in the annulus $\Asign_\O(u)$ to a partition function of random
walks in some \emph{simply connected} domain, below we cut $\Asign_\O
(u)$ along the appropriate nearest-neighbor paths $\gamma=(c_\mathrm
{int}\sim\cdots\sim d_\mathrm{int})$ such that $c\in\Csign_\O$ and
$d\in
\pa\O\setminus[ab]_\O$. For a given $\gamma$ (which is always
assumed to be
a nonself-intersecting path on the \emph{universal cover} $\Asign
^\circlearrowleft_\O$ of $\Asign_\O$), we define a simply connected
domain $\Asign_\O^\gamma$ [see Figure~\ref{Fig:
annulus}(A)] as follows:
\[
\begin{tabular}{p{330pt}@{}}
\emph{if} $\gamma^\mathrm{left}$, $\gamma^\mathrm{right}$ \emph{are two
copies of} $\gamma$
\emph{lying on consecutive sheets of} ${\Asign}^\circlearrowleft_\O$,
\emph{then}
$\Int\Asign_\O^\gamma:= \gamma^\mathrm{left} \cup
[(\Int\Asign _\O)\setminus\gamma ] \cup\gamma^\mathrm{right}
\ss\Int\Asign^\circlearrowleft_\O$.
\end{tabular}
\]
In other words, we cut $\Asign_\O$ along $\gamma$, accounting both
sides of
the slit as \emph{interior} parts of a discrete domain $\Asign_\O
^\gamma$
(which is, in particular, always connected and simply connected). We
then denote by $\gamma^\mathrm{left}_\mathrm{bd}$ and $\gamma
^\mathrm
{right}_\mathrm{bd}$ the corresponding parts of $\pa\Asign_\O
^\gamma$, thus
\[
\pa\Asign_\O^\gamma=\bigl(d^\mathrm{left}d^\mathrm{right}
\bigr)_{{\Asign
}^\circlearrowleft_\O}\cup\gamma^\mathrm{left}_\mathrm{bd}\cup
\bigl(c^\mathrm{right}c^\mathrm{left}\bigr)_{{\Asign}^\circlearrowleft_\O
}\cup
\gamma^\mathrm{right}_\mathrm{bd} ,
\]
where disjoint parts\vspace*{1pt} of $\pa\Asign_\O^\gamma$ are listed counterclockwise
with respect to $\Asign_\O^\gamma$. We also use simpler notation
$(d^\mathrm{left}d^\mathrm{right})_{{\Asign}^\circlearrowleft_\O
}=\pa\O
$ and $(c^\mathrm{right}c^\mathrm{left})_{{\Asign}^\circlearrowleft
_\O
}=\Csign_\O$, if no confusion arises.

\begin{corollary}
\label{Cor:HMviaZg}
Let a simply connected discrete domain $\O$ and $u\in\Int\O$ be such
that $\A{\O}u$ is doubly connected, and $[ab]_\O\ss\pa\O$. Then, for
any nearest-neighbor path $\gamma$ running from $\Csign_\O(u)$ to
$(ba)_\O
$, the following is fulfilled:
\[
\const\cdot\,
\ZRWsign_{\Asign_\O^\gamma}\bigl(\Csign_\O;[ab]_\O\bigr)
\le \omega_\O
\bigl(u;[ab]_\O\bigr) \le\const\cdot\,
\ZRWsign_{\Asign_\O^\gamma}\bigl(\gamma^\mathrm {left}_\mathrm {bd}\cup\Csign_\O\cup
\gamma^\mathrm{right}_\mathrm{bd};[ab]_\O\bigr).
\]
%
\end{corollary}

\begin{pf}
Indeed,
\[
\ZRWsign_{\Asign_\O^\gamma}\bigl(\Csign_\O;[ab]_\O\bigr)
\le
\ZRWsign_{\Asign_\O}\bigl(\Csign_\O;[ab]_\O\bigr)
\le
\ZRWsign_{\Asign_\O^\gamma}\bigl(\gamma^\mathrm {left}_\mathrm{bd}\cup \Csign_\O\cup
\gamma^\mathrm{right}_\mathrm{bd};[ab]_\O\bigr)
\]
due to simple monotonicity properties of the random walk partition
function $\ZRWsign_\O$ with respect to domain $\O$; for example, for
the left bound, one forbids the random walks running from $\Csign_\O$
to $[ab]_\O$ to cross $\gamma$ (still allowing them to touch
$\gamma$ or to run
along it).
\end{pf}

Theorem~\ref{Thm:DblSidedEstimates} [namely, (\ref{BeurlingForZ})] allows
one to estimate both partition functions $\ZRW{\Asign_\O^\gamma
}{\Csign_\O
}{[ab]_\O}$ and $\ZRW{\Asign_\O^\gamma}{\gamma^\mathrm
{left}_\mathrm
{bd}\cup\Csign_\O\cup\gamma^\mathrm{right}_\mathrm{bd}}{[ab]_\O
}$ via
corresponding discrete extremal lengths. We now prove that one can
choose $\gamma$ so that \emph{both} those extremal lengths are comparable
to the extremal length of nearest-neighbor paths connecting $\Csign_\O$
and $[ab]_\O$ in the annulus $\Asign_\O$.

\begin{remark}
\label{Rem:ELdiscr=ELcont=U/Iannulus}
Below we apply Propositions \ref{Prop:ELdiscr=ELcont} and \ref
{Prop:EL=U/I} to a doubly connected discrete domain $\Asign_\O$ and its
inner boundary $\Csign_\O$ instead of a boundary arc $[cd]_\O$ of a
simply connected domain $\O$. It is worth noting that we did not use
any ``topological'' arguments in the proofs of those statements.
\end{remark}

\begin{proposition}
\label{Prop:ELgAsympELag}
Let a simply connected discrete domain $\O$ and $u\in\Int\O$ be such
that $\A{\O}u$ is doubly connected, and $[ab]_\O\ss\pa\O$. Then:
\begin{longlist}[(ii)]
\item[(i)] there exists a nearest-neighbor path $\gamma$ running from
$\Csign_\O$ to $(ba)_\O$ such that
\[
\ELsign_{\Asign_\O^\gamma}\bigl(\Csign_\O;[ab]_\O\bigr)\le2
\ELsign_{\Asign_\O}\bigl(\Csign_\O;[ab]_\O\bigr);
\]

\item[(ii)] for any given $q>1$, either $\EL{\Asign_\O}{\Csign_\O
}{[ab]_\O}< q^2 \EL{\Asign_\O}{\Csign_\O}{\pa\O}$ (i.e., the arc
$[ab]_\O$ is not so far from $u$), or there exists a nearest-neighbor
path $\gamma$ running from $\Csign_\O$ to $(ba)_\O$ such that
\[
\ELsign_{\Asign_\O^\gamma}\bigl(\gamma^\mathrm{left}_\mathrm{bd}\cup
\Csign_\O\cup \gamma^\mathrm{right}_\mathrm{bd};[ab]_\O
\bigr) \ge\bigl(1 - q^{-1}\bigr)^2 \ELsign_{\Asign_\O}
\bigl(\Csign_\O;[ab]_\O\bigr).
\]
\end{longlist}
\end{proposition}

\begin{remark} (i) The constant $2$ in the first estimate is a big
overkill: as can be seen from the proof, both sides are almost equal to
each other for a proper slit~$\gamma$.

(ii) Since discrete and continuous extremal lengths are
uniformly comparable to each other, for any $\O$ and $u$, one has
\[
\ELsign_{\Asign_\O}(\Csign_\O;\pa\O) \asymp
\ELsign_{\Asign_\O^\C}\bigl(\Csign_\O^\C;\pa
\O^\C\bigr) \asymp1.
\]
\end{remark}

\begin{pf} Let $V=V_{(\Asign_\O;[ab]_\O,\Csign_\O)}\dvtx\Asign
_\O\to
[0;1]$ be the unique discrete harmonic function such that $V\equiv0$
on $[ab]_\O$, $V\equiv1$ on $\Csign_\O$, and $V$ satisfies Neumann
boundary conditions on $\pa\O\setminus[ab]_\O$. Recall that
Proposition~\ref{Prop:EL=U/I} (see also Remark~\ref
{Rem:ELdiscr=ELcont=U/Iannulus}) says
\begin{eqnarray*}
\bigl( \ELsign_{\Asign_\O}\bigl(\Csign_\O;[ab]_\O
\bigr)\bigr)^{-1}=I(V) & =& \sum_{x\in
[ab]_\O}
\mathrm{w}_{xx_\mathrm{int}}V(x_\mathrm{int})
\\
&= &\sum_{(vv')\
\mathrm{in}\ {\Asign_\O^\gamma}}\mathrm{w}_{vv'} \bigl(V
\bigl(v'\bigr) - V(v)\bigr)^2. 
\end{eqnarray*}

(i) Let $V^*$ denote a harmonic conjugate function to $V$
[see (\ref{V*def}), (\ref{V*harm}) and Remark~\ref{Rem:DualFct}] which
is defined on the universal cover ${\Asign}^\circlearrowleft_\O$ of
$\Asign_\O$. Tracking its increment along $[ab]_\O$, one easily
concludes that $V^*$ has an additive monodromy $I(V)$ when passing
around $\Csign_\O$ counterclockwise. Moreover, as $V\in[0;1]$
everywhere in $\Asign_\O$, the boundary values of $V^*$ increases when
going counterclockwise along $\Csign_\O$, as well as along $\pa\O$
(recall that $V^*$ satisfies Neumann boundary conditions on $\Csign_\O$
and $[ab]_\O$).

Let an additive constant in definition of $V^*$ be chosen so that
$V^*\equiv0$ on $\pa\O\setminus[ab]_\O$ (on some sheet of ${\Asign
}^\circlearrowleft_\O$). Then, there exists a nonself-intersecting
nearest-neighbor path $\gamma$ running from $\Csign_\O$ to $\pa\O
\setminus
[ab]_\O$ in ${\Asign}^\circlearrowleft_\O$ which separates nonnegative
(to the left of $\gamma$) and nonpositive (to the right of $\gamma
$) values of
$V^*$. We cut $\Asign_\O$ along $\gamma$ and choose a branch of
$V^*$ in
$\Asign_\O^\gamma$ so that
\begin{eqnarray*} V^*&\le&0 \qquad\mbox{at faces touching }
\gamma^\mathrm{right}_\mathrm {bd},\qquad V^*\ge I(V) \qquad\mbox{at faces
touching } \gamma^\mathrm{left}_\mathrm{bd} ,
\\
V^*
&\equiv&0 \qquad \mbox{at faces touching } [da]_\O, \qquad V^*\equiv I(V) \qquad \mbox
{at faces touching } [bd]_\O
\end{eqnarray*}
(recall that $V^*$ satisfies Neumann boundary conditions on $[ab]_\O$).
Putting on dual edges $(ff')=(vv')^*$ of $\Asign_\O^\gamma$ a
discrete metric
\[
g_{ff'}:=\bigl|V^*\bigl(f'\bigr)-V^*(f)\bigr|=
\mathrm{w}_{vv'}\bigl|V\bigl(v'\bigr)-V(v)\bigr| ,
\]
one obtains the following estimate for the \emph{dual} discrete length
$\ELsign^*$ (see Remark~\ref{Rem:DualFct}) between opposite sides
$\gamma
^\mathrm{right}\cup[da]_\O$ and $[bd]_\O\cup\gamma^\mathrm
{left}$ of $\Asign
_\O^\gamma$:
\[
\ELsign^*\ge\frac{[I(V)]^2}{\sum_{(vv')\ \mathrm{in}\ {\Asign_\O
^\gamma
}}\mathrm{w}_{vv'}|V(v') - V(v)|^2}\ge\frac
{[I(V)]^2}{2I(V)}= \frac{1}{2}I(V)
\]
(the constant $2$ is a big overkill, since each edge of $\Asign_\O$
except $\gamma$ is counted once in $\Asign_\O^\gamma$, and only those
constituting $\gamma$ are counted twice). Therefore,
\[
\ELsign_{\Asign_\O^\gamma}\bigl(\Csign_\O;[ab]_\O\bigr) =
\bigl(\ELsign^*\bigr)^{-1} \le 2\bigl[I(V)\bigr]^{-1}=2
\ELsign_{\Asign_\O}\bigl(\Csign_\O;[ab]_\O\bigr).
\]

(ii) Let $d\in\pa\O\setminus[ab]_\O$ be a boundary vertex
where $V$ attains its maximum on $\pa\O$ (recall that $V\equiv0$ on
$[ab]_\O$). If $V(d)<1 - q^{-1}$, then the metric $g_{vv'}:=|V(v')
- V(v)|$ [which is extremal for the family $(\Asign_\O;\Csign_\O
\leftrightarrow[ab]_\O)$; see Proposition~\ref{Prop:EL=U/I}] provides
an estimate
\[
\ELsign_{\Asign_\O}(\Csign_\O;\pa\O) >\frac{q^{-2}}{I(V)}=q^{-2}
\ELsign_{\Asign_\O}\bigl(\Csign_\O;[ab]_\O\bigr) .
\]
If $V(d)\ge1 - q^{-1}$, let $\gamma$ denote a nearest-neighbor path
running from $d$ to $\Csign_\O$ such that $V\ge1 - q^{-1}$ along
this path ($\gamma$ exists due to the maximum principle). Then the same
metric as above (we assign zero weights to all edges constituting
$\gamma
^\mathrm{left},\gamma^\mathrm{right}$ and corresponding boundary
ones) gives
\begin{eqnarray*}
\ELsign_{\Asign_\O^\gamma}\bigl(\gamma^\mathrm{left}_\mathrm{bd}\cup
\Csign_\O\cup \gamma^\mathrm{right}_\mathrm{bd};[ab]_\O
\bigr) &\ge&\frac{(1 -
q^{-1})^2}{I(V)}\\
&=&\bigl(1 - q^{-1}\bigr)^2
\ELsign_{\Asign_\O}\bigl(\Csign_\O;[ab]_\O\bigr) .
\end{eqnarray*}
\upqed
\end{pf}

Combining estimates given above, we are now able to prove a uniform
double-sided estimate relating the logarithm of the discrete harmonic
measure $\hm\O{u}{[ab]_\O}$ in a simply connected domain $\O$ and the
discrete extremal length $\EL{\Asign_\O}{\Csign_\O}{[ab]_\O}$ in the
annulus-type domain $\Asign_\O(u)$.

\begin{theorem}
\label{Thm:LogHm=ELdiscr}
Let a simply connected discrete domain $\O$ and $u\in\Int\O$ be such
that $\A{\O}u$ is doubly connected, and $[ab]_\O\ss\pa\O$. Then
%
\begin{equation}
\label{Hm=ELdiscr} \log\bigl(1 + \bigl( \omega_\O\bigl(u;[ab]_\O
\bigr)\bigr)^{-1}\bigr)\asymp \ELsign_{\Asign_\O
(u)}\bigl(
\Csign_\O (u);[ab]_\O\bigr),
\end{equation}
with constants independent of $\O$, $u$ and $[ab]_\O$.
\end{theorem}

\begin{pf} Let $\ELsign:=\EL{\Asign_\O}{\Csign_\O}{[ab]_\O}$ and
$\omega
:=\hm\O{u}{[ab]_\O}$. Corollary~\ref{Cor:HMviaZg}, Theorem~\ref
{Thm:DblSidedEstimates} and Proposition~\ref{Prop:ELgAsympELag} provide
us the following diagram (for some proper discrete cross-cuts $\gamma$
which can be different for lower and upper bounds):
\[
\begin{array} {ccc} \const\cdot\,
\ZRWsign_{\Asign_\O^\gamma}\bigl(\Csign_\O;[ab]_\O\bigr)
& \le \omega\le&
\ZRWsign_{\Asign_\O^\gamma}\bigl(\gamma^\mathrm{left}_\mathrm{bd}
\cup \Csign_\O\cup \gamma^\mathrm{right}_\mathrm{bd};[ab]_\O\bigr)
\\
\updownarrow& & \updownarrow
\\
2\ELsign
\ge \ELsign_{\Asign_\O^\gamma}\bigl(\Csign_\O;[ab]_\O\bigr)
&& \ELsign_{\Asign_\O
^\gamma}\bigl(\gamma^\mathrm{left}_\mathrm{bd}
\cup\Csign_\O\cup \gamma^\mathrm {right}_\mathrm{bd};[ab]_\O
\bigr) \ge\frac{1}{2} \ELsign
\end{array} %
\]
[the last inequality holds true if $\ELsign\ge\lambda_0$, where
$\lambda_0$ is
some absolute constant: recall that $\EL{\Asign_\O}{\Csign_\O}{\pa
\O
}\asymp1$ for all $\O$ and $u$]. Above, the arrows ``$ \updownarrow
$'' mean double-sided estimates of $\ZRWsign_{\Asign_\O^\gamma}$ via
$\ELsign_{\Asign_\O^\gamma}$ given by Theorem~\ref{Thm:DblSidedEstimates}.
Recall that it is \emph{inverse} monotone: an upper bound for $\ELsign
_{\Asign_\O^\gamma}$ gives a lower bound for $\ZRWsign_{\Asign_\O
^\gamma}$ and
vice versa.

In particular, if $\ELsign\ge\lambda_0$, condition (\ref
{EstimEquiv}) holds
for both (right, and therefore, left) columns. Thus, in this case, one
can replace both ``$ \updownarrow$'' by (\ref{BeurlingForZ}),
arriving at $\log\omega\asymp-\ELsign$.
If $\ELsign<\lambda_0$, then the left column gives $\omega\ge
\const$, and both
sides of~(\ref{Hm=ELdiscr}) are uniformly comparable to $1$ [note that
$\ELsign$ is uniformly bounded below by $\EL{\Asign_\O}{\Csign_\O
}{\pa\O
}\asymp1$].
\end{pf}

\begin{corollary}
\label{Cor:LogHmDisc=LogHmCont}
Let $\O$ be a simply connected domain, $u\in\Int\O$ and {$[ab]_\O
\in\pa
\O$}. Denote $\omega_\mathrm{disc}:=\hm\O{u}{[ab]_\O}$ and
$\omega_\mathrm
{cont}:=\hm{\O^\C}{u}{[ab]_\O^\C}$. Then
\[
\log\bigl(1+\omega_\mathrm{disc}^{-1}\bigr)\asymp
\log\bigl(1+\omega_\mathrm {cont}^{-1}\bigr)
\]
%
with some uniform (i.e., independent of $\O,u,a,b$) constants.
\end{corollary}

\begin{pf}
First, let us assume that $\A\O{u}$ is doubly connected, so $\O$ and
$u$ fit the setup of Theorem~\ref{Thm:LogHm=ELdiscr}. Let $\ELsign
_\mathrm{disc}:=\EL{\Asign_\O}{\Csign_\O}{[ab]_\O}$ and $\ELsign
_\mathrm
{cont}:=\EL{\Asign_\O^\C}{\Csign_\O^\C}{[ab]_\O^\C}$ be its continuous
counterpart. Recall that $\ELsign_\mathrm{disc}\asymp\ELsign
_\mathrm
{cont}$ due to Proposition~\ref{Prop:ELdiscr=ELcont} (and Remark~\ref
{Rem:ELdiscr=ELcont=U/Iannulus}). Then
\[
\log\bigl(1 + \omega_\mathrm{disc}^{-1}\bigr)\asymp
\ELsign_\mathrm {discr}\asymp \ELsign_\mathrm{cont}\asymp\log\bigl(1
+ \omega_\mathrm{cont}^{-1}\bigr),
\]
where the last estimate is an easy corollary of the classical estimates
for harmonic measure via extremal lengths; for example, see \cite{GM05},
Theorem~5.2.

If $\A\O{u}$ is \emph{not} doubly connected, then $u$ belongs to a face
touching $\pa\O$. If $u$ shares a face with $[ab]_\O$, then $\omega
_\mathrm
{disc}\ge\const$ and $\omega_\mathrm{cont}\ge\const$ as well.

Thus, without loss of generality, we may assume that both $\omega
_\mathrm
{disc}$ and $\omega_\mathrm{cont}$ are uniformly bounded away from
$1$, and
there exists a connected (and simply connected) component of $\Int\A
\O
{u}$ whose boundary contains the whole arc $[ab]_\O$. Let $\Asign_\O'$
denote this component of $\Asign_\O$ and $\Csign_\O^{ \prime}\ss
\pa
\Asign_\O'$ be the corresponding part of $\Csign_\O$ slightly enlarged
so that it includes two nearby boundary points of $\pa\O$; see
Figure~\ref{Fig: annulus}(B). Further, let $\ELsign_\mathrm
{disc}':=\EL
{\Asign_\O'}{\Csign_\O^{ \prime}}{[ab]_\O}$ and $\ELsign_\mathrm
{cont}':=\EL{\Asign_\O^{\prime\C}}{\Csign_\O^{ \prime\C
}}{[ab]_\O
^\C}$ denote its continuous counterpart. It is easy to see that one
still has
\[
\omega_\mathrm{discr}\asymp
\ZRWsign_{\O}\bigl(u;[ab]_\O\bigr)
\asymp
\ZRWsign_{\Asign _\O'}\bigl(\Csign _\O\cap\pa
\Asign'_\O;[ab]_\O\bigr)
\asymp
\ZRWsign_{\Asign_\O'}\bigl(\Csign _\O^{
\prime};[ab]_\O\bigr)
\]
the proof of Lemma~\ref{Lemma:HmAsympZa} works well without any
changes, and replacing $\Csign_\O\cap\pa\Asign'_\O$ by $\Csign
'_\O$
costs no more than an absolute multiplicative constant). Applying (\ref
{BeurlingForZ}) and Proposition~\ref{Prop:ELdiscr=ELcont}, one obtains
\[
\log\omega_\mathrm{disc}\asymp-\ELsign'_\mathrm{discr}
\asymp -\ELsign '_\mathrm{cont}\asymp\log
\omega_\mathrm{cont} ,
\]
[to prove the last estimate, e.g., draw a circle ${c}_u\ss\O^\C$
of radius $\frac{1}{2}r_u\asymp\dd\O{u}$ around $u$, then $-\log
\omega
_\mathrm{cont}\asymp\EL{\O^\C}{{c}_u}{[ab]_\O^\C}\asymp\ELsign
'_\mathrm
{cont}$].
\end{pf}

\begin{appendix}
\section*{Appendix}\label{app}
\setcounter{equation}{0}

In order to make the presentation self-contained, in this appendix we
provide proofs of all the statements from Section~\ref
{SubSect:BasicFacts} based on 
properties \ref{PropertyS}, \ref{PropertyT} of the random walk (\ref
{RWtransprobaDef}) on $\G$.
We begin with a slightly weaker version of Lemma~\ref
{Lemma:AroundAnnulus}, then prove Lemma~\ref{Lemma:AroundAnnulus}
itself and deduce all the other statements from these lemmas.

\begin{lemmaa} \label{Lemma:App_AroundAnnulus}
There exist constants $\t_0=\t_0(\varpi_0,\eta_0,\varkappa_0)>1$ and
$\ve_0=\ve_0(\varpi_0,\eta_0,\varkappa_0)>0$ such that,
for any two vertices $v,w\in\Gamma$, $v\ne w$, the probability of the
event that the random walk (\ref{RWtransprobaDef}) started at $v$ makes
a full turn around $w$ in a given direction (clockwise or
counterclockwise) staying in $A(w,\t_0^{-1}|v - w|, \t_0|v - w|)$
is at least $\ve_0$.
\end{lemmaa}

\begin{pf} Denote $v_0:=v$. We intend to ``drive'' the trajectory of
the random walk using a finite sequence of the following ``moves''
based on Property~\ref{PropertyS} [see also Figure~\ref{Fig: RW-in-annulus}(A)]:
\begin{itemize}
\item if $v_k$ is the current position of the random walk, apply \ref
{PropertyS} to a disc of radius $(\varkappa_0\nu_0 + 1)^{-1}|v_k -
w|$ centered at $v_k$ and the interval of directions
\[
I:=\bigl[\arg(v_k - w)+\tfrac{1}{2}\eta_0 ;
\arg(v_k - w)+\pi- \tfrac{1}{2}\eta_0\bigr],
\]
and denote by $v_{k+1}\in\pa\BBsign^\G_{(\varkappa_0\nu
_0+1)^{-1}|v_k-w|}$ the corresponding terminal vertex.
\end{itemize}
Using Remark~\ref{Rem:RvBound}, one can find two constants \mbox
{$\theta
_0=\theta_0(\varpi_0,\eta_0,\varkappa_0)>0$} and $\alpha_0=\alpha
_0(\varpi_0,\eta_0,\varkappa_0)>1$ such that, for all $k$, $\arg
(v_{k+1} - w)-\arg(v_k - w)\ge\theta_0$ and the random walk does
not leave the annulus $A(w,\alpha_0^{-1}|v_k - w|,\alpha_0|v_k -
w|)$ during the $k$th ``move'' described above. Thus the claim holds
true with \mbox{$\t_0:=\alpha_0^{N_0}$} and $\ve_0:=c_0^{N_0}$, where
\mbox{$N_0:={\lfloor2\pi/\theta_0\rfloor+1}$} is the maximal
number of
moves needed to perform the full turn.
\end{pf}

\begin{figure}

\includegraphics{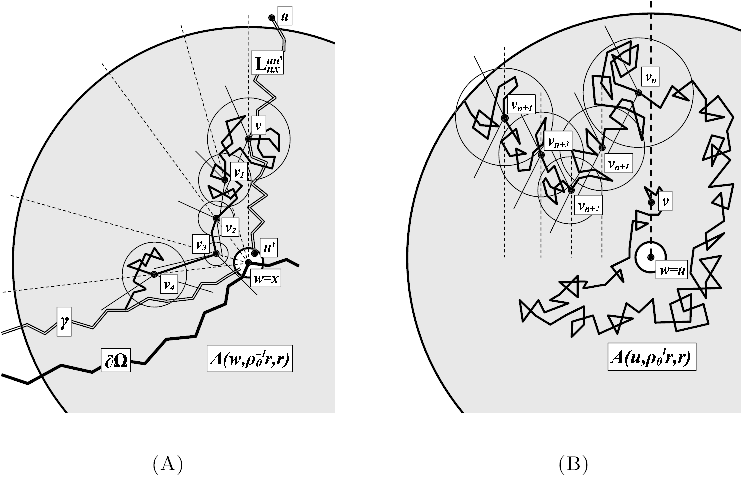}

\caption{\textup{(A)} A schematic drawing of a sequence of ``moves'' used in
the proofs
of Lemma \protect\ref{Lemma:App_AroundAnnulus} and Lemma \protect
\ref
{Lemma:BoundaryDecay}. For each vertex $v_k$, the corresponding disc
around $v_k$ and the interval of directions are shown. Applying
property \protect\ref{PropertyS} to these discs step by step, one can
``drive''
a trajectory of the random walk around $w$, uniformly with respect to
the local sizes $r_{v_k}$ (e.g., $v_4$ is a neighboring vertex to $v_3$
on the picture). For the proof of Lemma \protect\ref
{Lemma:BoundaryDecay}, the
paths $\rL^{uu'}_{ux}$, $\gamma$ and a part of $\pa\O$ are shown: the
random walk trajectory constructed in this way must hit $\gamma$
before $\pa
\O$.
\textup{(B)} A schematic drawing of an additional sequence of ``moves'' used in
the proof of Lemma \protect\ref{Lemma:AroundAnnulus}. It may happen
that the
random walk trajectory constructed in this way does not disconnect two
boundary components of $A(u,\rho_0^{-1}r,r)$ and does not intersect a
path $\gamma$ that crosses $A(u,\rho_0^{-1}r,r)$. Nonetheless, the
union of such a ``counterclockwise'' trajectory and a similar
``clockwise'' one must intersect $\gamma$.}
\label{Fig: RW-in-annulus}
\end{figure}

%
%
%

\begin{pf*}{Proof of Lemma~\ref{Lemma:AroundAnnulus}}
Let
$\rho_0:=(\varkappa_0\n_0 + 1)\t_0^2$. If $r':=\rho_0^{-1}r\le r_u$,
then there is nothing to prove as $\gamma$ should start at $u$ which is
the unique vertex inside of $A(u,r',r)$. Thus it is sufficient to
consider the case $r'>r_u$. In this case, Remark~\ref{Rem:RvBound}
implies that there is no edge of $\Gamma$ crossing the annulus $A(u,\t
_0r',(\varkappa_0\nu_0 + 1)\t_0 r')$. Let $v$ denote the first vertex
visited of the random walk (\ref{RWtransprobaDef}) traveling across the
annulus $A(u,r',r)$. Thus it is sufficient to prove that, being
re-started at~$v$, the random walk (\ref{RWtransprobaDef}) hits a
cross-cut $\gamma$ before exiting $A(u,r',r)$ with a probability
uniformly bounded below. Similar to the proof of Lemma~\ref
{Lemma:App_AroundAnnulus}, we intend to ``drive'' the trajectory of the
random walk using a finite sequence of ``moves'' provided by \ref
{PropertyS} so that it necessarily intersects $\gamma$:
\begin{itemize}
\item first, we follow the proof of Lemma~\ref{Lemma:App_AroundAnnulus}
(with $w:=u$) and perform $n\le N_0$ moves around $u$ in the
counterclockwise direction so that the random walk remains in
$A(u,r',r)$, and its terminal vertex $v_n$ satisfies
\[
\arg(v_n - u)-\arg(v_0 - u)\ge2\pi;
\]
\item second, we continue the trajectory by performing yet another
finite sequence of similar moves in the fixed range of directions
\[
I:=\bigl[\arg(v_0 - u)-\tfrac{1}{2}(\pi- \eta_0) ;
\arg(v_0 - u)+\tfrac{1}{2}(\pi- \eta_0)\bigr]
\]
until the trajectory hits the outer boundary of $(A,r',r)$; see
Figure~\ref{Fig: RW-in-annulus}(B).
\end{itemize}
Note that the number of moves needed to perform the second part is
uniformly bounded by some constant $M_0$: the distance from the line
passing through $v_0$ and $u$ becomes comparable to $r$ after the first
move of the second part and then grows exponentially. In principle, it
might happen that a ``counterclockwise'' trajectory constructed above
does not hit the cross-cut $\gamma$; see Figure~\ref{Fig:
RW-in-annulus}(B). Nonetheless, if this happens (for some trajectory),
then all the similar ``clockwise'' trajectories must hit $\gamma$ for
topological reasons. Thus the result follows with $\delta_0:=c_0^{N_0+M_0}$.
\end{pf*}

\begin{pf*}{Proof of Proposition~\ref{Proposition:Harnack}}
Denote by $v_\mathrm{max}$ and $v_\mathrm{min}$ the vertices in $\Int
\BBsign^\G_{r}(u)$ where $H$ attains its maximal and minimal values,
respectively. First, let $\rho\ge\rho_0$, where $\rho_0=\rho
_0(\varpi
_0,\eta_0,\varkappa_0)$ is the constant from Lemma~\ref
{Lemma:AroundAnnulus}. Since $H$ is a discrete harmonic function, there
exists a path $\gamma$ running from $v_\mathrm{max}$ to $\pa\BBsign
^\G
_{\rho_0r}(u)$ such that $H( \cdot)\ge H(v_\mathrm{max})$ along this
path. Applying Lemma~\ref{Lemma:AroundAnnulus}, one easily obtains
\[
H(v_\mathrm{min})\ge\delta_0\cdot H(v_\mathrm{max}),
\]
which gives the desired estimate for all $\rho\ge\rho_0$ with
$c(\rho
)=c(\rho_0)=\delta_0$. To obtain the result for $\rho<\rho_0$,
note that
the path joining $v_\mathrm{max}$ and $v_\mathrm{min}$ in $\Int
\BBsign
^\G_r(u)$ can be covered by a uniformly bounded number $N=N(\rho)$ of
discrete discs $\Int\BBsign^\G_{r'}(v_k)$ with $v_k\in\Int\BBsign
^\G
_r(u)$ and $r':=\rho_0^{-1}(\rho- 1)r$. Since $\BBsign^\G_{\rho
_0r'}(v_k)\subset\BBsign^\G_{\rho r}(u)$ and the values of $H$ at
neighboring vertices belonging to consecutive discs are uniformly
comparable with the constant $\varpi_0^2$, one can apply the already
proven estimate in each of these discs and arrive at the inequality
$H(v_\mathrm{min})\ge c(\rho)H(v_\mathrm{max})$ with the constant
$c(\rho)=(\varpi_0^2\delta_0)^{N(\rho)}$.
\end{pf*}

\begin{pf*}{Proof of Lemma~\ref{Lemma:GreenInDiscs}}
For
$u\in\G$ and $R>r>0$, let
\[
M(u,r,R):=\max_{v\in\pa\BBsign^\G_r(u)} G_{\BBsign^\G_R(u)}(v;u).
\]
%
It is easy to see that
%
\begin{equation}\qquad
\label{xGRviaMrRBound} G_{\BBsign^\G_{\rho r}(u)}\bigl(v';u\bigr) \ge
\delta_0M(u,r,\rho r) \qquad\mbox {for all } v'\in\Int
\BBsign_{r'}^\G(u), r':=\rho_0^{-1}r.
\end{equation}
Indeed, the maximum principle implies that $G_{\BBsign^\G_{\rho
r}(u)}(\cdot;u)\ge M(u,r,\rho r)$ along some nearest-neighbor path
$\gamma$ starting at some $v\in\pa\BBsign^\G_r(u)$ and going to
$u$. As
$\gamma$ crosses the annulus $A(u,r',r)$, estimate (\ref
{xGRviaMrRBound}) follows from Lemma~\ref{Lemma:AroundAnnulus}. It is
easy to see that the uniform upper bound $M(u,r,\rho r)\le c_2(\rho)$
now follows from~(\ref{xGRviaMrRBound}), estimate (\ref{SumRvBound})
and the upper bound in (\ref{BoundPropertyT}).

The next step is to prove that $M(u,(2C_0)^{-1}R,R)$ is uniformly
bounded from below for $u\in\G$ and $R>r_u$, where $C_0$ is the
constant from (\ref{BoundPropertyT}). If $r:=(2C_0)^{-1}R\le r_u$, then
there is nothing to prove as $G_{\BBsign^\G_R(u)}(v;u)\ge\mu
_v^{-1}\mathrm{w}_{vu}\mu_u^{-1}\ge\varpi_0^3$ for all vertices
$v\sim
u$ lying on $\pa\BBsign^\G_r(u)\cap\Int\BBsign^\G_R(u)\ne
\varnothing$.
For $r>r_u$, the maximum principle implies
\[
G_{\BBsign^\G_R(u)}(v;u) \le %
\cases{ M(u,r,R), &\quad  $v\in\Int
\BBsign^\G_R(u)\setminus\Int\BBsign^\G
_r(u)$;
\vspace*{2pt}\cr
M(u,r,R)+G_{\BBsign^\G_r(u)}(v;u), &\quad  $v\in\Int
\BBsign^\G_r(u)$.} %
\]
Therefore,
\[
C_0^{-1}R^2\le\sum
_{v\in\Int\BBsign^\G_R(u)}r_v^2G_{\BBsign
^\G_R(u)}(v;u) \le
M(u,r,R)\cdot\sum_{v\in\Int\BBsign^\G
_R(u)}r_v^2
+ C_0r^2,
\]
which can be rewritten as $M(u,r,R)\ge\frac{3}{4}C_0^{-1}R^2\cdot
[\sum_{v\in\Int\BBsign^\G_R(u)}r_v^2]^{-1}$. Due to (\ref{SumRvBound}), the
latter quantity is uniformly bounded away from $0$. Taking into account
(\ref{xGRviaMrRBound}), we arrive at the uniform estimate
\[
G_{\BBsign^\G_{R}(u)}\bigl(v';u\bigr) \ge c_1^{(0)}\qquad
\mbox{for all } v'\in \Int\BBsign_{r'}^\G(u),
r':=(2C_0\rho_0)^{-1}R
\]
with some constant $c_1^{(0)}>0$ (note that this estimate remains true
if $R\le r_u$). Thus we have shown that $G_{\BBsign^\G_{\rho
r}}(v';u)\ge c_1^{(0)}$ for all $v'\in\Int\BBsign^\G_{r}(u)$ provided
that $\rho\ge2C_0\rho_0$.

The case $\rho<2C_0\rho_0$ can now be handled similarly to the proof of
Proposition~\ref{Proposition:Harnack}. For $v\in\Int\BBsign^\G_r(u)$,
let $\gamma$ be a nearest-neighbor path connecting $v$ to $u$ inside of
$\BBsign^\G_r(u)$, and let $v'$ denote the first vertex of $\gamma$
belonging to $\Int\BBsign^\G_{r'}(u)$, where $r':=(2C_0\rho
_0)^{-1}\cdot
\rho r$. The portion of $\gamma$ joining $v$ and the vertex just before
$v'$ can be covered by a uniformly bounded number $N=N(\rho)$ of
discrete discs $\Int\BBsign_{r''}^\G(v_k)$, where $r'':=\min\{
(2C_0\rho
_0)^{-1},(1-\rho^{-1})\}\cdot r$. Since $\Int\BBsign_{\rho r''}^\G
(v_k)\subset\Int\BBsign_{\rho r}^\G(u)\setminus\{u\}$,
Proposition~\ref{Proposition:Harnack} applied to each of these discs yields
\begin{eqnarray*}
G_{\BBsign^\G_{\rho r}}(v;u) &\ge&\bigl(\varpi_0^2 c(\rho)
\bigr)^{N(\rho
)}\cdot G_{\BBsign^\G_{\rho r}}\bigl(v';u\bigr) \ge
c_1(\rho)\\
&:=&\bigl(\varpi_0^2 c(\rho )
\bigr)^{N(\rho)}\cdot c_1^{(0)}.
\end{eqnarray*}
\upqed\end{pf*}

\begin{pf*}{Proof of Lemma~\ref{Lemma:b0-weakBeurling}}
To derive the first estimate, set $r:=\dist(u;\partial\O)$, and note
that the probability of the event that the random walk started at $u$
crosses an annulus $A(u,\rho_0^{s-1}r,\rho_0^{s} r)$, $s=1,\ldots
,\lfloor
\log(r^{-1}\dist_{\O}(u;E))/\log\rho_0 \rfloor$, is bounded from above
by $1-\delta_0$. To prove the second estimate, set $r:=\diam E$, and
consider crossings of the annuli $A(x,\rho_0^{s-1}r,\rho_0^sr)$
centered at a fixed vertex $x\in E$.
\end{pf*}

\begin{pf*}{Proof of Lemma~\ref{Lemma:BoundaryDecay}}
Recall
that $r'=\rho_0^{-1}r$.
The proof is divided into two steps. First, it follows from Lemma~\ref
{Lemma:App_AroundAnnulus} that all values of $H$ in the
$r'$-neighborhood of $x$ in $\Omega$ are bounded from above by $\delta
_0^{-1}\max_{v\in\rL_{ux}^{uu'}}H(v)$. Indeed, let $v_\mathrm
{max}\in\pa\BBsign^\O_{r'}(x)$ be the vertex where $H$ attains its
maximal value in $\BBsign^\O_{r'}(x)$. Then $H( \cdot)\ge
H(v_\mathrm
{max})$ along some path $\gamma$ running from $v_\mathrm{max}$ to
$\pa
\BBsign^\O_r\setminus\pa\O$. If $\gamma$ intersects $\rL_{ux}^{uu'}$,
then there is nothing to prove. Otherwise, there are three mutually
disjoint paths crossing the annulus $A(x,r',r)$: $\rL
_{ux}^{uu'}$, $\gamma$ and the corresponding part of $\pa\O$ which has
to cross $A(x,r',r)$ since $\Omega$ is simply connected. Let us assume
that these paths are ordered counterclockwise (the other case is
similar). Due to Remark~\ref{Rem:RvBound}, there exists a vertex
$u''\in
\rL_{ux}^{uu'}$ such that $\t_0r'\le|u''-x|\le\t_0^{-1}r$, where
$\t
_0=(\rho_0/(\varkappa_0\n_0 + 1))^{1/2}$. For topological reasons,
each of the random walk trajectories constructed in Lemma~\ref
{Lemma:App_AroundAnnulus}, started at $v=u''$ and running in $A(x,\t
_0^{-1}|u'' - x|,\t_0|u'' - x|)\subset A(x,r',r)$ in the
counterclockwise direction, must hit $\gamma$ before $\pa\O$ [which
must happen before it makes the full turn and reaches the path $\rL
_{ux}^{uu'}\subset\Int\O$ again; see Figure~\ref{Fig: RW-in-annulus}(A)].
Note also that those trajectories cannot hit $\pa\O$ during first
``moves'' due to (\ref{DistToBdAlongL}). Therefore, Lemma~\ref
{Lemma:App_AroundAnnulus} gives
\[
\max_{v\in\rL_{ux}^{uu'}}H(v) \ge H\bigl(u''
\bigr) \ge\delta _0H(v_\mathrm{max}) = \delta_0\max
_{v\in\pa\BBsign^\O
_{r'}(x)}H(v).
\]
Second, similar to the proof of Lemma~\ref{Lemma:b0-weakBeurling},
one can easily derive from Lemma~\ref{Lemma:AroundAnnulus} the
following uniform estimate:
\[
H\bigl(v'\bigr) \le\bigl[ \rho_0\cdot\bigl|v'
- x\bigr|/r' \bigr]^{\beta_0}\cdot\max_{v\in\pa\BBsign^\O_{r'}(x)}H(v)
\]
for all $v'\in\BBsign^\O_{r'}(x)$. Being combined, these two
inequalities yield the claim.
\end{pf*}

\begin{pf*}{Proof of Lemma~\ref{Lemma:GreenEstimatesViaG}}
The lower bound is trivial, as $\BBsign^\G_r(u)\subset\O$ and the Green
function $G_\O$ is monotone with respect to $\O$. To prove the upper
bound, recall that $R=\rho_0^{2n_0}r$, and denote by $\O'$ the minimal
simply connected domain $\O'\supset\O$ such that
\[
\Int\O'\supset\Int\O\cup\Int\BBsign^\G_{R'}(u),\qquad
\mbox{where } R'=\rho_0^{n_0}r.
\]
Note that $\pa\O'\cap\pa\BBsign^\G_{R'}(u)\ne\varnothing$ since
$\O$ is
simply connected and $\dist(u;\pa\O)=r<R'$. It follows from
Lemma~\ref
{Lemma:AroundAnnulus} that
\[
G_{\O'}(v;u)\le G_{\BBsign^\G_{R}(u)}(v;u)+(1-\delta_0)^{n_0}
\cdot \max_{v'\in\pa\BBsign^\G_{R'}(u)}G_{\O'}\bigl(v';u\bigr)
\]
for any $v\in\Int\BBsign^\G_r(u)$. Indeed, if the random walk started
at $v$ reaches $\pa\BBsign^\G_{R}(u)$ before hitting $\pa\O'$,
then it
has the chance $(1 - \delta_0)^{n_0}$ to hit $\pa\O'$ before coming
back to $\pa\BBsign^\G_{R'}(u)$. Moreover, since $G( \cdot;u)\ge
\max_{v'\in\pa\BBsign^\G_{R}(u)}G_{\O'}(v';u)$ along some
path $\gamma$ running from $\pa\BBsign^\G_{R'}(u)$ to $u$ (this
follows from the maximum principle), Lemma~\ref{Lemma:AroundAnnulus}
also implies
\[
G_{\O'}(v;u)\ge\bigl(1-(1 - \delta_0)^{n_0}\bigr)
\cdot\max_{v'\in\pa
\BBsign^\G_{R'}(u)}G_{\O'}\bigl(v';u\bigr)
\]
[indeed, the probability of the event that the random walk started at
$v$ hits $\gamma$ before exiting $\BBsign^\G_{R'}(u)\subset\O'$ is at
least $1-(1 - \delta_0)^{n_0}$]. Therefore,
\[
G_{\O'}(v;u) \le\biggl[1-\frac{(1 - \delta_0)^{n_0}}{1-(1 - \delta
_0)^{n_0}}\biggr]^{-1}G_{\BBsign^\G_{R}(u)}(v;u)
\le2G_{\BBsign^\G
_{R}(u)}(v;u),
\]
and we complete the proof by noting that $G_\O(v;u)\le G_{\O'}(v,u)$
since $\O\subset\O'$.
\end{pf*}
\end{appendix}

\section*{Acknowledgments}
The author would like to thank Hugo Duminil-Copin for several fruitful
and motivating discussions, and many other colleagues, particularly
Omer Angel, Ori Gurel-Gurevich, Cl\'ement Hongler, Konstantin Izyurov,
Antti Kemppainen, Eugenia Malinnikova, Asaf Nachmias, Marianna
Russkikh, Stanislav Smirnov, Mikhail Sodin and the unknown referee, for
helpful comments and remarks.

Some parts of this paper were written at the
IH\'ES, Bures-sur-Yvette (October 2010), CRM, Bellaterra (May 2011) and
Universit\'e de Gen\`eve (October 2012). The author is grateful to
these institutes for their hospitality.

%

\printaddresses

\begin{thebibliography}{17}

\bibitem{Ahl73}
%
\begin{bbook}[mr]
\bauthor{\bsnm{Ahlfors},~\bfnm{Lars~V.}\binits{L.~V.}}
(\byear{1973}).
\btitle{Conformal Invariants: Topics in Geometric Function Theory}.
\bpublisher{McGraw-Hill Book},
\blocation{New York}.
\bid{mr={0357743}}
\end{bbook}
%
\bptok{imsref}%
\endbibitem

\bibitem{ABGGN}
%
\begin{bmisc}[author]
\bauthor{\bsnm{Angel},~\bfnm{Omer}\binits{O.}},
\bauthor{\bsnm{Barlow},~\bfnm{Martin}\binits{M.}},
\bauthor{\bsnm{Gurel-Gurevich},~\bfnm{Ori}\binits{O.}} \AND
\bauthor{\bsnm{Nachmias},~\bfnm{Asaf}\binits{A.}}
(\byear{2013}).
\bhowpublished{Boundaries of planar graphs, via circle packings.
Preprint. Available at \arxivurl{arXiv:1311.3363} [math.PR]}.
\end{bmisc}
%
\bptok{imsref}%
\endbibitem

\bibitem{BS01}
%
\begin{barticle}[mr]
\bauthor{\bsnm{Benjamini},~\bfnm{Itai}\binits{I.}} \AND
\bauthor{\bsnm{Schramm},~\bfnm{Oded}\binits{O.}}
(\byear{2001}).
\btitle{Recurrence of distributional limits of finite planar graphs}.
\bjournal{Electron. J. Probab.}
\bvolume{6}
\bpages{13 pp. (electronic)}.
\bid{doi={10.1214/EJP.v6-96}, issn={1083-6489}, mr={1873300}}
\end{barticle}
%
\bptok{imsref}%
\endbibitem

\bibitem{CDCH13}
%
\begin{bmisc}[author]
\bauthor{\bsnm{Chelkak},~\bfnm{Dmitry}\binits{D.}},
\bauthor{\bsnm{Duminil-Copin},~\bfnm{Hugo}\binits{H.}} \AND
\bauthor{\bsnm{Hongler},~\bfnm{Cl{\'e}ment}\binits{C.}}
(\byear{2013}).
\bhowpublished{Crossing probabilities in topological rectangles
for the critical planar FK {I}sing model.
Preprint. Available at \arxivurl{arXiv:1312.7785} [math.PR]}.
\end{bmisc}
%
\bptok{imsref}%
\endbibitem

\bibitem{CDCHKS13}
%
\begin{barticle}[mr]
\bauthor{\bsnm{Chelkak},~\bfnm{Dmitry}\binits{D.}},
\bauthor{\bsnm{Duminil-Copin},~\bfnm{Hugo}\binits{H.}},
\bauthor{\bsnm{Hongler},~\bfnm{Cl{\'e}ment}\binits{C.}},
\bauthor{\bsnm{Kemppainen},~\bfnm{Antti}\binits{A.}} \AND
\bauthor{\bsnm{Smirnov},~\bfnm{Stanislav}\binits{S.}}
(\byear{2014}).
\btitle{Convergence of {I}sing interfaces to {S}chramm's SLE curves}.
\bjournal{C. R. Math. Acad. Sci. Paris}
\bvolume{352}
\bpages{157--161}.
\bid{doi={10.1016/j.crma.2013.12.002}, issn={1631-073X}, mr={3151886}}
\end{barticle}
%
\bptok{imsref}%
\endbibitem

\bibitem{CS11}
%
\begin{barticle}[mr]
\bauthor{\bsnm{Chelkak},~\bfnm{Dmitry}\binits{D.}} \AND
\bauthor{\bsnm{Smirnov},~\bfnm{Stanislav}\binits{S.}}
(\byear{2011}).
\btitle{Discrete complex analysis on isoradial graphs}.
\bjournal{Adv. Math.}
\bvolume{228}
\bpages{1590--1630}.
\bid{doi={10.1016/j.aim.2011.06.025}, issn={0001-8708}, mr={2824564}}
\end{barticle}
%
\bptok{imsref}%
\endbibitem

\bibitem{Duf62}
%
\begin{barticle}[mr]
\bauthor{\bsnm{Duffin},~\bfnm{R.~J.}\binits{R.~J.}}
(\byear{1962}).
\btitle{The extremal length of a network}.
\bjournal{J. Math. Anal. Appl.}
\bvolume{5}
\bpages{200--215}.
\bid{issn={0022-247X}, mr={0143468}}
\end{barticle}
%
\bptok{imsref}%
\endbibitem

\bibitem{DCHN11}
%
\begin{barticle}[mr]
\bauthor{\bsnm{Duminil-Copin},~\bfnm{Hugo}\binits{H.}},
\bauthor{\bsnm{Hongler},~\bfnm{Cl{\'e}ment}\binits{C.}} \AND
\bauthor{\bsnm{Nolin},~\bfnm{Pierre}\binits{P.}}
(\byear{2011}).
\btitle{Connection probabilities and RSW-type bounds for the
two-dimensional FK {I}sing model}.
\bjournal{Comm. Pure Appl. Math.}
\bvolume{64}
\bpages{1165--1198}.
\bid{doi={10.1002/cpa.20370}, issn={0010-3640}, mr={2839298}}
\end{barticle}
%
\bptok{imsref}%
\endbibitem

\bibitem{GM05}
%
\begin{bbook}[mr]
\bauthor{\bsnm{Garnett},~\bfnm{John~B.}\binits{J.~B.}} \AND
\bauthor{\bsnm{Marshall},~\bfnm{Donald~E.}\binits{D.~E.}}
(\byear{2005}).
\btitle{Harmonic Measure}.
\bseries{New Mathematical Monographs}
\bvolume{2}.
\bpublisher{Cambridge Univ. Press},
\blocation{Cambridge}.
\bid{doi={10.1017/CBO9780511546617}, mr={2150803}}
\end{bbook}
%
\bptok{imsref}%
\endbibitem

\bibitem{GR13}
%
\begin{barticle}[mr]
\bauthor{\bsnm{Gill},~\bfnm{James~T.}\binits{J.~T.}} \AND
\bauthor{\bsnm{Rohde},~\bfnm{Steffen}\binits{S.}}
(\byear{2013}).
\btitle{On the {R}iemann surface type of random planar maps}.
\bjournal{Rev. Mat. Iberoam.}
\bvolume{29}
\bpages{1071--1090}.
\bid{doi={10.4171/RMI/749}, issn={0213-2230}, mr={3090146}}
\end{barticle}
%
\bptok{imsref}%
\endbibitem

\bibitem{GN13}
%
\begin{barticle}[mr]
\bauthor{\bsnm{Gurel-Gurevich},~\bfnm{Ori}\binits{O.}} \AND
\bauthor{\bsnm{Nachmias},~\bfnm{Asaf}\binits{A.}}
(\byear{2013}).
\btitle{Recurrence of planar graph limits}.
\bjournal{Ann. of Math. (2)}
\bvolume{177}
\bpages{761--781}.
\bid{doi={10.4007/annals.2013.177.2.10}, issn={0003-486X}, mr={3010812}}
\end{barticle}
%
\bptok{imsref}%
\endbibitem

\bibitem{HS95}
%
\begin{barticle}[mr]
\bauthor{\bsnm{He},~\bfnm{Zheng-Xu}\binits{Z.-X.}} \AND
\bauthor{\bsnm{Schramm},~\bfnm{O.}\binits{O.}}
(\byear{1995}).
\btitle{Hyperbolic and parabolic packings}.
\bjournal{Discrete Comput. Geom.}
\bvolume{14}
\bpages{123--149}.
\bid{doi={10.1007/BF02570699}, issn={0179-5376}, mr={1331923}}
\end{barticle}
%
\bptok{imsref}%
\endbibitem

\bibitem{Kak44}
%
\begin{barticle}[mr]
\bauthor{\bsnm{Kakutani},~\bfnm{Shizuo}\binits{S.}}
(\byear{1944}).
\btitle{Two-dimensional {B}rownian motion and harmonic functions}.
\bjournal{Proc. Imp. Acad. Tokyo}
\bvolume{20}
\bpages{706--714}.
\bid{mr={0014647}}
\end{barticle}
%
\bptok{imsref}%
\endbibitem

\bibitem{KS12}
%
\begin{bmisc}[author]
\bauthor{\bsnm{Kemppainen},~\bfnm{Antti}\binits{A.}} \AND
\bauthor{\bsnm{Smirnov},~\bfnm{Stanislav}\binits{S.}}
(\byear{2012}).
\bhowpublished{Random curves, scaling limits and {L}oewner evolutions.
Preprint. Available at \arxivurl{arXiv:1212.6215} [math-ph]}.
\end{bmisc}
%
\bptok{imsref}%
\endbibitem

\bibitem{Koz07}
%
\begin{barticle}[mr]
\bauthor{\bsnm{Kozma},~\bfnm{Gady}\binits{G.}}
(\byear{2007}).
\btitle{The scaling limit of loop-erased random walk in three dimensions}.
\bjournal{Acta Math.}
\bvolume{199}
\bpages{29--152}.
\bid{doi={10.1007/s11511-007-0018-8}, issn={0001-5962}, mr={2350070}}
\end{barticle}
%
\bptok{imsref}%
\endbibitem

\bibitem{Smi06}
%
\begin{bincollection}[mr]
\bauthor{\bsnm{Smirnov},~\bfnm{Stanislav}\binits{S.}}
(\byear{2006}).
\btitle{Towards conformal invariance of 2{D} lattice models}.
In \bbooktitle{International {C}ongress of {M}athematicians. {V}ol. {II}}
\bpages{1421--1451}.
\bpublisher{Eur. Math. Soc.},
\blocation{Z\"urich}.
\bid{mr={2275653}}
\end{bincollection}
%
\bptok{imsref}%
\endbibitem

\bibitem{Smi10}
%
\begin{binproceedings}[mr]
\bauthor{\bsnm{Smirnov},~\bfnm{Stanislav}\binits{S.}}
(\byear{2010}).
\btitle{Discrete complex analysis and probability}.
In \bbooktitle{Proceedings of the {I}nternational {C}ongress of
{M}athematicians. {V}olume {I}}
\bpages{595--621}.
\bpublisher{Hindustan Book Agency},
\blocation{New Delhi}.
\bid{mr={2827906}}
\end{binproceedings}
%
\bptok{imsref}%
\endbibitem
\end{thebibliography}
\end{document}